\documentclass[12pt, twoside, b5 paper]{report}
\usepackage[T1]{fontenc}
\usepackage{titlesec, blindtext, color}
\usepackage{pdfpages}
\definecolor{gray75}{gray}{0.75}
\newcommand{\hsp}{\hspace{20pt}}
\titleformat{\chapter}[hang]{\Huge\bfseries}{\thechapter\hsp\textcolor{gray75}{}\hsp}{0pt}{\Huge\bfseries}
\usepackage{setspace}
\usepackage[inner=2.5 cm,outer=2.0 cm,top=2.0 cm,bottom=2.0 cm]{geometry}
\usepackage{amsmath,amssymb,amsfonts,amsthm}

\usepackage{geometry}

\usepackage[english]{babel}
\usepackage[latin1]{inputenc}
\usepackage{times}

\usepackage{graphicx}
\usepackage{amsfonts}
\usepackage[all]{xy}

\usepackage{perpage} 
\MakePerPage{footnote}

\numberwithin{equation}{section}

\newcounter{item}[section]
\newcounter{kirshr}
\newcounter{kirsha}
\newcounter{kirshb}

\newtheorem{theorem}{Theorem}[section]

\newtheorem{proposition}[theorem]{Proposition}

\newtheorem{example}[theorem]{Example}
\newtheorem{lemma}[theorem]{Lemma}
\newtheorem{corollary}[theorem]{Corollary}

\theoremstyle{definition}

\newtheorem{definition}[theorem]{Definition}

\title{Modules with Non-Cyclic Socle and the
Extension Property}
\author{By\\Ali Assem Abd-AlQader Mahmoud}

\date{}
\begin{document}
\onehalfspacing
\begin{titlepage}
\begin{center}
\Large \textbf{Modules with Non-Cyclic Socle and the
Extension Property}\\[1.5cm]

Presented by\\ \large \textbf{Ali Assem Abd-AlQader Mahmoud}\\[1cm]

\normalsize A Thesis Submitted \\ to\\ \textbf{Faculty of Science} \\ In Partial Fulfillment of the\\ Requirements for \\ the Degree of \\ Master of Science \\ (Pure Mathematics)\\ Mathematics Department \\ Faculty of Science \\ Cairo University \\ (2015)
\end{center}
\end{titlepage}
\newpage
\thispagestyle{empty}
\mbox{}
.
\newpage
\thispagestyle{empty}
\mbox{}

\begin{center}
\large \textbf{APPROVAL SHEET FOR SUBMISSION}\\[1cm]
\end{center}
\textbf{Thesis Title}: Modules with Non-Cyclic Socle and the
Extension Property.\\
\textbf{Name of candidate}: Ali Assem Abd-AlQader Mahmoud.\\[1cm]

This thesis has been approved for submission by the supervisors:\\[1cm]

Dr. Nefertiti Megahed \\ \mbox{  }\textbf{Signature:}   \\[1cm]
Prof. Tarek Sayed Ahmad\\ \mbox{  }\textbf{Signature:}\\[0.2cm]
\begin{flushright}
Prof.Dr. N.H. Sweilam.\\
Chairman of Mathematics Department\\
Faculty of Science - Cairo University
\end{flushright}

\newpage
\thispagestyle{empty}
\mbox{}
.
\newpage
\thispagestyle{empty}
\mbox{}

\begin{center}
\large \textbf{ABSTRACT}\\[1cm]
\end{center}
\textbf{Student Name}: Ali Assem Abd-AlQader Mahmoud.\\
\textbf{Title of the Thesis}: Modules with Non-Cyclic Socle and the Extension Property.\\
\textbf{Degree}: M.Sc. (Pure Mathematics).\\[0.1cm]

 In 2009, J. Wood \cite{r7} proved that Frobenius bimodules have the extension property for symmetrized weight compositions. More generally, in \cite{r4}, it is shown that having a cyclic socle is sufficient for satisfying the property, while the necessity remained an open question.
\newline Here, landing in Midway, a partial converse is proved. For a significant class of finite module alphabets, the cyclic socle is shown necessary to satisfy the extension property. The idea is bridging to the case of Hamming weight through a new weight function.\\[1cm]
\textbf{Keywords}: linear codes; characters; Frobenius rings; extension property.\\[0.2cm]

\textbf{Supervisors}:\hspace{6cm} \textbf{Signature:} \\\\
Dr. Nefertiti Megahed   \\[1cm]
Prof. Tarek Sayed Ahmad\\[0.2cm]
\begin{flushright}
Prof.Dr. N.H. Sweilam.\\
Chairman of Mathematics Department\\
Faculty of Science - Cairo University
\end{flushright}
\pagenumbering{roman}
\newpage
\thispagestyle{empty}
\mbox{}

\newpage
\thispagestyle{empty}
\mbox{}

Scores
\newpage
\thispagestyle{empty}
\mbox{}

\newpage
\thispagestyle{empty}
\mbox{}

\begin{center}\huge{\textbf{Acknowledgements}}\\\end{center}

In the turn of this rough stage, the current last page had to bear the gratitude and appreciation to many people. Many thanks are to my supervisors Dr. Nefertiti Megahed, and Dr. Tarek Sayed Ahmad. Perhaps they are unaware of this, but each has provided me with a different kind of support. 

Dr. Hany ElHosseiny, for you, a thesis of acknowledgements wouldn't be enough, thanks!

Prof. M. Lamei, without your course in the third year, I wouldn't have been attracted to ring theory. In the same algebraic line, my greatest thanks are to Prof. Ismail Ameen, Dr. Waleed Ahmed, Dr. Laila soueif, Dr. Fatma Ismail, and Dr. Yasser Ibrahim. It has been a big algebraic bulk I have learned from you.

As they were more than demonstrators, they were actually some of the best friends, I'd like to thank  Dr. Youssri Hassan, and Dr. Mohamad Adel for being so supportive from the early beginning. 

Many thanks are to Prof. Ahmad Nasr and Prof. Malak Rizk, I spent some of my best times lecturing calculus to our students, maybe this served as a relief!

Finally, to my best doctors, Dr. Amr Sid-Ahmad and Prof. Alaa E. Hamza, you left me with the most delightful memories through the many semesters.

\newpage
\thispagestyle{empty}
\mbox{}

\tableofcontents
\newpage
\newpage

\pagenumbering{arabic}
\setcounter {chapter}{-1}
\chapter{Introduction}

In 1962, in her doctoral dissertation \cite{Mac}, F. J. MacWilliams proved a certain extension-of-maps' theorem, that was posed when two notions of equivalence of codes were conceived. In that classical context, a code of length $n$ over a finite field alphabet $\mathbb{F}_q$, is defined to be just a subspace of $\mathbb{F}^n_q$. Then, in a sense, two codes are equivalent when they are isomorphic vector spaces, with ``distance'' between codewords being preserved through the indicated isomorphism. In another sense, two codes $C_1,C_2$ are equivalent if the codewords of $C_2$ are obtained, bijectively, from those of $C_1$ by permuting the components and scaling them, using a fixed permutation, and fixed $\mathbb{F}_q$-automorphism for scaling.  MacWilliams extension theorem, also known as the \emph{extension property}  (EP) for linear codes, was re-proved, later, through other proofs \cite{bogart}, \cite{r1996}. However, the character-theoretic proof which H. Ward and J. Wood gave in \cite{r1996} is a one that had the impetus to run in more general settings, and its machinery was operated, over and over, to `forge' many proofs for the EP, in the various contexts. \\

\begin{par}More generally, a (left linear) code of length $n$ over a module alphabet $_RA$ is a (left) submodule $C\subset A^n$. $A$ has the EP with respect to the weight $w$ if, for any $n$ and any two codes $C_1,C_2\subset A^n$, any isomorphism $f: C_1\rightarrow C_2$ preserving $w$ extends to a $G$-monomial transformation of $A^n$, where $G$ is the  \emph{right symmetry group}\footnote{The right symmetry group of a weight $w$ is the group $\{\tau\in\mathrm{Aut}_R(A):w(a\tau)=w(a), \; \text{for all }\; a\in A\}\leq\mathrm{Aut}_R(A).$} of $w$. The first attempt to generalize towards codes over arbitrary finite rings was, perhaps, the treatment of codes over the rings $\mathbb{Z}_m$ accomplished by I. Constantinescu and W. Heise \cite{const1}.

In 1999, following the character-theoretic proof, J. Wood \cite{duality}, proved that Frobenius rings have the EP for Hamming weight. Besides, a partial converse was proved: commutative rings satisfying the EP for Hamming weight are necessarily Frobenius.  It was also proved in \cite{r5} that Frobenius rings have the EP for \emph{symmetrized weight compositions}.\end{par}
\begin{par}In 2004, Greferath et al. \cite{2004} showed that, in general, Frobenius bimodules  have the EP for Hamming weight. In \cite{r2}, Dinh and L\'{o}pez-Permouth suggested a strategy for proving the full converse. The strategy has three parts. (1) If a finite ring is not Frobenius, its socle contains a matrix module of a particular type. (2) Provide a counter-example to the EP in the context of linear codes over this special module. (3) Show that this counter-example over the matrix module pulls back to give a counter-example over the original ring. In 2008, J. Wood \cite{r6} provided a technical lemma (Theorem \ref{wood} below) for carrying out step 2 in  the strategy, and thereby proving that rings having the EP for Hamming weight are necessarily Frobenius. The proof was easily adapted in \cite{r7} (2009) to prove that a module alphabet $_RA$ has the EP for Hamming weight if and only if $A$ is pseudo-injective with cyclic socle.\end{par}

\begin{par}On the other lane, in \cite{r7}, J. Wood proved that Frobenius bimodules have the EP for $\mathrm{swc}_G$, for all $G\preceq \mathrm{Aut}_R(A)$; and in \cite{r4} it was shown that having a cyclic socle is sufficient (Theorem \ref{Noha} below), while the necessity remained an open question. Here, for a certain significant class of module alphabets, we show that having a cyclic socle is necessary.  We define a  weight function that depends on annihilators, and use a monomorphism's preservation of this weight to study the preservation of Hamming weight as well as $\mathrm{swc}_{\mathrm{Aut}_R(A)}$ (Chapter 3). \\\end{par}
\begin{par} Having done this, we find that an alphabet of the form $_KL$, where $K$ is a subfield of another field $L$, has the EP with respect to all symmetrized weight compositions if and only if $K=L$. This handles the case of additive codes. Moreover, the class of modules, on which the necessity is proved, consists of modules that already have qood extension properties (e.g. pseudo-injective modules, semisimple modules, etc.). Were it, in general, unnecessary to have a cyclic socle in order to satisfy the indicated EP, it would have been quite likely that the described class should satisfy the EP under a condition weaker than having a cyclic socle. However, afterall, this wasn't the case. For this reason, it is conjectured that having a cyclic socle is necessary, in general, for satisfying the property. \\\end{par}\\
\setlength{\parindent}{0cm}\textbf{\large{Layout of the thesis}}\\\mbox{}\\
In the most of it, this thesis tends to maintain a certain fashion by displaying some results of pure ring theory, and then, without delay, these make their own performance in our treatment of the extension problem. Thus, it was somehow inevitable that the text should be preluded with a big bulk of pure ring theory.\\
Indeed, in chapter 1, we study semisimplicity, the Wedderburn-Artin theorem, and the Jacobson radical, this is section 1.1. In section 1.2, we study local, semilocal, and semiperfect rings, and what these classes of rings imply for idempotents. Section 1.3 is devoted for some elements of character theory  that will be crucial in later chapters.\\
\setlength{\parindent}{0.5cm}
\begin{par}In chapter 2, we directly apply section 1.3 to prove MacWilliams extension theorem for codes over finite fields, but, of course, this is done in section 2.2, after defining  the basic terms of coding theory in section 2.1. Anyway, section 2.1 displays an inadequate history of coding theory, and, for this reason, a history compensation is done by displaying, in many occasions, a fair amount of historical remarks is divided throughout the text. In section 2.3, we return, again, to some pure ring theory, studying Frobenius an d quasi-Frobenius rings. One crucial result about Frobenius rings is that a finite ring is Frobenius if and only if its character module is cyclic. \\
In \cite{duality}, J. Wood proved this by appealing to Morita duality. Were we to adopt this approach, the thesis would either witness a large degression, or, in the best case we would have stated, without proofs, many powerful theorems. Nevertheless, the thesis managed in avoiding the dilemma, using a theorem of T. Honold (Theorem \ref{hon}). Honold \cite{Honold} proved that, for a finite ring $R$, $_R\widehat{R}$ is cyclic if and only if $_R(R/\mathrm{rad}R)\cong \mathrm{soc}(_RR)$. It remained for us to prove, using only the original definition of QF rings, and guided by \cite{Lam},  that a finite ring is Frobenius if and only if $_R(R/\mathrm{rad}R)\cong \mathrm{soc}(_RR)$. This heavy part about Frobenius rings isn't included in chapter 1 so as to end chapter 1 in a good timing, allowing a reader longing to see the applications in coding theory to refresh himself. Another reason is that, it seemed better to place results from papers of coding theory in their context of coding theory even if they are  pure ring theory results.\end{par}

\begin{par}In chapter 3, we investigate the extension problem, this time with respect to symmetrized weight compositions. In section 3.1 and the first two sub-sections of 3.2, the exposition mainly follows \cite{r5}, \cite{r7}, \cite{2004} and \cite{r4}. \\
Section 3.2.3  contains the main result of the thesis, where we define annihilator weight ($aw$) and use it to prove that, for a considerable class of modules,  a cycle socle is necessary for satisfying the extension property with respect to symmetrized weight compositions.  \end{par}

A talk on this manipulation was delivered by the author in the "Congress on Non-commutative Rings and their Applications" held in Lens, France, in June 2015. One paper has been extracted out of this work \cite{JAA}, and was accepted for publication in the Journal of Algebra and Its Applications (published by World Scientific).

\chapter{Algebraic Context and Tools}

\begin{par}At the end of this chapter, a certain algebraic theme, that prevails throughout the thesis,  would have been brought into action. As a pleasant way to go, this chapter (the thesis, thereby) shall grow from a common ground, the Wedderburn-Artin theory seemed to serve this purpose, this is section 1. In Section 2, we call on semilocal, semiperfect rings, and the theory of idempotents as an outgrowth. Section 3 triggers some character theory  and investigates interactions with  preceding sections.  The definitions and basic properties of Frobenius rings, however, will  be covered in the next chapter. This is justified with that the intended exposition of Frobenius rings in this thesis will mainly focus on the relation of such a ring with its character bimodule, the ideas involved are fairly new ones. Besides, it was preferred that these ideas be only a few pages apart from their applications in coding theory. \end{par}\\[0.1cm]

\textbf{Note:} All rings are with unity, and modules are unitary, unless otherwise stated. Proofs corresponding to famous results will  often be omitted.

\section{Semisimple Rings}
\begin{par}Traced back to 1908, starting with a theorem of J. H. Wedderburn in his  ``On hypercomplex numbers'' (Proc. London Math. Soc. V.6, N.2), a sequence of refined results eventually formed the major repertoirs in the structure theory of rings.
Working on finite dimensional algebras (they were known as \emph{systems of hypercomplex numbers} at that time), Wedderburn defined the radical of  an algebra $A$ to be the largest nilpotent ideal in $A$, and call $A$ semisimple if it has a zero radical.\end{par}
\begin{par}The next refinement was about twenty years later, when E. Artin generalized Wedderburn's work to the class of rings satisfying both chain conditions on left ideals (Levitzki and Hopkins, later in 1939, showed that the \emph{DCC} on left ideals  implies the \emph{ACC}  on left ideals, Theorem \ref{4.15lam}). Rings satisfying the \emph{DCC} (now called \emph{Artinian}) have a largest nilpotent ideal -- as Artin showed -- and whence semisimple rings were naturally defined.\end{par}
\begin{par} In our exposition, however, we shall start out in a rather different approach, away from radicals, mainly following \cite{Lamnoncom}.\end{par}

\subsection{Semisimple Modules}

\begin{definition} Let $R$ be a ring, and $M$ a (left) R-module.
\begin{enumerate}
\item[$(1)$]$M$ is called a \emph{simple} (or \emph{irreducible}) module if $M\neq0$, and $M$ has no $R$-submodules but 0 and $M$.
\item[$(2)$] $M$ is called a \emph{semisimple} (or \emph{completely reducible}) module if every $R$-submodule of $M$ is direct summand of  $M$. \end{enumerate}\end{definition}

\begin{lemma}\label{remark1}Any submodule (resp., quotient module) of a semisimple $R$-module is semisimple.\end{lemma}\begin{proof}Clear.\end{proof}
\begin{lemma}Any nonzero semisimple module contains a simple submodule.\end{lemma}\begin{proof} Let $m \in M$ be a nonzero element. By Zorn's Lemma, the family of submodules not containing $m$ has a maximal element $N$. Let $N'\neq0$ be such that $M=N\oplus N'$. We claim that $N'$ is simple. Indeed, if $K$ is a nonzero submodule of $N'$, then $N\oplus K$ must contain $m$ by the maximality of $N$, hence $N\oplus K=M$, which implies $K=N'$.\end{proof}
\begin{theorem}\label{ss} The following are equivalent for any module $_RM$:
\begin{enumerate}
\item[$(1)$]$M$ is semisimple.
\item[$(2)$] $M$ is the direct sum of a family of simple submodules.
\item[$(3)$] $M$ is the sum of a family of simple submodules.
 \end{enumerate}
\end{theorem}
\begin{proof}\begin{par}$(1)\Longrightarrow(3).$ Let $M$ be semisimple, and $M_1$ be the sum of all simple submodules in $M$, and write $M=M_1\oplus M_2$, where $M_2$ is a suitable submodule. If $M_2\neq 0$, the previous lemma implies that $M_2$ contains a simple submodule, which is a contradiction. \end{par}
\begin{par} $(3)\Longrightarrow(1).$ Write $M=\sum_{i\in I}M_i$, where $M_i$'s are  simple submodules of $M$. Let $N\subseteq M$ be a submodule. Consider the subsets $J\subseteq I$ with the property that $N+\sum_{j\in J}M_j$ is a direct sum. Applying Zorn's Lemma, there is a maximal such subset $J\subseteq I$, and we show that $M=N\oplus\bigoplus_{j\in J}M_j$. If for some $i\in I$,
$M_i\nsubseteq N\oplus\bigoplus_{j\in J}M_j$,  the simplicity of $M_i$ implies that $M_i\cap N\oplus\bigoplus_{j\in J}M_j=0$, which makes $N+M_i+\bigoplus_{j\in J}M_j$ a direct sum, contradicting the maximality of $J$. This shows that $M=\bigoplus_{j\in J}N_j$.\end{par}

\begin{par}$(3)\Longrightarrow(2)$  follows by applying the above arguments for $N=0$.\end{par}

\end{proof}

\begin{theorem}\label{semisimple rings}(Semisimple Rings)
For a ring $R$, the following are equivalent:
    \begin{enumerate}
        \item[$(1)$] All short exact sequences of left $R$-modules split.
        \item[$(2)$] All left $R$-modules are semisimple.
        \item[$(3)$] All finitely generated left $R$-modules are semisimple.
        \item[$(4)$] All cyclic left $R$-modules are semisimple.
        \item[$(5)$] The left regular\footnote{The ring $R$, considered as a left (right) module over itself is called the left (right) regular $R$-module.} $R$-module $_RR$ is semisimple.
    \end{enumerate}
If any of these conditions holds, $R$ is said to be a left semisimple ring.
\end{theorem}
\begin{par}\textbf{Note:} We shall see later that a ring is left  semisimple if and only if it is right semisimple.\end{par}
\begin{proof}
(1) and (2) are clearly equivalent, and of course $$(2)\Longrightarrow(3)\Longrightarrow(4)\Longrightarrow(5).$$ Now we prove that $(5)\Longrightarrow(2)$. Let $M$ be any left $R$-module where $R$ satisfies (5). In view of Lemma \ref{remark1}, any cyclic submodule $Rm$ of $M$ is semisimple (being isomorphic to a quotient of $R$). Since $M=\sum_{m\in M}Rm$, by (3) in Theorem \ref{ss}, $M$ is semisimple.
\end{proof}
\begin{corollary}A left semisimple ring $R$ is both left noetherian and left artinian.\end{corollary}
\begin{proof}By the previous theorem, express $_RR$ as a direct sum of simple modules, and obtain a finite expression for 1, hence $R$ is actually having a finite decomposition into simple left ideals. The result then follows.\end{proof}

\begin{definition}(Projective Modules). A module $P$ over a ring $R$ (not necessarily finite) is said to be \emph{projective} if given any diagram of $R$-module homomorphisms
\begin{displaymath}
\xymatrix{ &P \ar[d]^{f}&\\A\ar[r]^g  & B\ar[r]&0}
\end{displaymath}
with bottom row exact (that is, $g$ an epimorphism), there exists an $R$-module homomorphism $h:P\rightarrow A$ such that the diagram
\begin{displaymath}
\xymatrix{ &P \ar[d]^{f}\ar[dl]^h&\\A\ar[r]^g  & B\ar[r]&0}
\end{displaymath} is commutative (that is, $gh=f$).\end{definition}

\begin{definition}(Injective Modules). A module $J$ over a ring $R$ (not necessarily finite) is said to be \emph{injective} if given any diagram of $R$-module homomorphisms
\begin{displaymath}
\xymatrix{ 0\ar[r]&A \ar[r]^g\ar[d]_{f} &B\\&J&}
\end{displaymath}
with top row exact (that is, $g$ a monomorphism), there exists an $R$-module homomorphism $h:B\rightarrow J$ such that the diagram
\begin{displaymath}
\xymatrix{ 0\ar[r]&A \ar[r]^g\ar[d]_{f} &B\ar[dl]^h\\&J&}
\end{displaymath} is commutative (that is, $hg=f$).\end{definition}

The following proposition -- given without proof -- is a familiar result about projective and injective modules. (\cite{hungr}, Propositions IV.3.4 and IV.3.13)
\begin{proposition}\label{projinj} Let $P,I$ be (left) $R$-modules.  The following are true.
 \begin{enumerate}
        \item[$(1)$] $P$ is projective if and only if every short exact sequence\\ $0\rightarrow A\overset{\mathrm{f}}{\rightarrow} B\overset{\mathrm{g}}{\rightarrow} P\rightarrow0$ splits, and $B\cong A\oplus P$.
        \item[$(2)$] $I$ is injective if and only if  every short exact sequence\\ $0\rightarrow I\overset{\mathrm{f}}{\rightarrow} C\overset{\mathrm{g}}{\rightarrow} D\rightarrow0$ splits, and $C\cong I\oplus D$.\end{enumerate}\end{proposition}

In the next is another characterization of (left) semisimple rings.
\begin{theorem}\label{lastlast}For a ring $R$, the following are equivalent:
    \begin{enumerate}
        \item[$(1)$] $R$ is left semisimple.
        \item[$(2)$] All left $R$-modules are projective.
        \item[$(3)$] All left $R$-modules are injective.

    \end{enumerate}\end{theorem}

\begin{proof}Follows from (1) in Theorem \ref{semisimple rings} and the previous proposition.\end{proof}

\subsection{The Wedderburn-Artin Theorem}\label{wed}
Here we shall display, without proof, three of the most famous and essential theorems. These, in order, are: the Jordan-H\"{o}lder theorem, the Wedderburn-Artin theorem, and Wedderburn's little theorem. Proofs of these theorems appear in many textbooks, see for example \cite{cat}, \cite{radicalappr}, \cite{Lam}, and \cite{Lamnoncom}.
\begin{definition}A finite chain $0=M_0\subset M_1\subset\ldots\subset M_n=M$ of submodules of a module $M$ is called a \emph{composition series}  for the module $M$ if all \emph{factors} $M_{i+1}/M_i$ are simple. The number $n$ is called the \emph{length} of the series, the zero module is considered to have a composition series of length zero and no composition factors.  \end{definition}

\begin{par}If a module $M$ has a composition series $0=M_0\subset M_1\subset\ldots\subset M_n=M$, then for a  submodule $K\subset M$ the chains $0=M_0\cap K\subset M_1\cap K\subset\ldots\subset M_n\cap K=K$, and $0=(M_0+K)/K\subset (M_1+K)/K\subset\ldots\subset (M_n+K)/K=M/K$ are composition series for $K$ and $M/K$ respectively, one can check easily that the factors are simple. We shall refer these series as the \emph{relative series} of $K$ and $M/K$.\end{par}

\begin{theorem}\label{jordan} \textbf{(Jordan-H\"{o}lder Theorem)}. If a module $M$ is having a composition series, then any finite chain of submodules can be included in a composition series. Besides, any two composition series have the same length and a bijection can be established between the two series such that the corresponding factors are isomorphic.  \end{theorem}

\begin{theorem}\textbf{(Wedderburn-Artin Theorem)}. Let $R$ be any left semisimple\footnote{In many texts, this theorem is stated for J-semisimple artinian rings. This is actually equivalent to being semisimple (as shown in the next section).} ring. Then $R\overset{\varphi}{\cong} \mathbb{M}_{n_1}(D_1)\times\cdots\times\mathbb{M}_{n_k}(D_k)$, where $\mathbb{M}_{n_i}(D_i)$ is the  ring of all square matrices of size $n_i$ over the division ring $D_i$, and $\varphi$ is an isomorphism of rings. The number $k$ is unique, and there are exactly $k$ simple $R$-modules up to isomorphism. \end{theorem}

Since $\mathbb{M}_{n_1}(D_1)\times\cdots\times\mathbb{M}_{n_k}(D_k)$ is right semisimple as well as left semisimple, we have the following corollary.
\begin{corollary}\label{smsmpl} A ring is left semisimple if and only if it is right semisimple.\end{corollary}

\textbf{Remark: }
The matrix module $T_i=\mathbb{M}_{n_i\times 1 }(D_i)$ can be endowed with the following  left $R$-module structure:\begin{equation}\label{action}r \begin{bmatrix}a_1\\\vdots\\a_{n_i}\end{bmatrix}:= \pi_i(\varphi(r+ \mathrm{rad}R))\begin{bmatrix}a_1\\\vdots\\a_{n_i}\end{bmatrix},\end{equation} where $\pi_i:\mathbb{M}_{n_1}(D_1)\times\cdots\times\mathbb{M}_{n_k}(D_k)\rightarrow \mathbb{M}_{n_i}(D_i)$ is the $i^{\text{th}}$ projection, and $\varphi$ is the indicated isomorphism of rings in the theorem. $_RT_i$ is called the pullback to $R$ of the matrix module $_{\mathbb{M}_{n_i}(D_i)}\mathbb{M}_{n_i\times 1}(D_i)$.
The $T_i$'s form a complete list, up to isomorphism, of all simple left $R$-modules.

A similar argument provides the matrix modules $T'_i=\mathbb{M}_{1\times n_i }(D_i)$ with a right $R$-module structure. The $T'_i$'s form a complete list of simple right $R$-modules.

\begin{theorem}\textbf{(Wedderburn's Little Theorem)}.  Every finite division ring is a f{}ield.\end{theorem}

Proofs of the previous theorems may be found in \cite{cat}, \cite{Lamnoncom} and \cite{Lam}, ordered with respect to the theorems.

\subsection{The Jacobson Radical}

\begin{par} The Jacobson radical of a ring is defined  to be the intersection of all maximal left ideals. As probably noticed, it's name indicates nothing about a one-sided nature, since, it turns out that the radical, so defined, coincides with the intersection of all maximal right ideals.\end{par}

\begin{lemma}\label{4.1lam}(\cite{Lamnoncom}, (4.1)) The following are equivalent:
\begin{enumerate}
        \item[$(1)$] $y\in \mathrm{rad}R$;
        \item[$(2)$] $1-xy$ is left invertible for any $x\in R$;
        \item[$(3)$] $yM=0$ for any simple left $R$-module $M$.
    \end{enumerate}\end{lemma}

Noticing that the annihilator of a (left) module is a two-sided ideal, the following corollary follows at once.
\begin{corollary}\label{4.2lam}$\mathrm{rad}R=\underset{M\text{simple}}{\bigcap} \mathrm{ann} (M)$, and $\mathrm{rad}R$ is then an ideal. \end{corollary}

The next is a refinement of (1) in the previous lemma.
\begin{lemma}\label{4.3lam} $y\in \mathrm{rad}R$ if and only if $1-xyz$ is a unit for any $x,z\in R$.\end{lemma}
\begin{proof}If $y\in \mathrm{rad} R$, then, by the corollary, $yz\in \mathrm{rad}R$. Thus, by (2) in Lemma \ref{4.1lam}, there exists $u\in R$ such that $u (1-xyz)=1$. Once again, $xyz\in  \mathrm{rad}R$, and hence left invertible is the element $1-(-u)(xyz)=1+u(xyz)=u$, and thus $u\in \mathcal{U}$, the group of units of $R$. This completes the proof, the converse is immediate by Lemma \ref{4.1lam}. \end{proof}

\begin{proposition}$\mathrm{rad} (R/\mathfrak{A})=(\mathrm{rad} R)/\mathfrak{A}$ for any ideal $\mathfrak{A}$ contained in
$\mathrm{rad} R$.\end{proposition}
\begin{proof}Clear.\end{proof}
\begin{definition} A ring $R$ is called \emph{Jacobson semisimple} (J-\emph{semisimple}, or \emph{semiprimitive}) if $\mathrm{rad} R=0$. \end{definition}

\begin{par}The next proposition indicates some of the preserved features when passing from $R$ onto $R/\mathrm{rad}R$. \end{par}

\begin{proposition}\label{rad andsim}$R$ and $R/\mathrm{rad}R$ have the same simple left modules. An element $r\in R$ is left invertible (invertible) if and only if $\bar{r}$ is left invertible (invertible) in $R/\mathrm{rad} R$.\end{proposition}
\begin{proof}The first claim follows from \ref{4.2lam}. For the second claim, assume that $\bar{s}\bar{r}=\bar{1}$, then $$sr\in 1+\mathrm{rad} R\subset \mathcal{U}. $$ This gives that $r$ is left invertible in $R$.\end{proof}

\begin{lemma}\label{4.11lam}$\mathrm{rad}R$ contains every nil  left (right) ideal in $R$. \end{lemma}
\begin{proof} \begin{par}Let $\mathfrak{A}$ be a nil left ideal, and let $y\in\mathfrak{A}$. It follows that for any $x\in R$ there exists $n\in \mathbb{N}$ such that $(xy)^n=0$, and then $1-xy$ has the inverse  $\overset{n}{\underset{i=0}{\sum}}(xy)^i$, so $y\in\mathrm{rad}R $. The proof clearly works for right ideals.\end{par}\end{proof}

The following theorem, implicitly says that the Jacobson radical generalizes the Wedderburn radical and the two radicals coincide for left artinian rings.
\begin{theorem}\label{artradnilp} Let $R$ be a left artinian ring. Then $\mathrm{rad}R$ is the largest nilpotent left (right) ideal.\end{theorem}

\begin{proof}\begin{par} By the previous lemma we're done if $J=\mathrm{rad}R$ itself is shown nilpotent. Since $R$ is left artinian, the chain $$J\supset J^2\supset J^3\supset\cdots$$eventually stabilizes. There exists $N$ such that $J^N=J^{N+1}=\cdots$, and we show that $I=J^N$ is the zero ideal. \end{par}
\begin{par} If $I\neq0$ then the DCC for left ideals provides a left ideal $\mathfrak{A}_0$ that is minimal in the set of left ideals $\mathfrak{A}$ satisfying $I\mathfrak{A}\neq0$. Let $a\in \mathfrak{A}_0$ be an element that satisfy $Ia\neq 0$, then $$I(Ia)=I^2a=Ia\neq 0,$$ and the minimality of $\mathfrak{A}_0$ yields $Ia=\mathfrak{A}_0$. One gets that, for some $y\in I$, $a=ya$. A contradiction follows now from the fact that $(1-y)$ is  a unit (which implies $a=0$), and thus $J$ must be nilpotent as $I$ couldn't be nonzero.  \end{par}\end{proof}

\begin{theorem}\label{4.14lam} The following are equivalent for any ring $R$:
\begin{enumerate}
        \item[$(i)$] $R$ is semisimple;
        \item[$(ii)$] $R$ is J-semisimple and left artinian;
        \item[$(iii)$] $R$ is J-semisimple, and satisfies DCC on principal left ideals.
    \end{enumerate}
\end{theorem}

\begin{proof}\begin{par} $(i)\Longrightarrow(ii)$. Setting $J=\mathrm{rad}R$, the assumed semisimplicity of $R$ implies that $R=J\oplus \mathfrak{B}$ for some left ideal $\mathfrak{B}$. From the corresponding decomposition of 1, there exist idempotents $e,f$ such that $J=Re, \mathfrak{B}=Rf$, $e+f=1$. $e\in J=\mathrm{rad}R$ implies that $f=1-e$ is a unit. Being an idempotent, it follows that $f=1$, hence $J=0$.\end{par}
\begin{par}$(ii)\Longrightarrow(iii)$. Trivial.  \end{par}
\begin{par}$(iii)\Longrightarrow(i)$. Before proceeding, just notice that:\end{par}
 \begin{enumerate} \item[(1)] Every left ideal $I$ in $R$ contains a minimal left ideal, simply by scanning the family of principal ideals in $I$ and using the condition in (iii) to capture at a minimal principal ideal which is certainly a minimal left ideal.
 \item[(2)] Every minimal left ideal is a direct summand of $_RR$, for if $I$ is minimal, there exists a maximal left ideal $\mathfrak{m}$ not containing $I$ (since $\mathrm{rad}R=0$). Then $I\cap\mathfrak{m}=0$, and $_RR=I\oplus \mathfrak{m}$. \end{enumerate}
Assume $R$ isn't semisimple. Let $\mathfrak{B}_1$ be a minimal left ideal. Then $_RR=\mathfrak{B}_1\oplus I_1$, with $I_1\neq0$. By (1)  above, $I_1$ contains a minimal left ideal $\mathfrak{B}_2$,  then $\mathfrak{B}_2$ is a direct summand of $_RR$ and  thus of course a direct summand of $I_1$, say $I_1=\mathfrak{B}_2\oplus I_2$, $I_2\neq0$. Proceeding this way, we get a chain $$I_1\supsetneq I_2\supsetneq I_3\supsetneq\cdots,$$ of left ideals that are direct summands of $_RR$. A decomposition of 1 is enough to see that these left ideals are, in fact, principal left ideals, which contradicts (iii).

     \end{proof}

\begin{proposition}\label{socrad} For any left $R$-module $_RM$,\begin{center} $\mathrm{soc}(M)\subset\{m\in M: (\mathrm{rad}R)\cdot m=0\}$,\end{center} with equality holding if $R/\mathrm{rad}R$ is an artinian\footnote{Recall that a ring is called artinian if it is both left and right artinian.} ring.\end{proposition}

\begin{proof}The first inclusion is clear knowing that $\mathrm{rad}R$ annihilates all simple left $R$-modules. Now, assume that $R/\mathrm{rad}R$ is artinian, and that $m\in M$ has the property that $ (\mathrm{rad}R)\cdot m=0$. Then the module $Rm$ can be considered as an  $(R/\mathrm{rad}R)$-module. But, by Theorem \ref{4.14lam}, $R/\mathrm{rad}R$ is a semisimple ring, hence $Rm$ is a semisimple $(R/\mathrm{rad}R)$-module. Consequently, $Rm$ is a semisimple $R$-module since $R$ and $R/\mathrm{rad}R$ share the same simple modules (Proposition \ref{rad andsim}). Thus, $Rm\subset \mathrm{soc}(M)\subset M$, and $m\in \mathrm{soc}(M)$. \end{proof}

\begin{theorem}\label{4.15lam}\textbf{(Hopkins-Levitzki, 1939)}. Let $R$ be a semiprimary\footnote{A ring $R$ is said to be \emph{semiprimary} if $\mathrm{rad}R$ is nilpotent and $R/\mathrm{rad}R$ is semisimple.} ring. Then for any $R$-module $_RM$, the following are equivalent:
\begin{enumerate}
        \item[$(i)$] $M$ is noetherian;
        \item[$(ii)$] $M$ is artinian;
        \item[$(iii)$] $M$ has a composition series.
    \end{enumerate}

\end{theorem}
\begin{proof} \begin{par}We'll begin with the simple implications $(iii)\Longrightarrow(i)$ and (ii). Suppose that $M$ has a composition series $S$ of length  $n$, if it happened that either chain condition fails, one can obtain a normal series $T$ of length $n+1$ $$M=M_0\supsetneq M_1\supsetneq M_2\supsetneq \cdots\supsetneq M_{n+1}. $$ By the Jordan-H\"{o}lder theorem (\ref{jordan}), $T$ is supposed to have a refinement that is equivalent to $S$, which is impossible.\end{par}

\begin{par}Now we prove that either chain condition on $M$ implies (iii). As usual, for the sake of a simple reference, set $J=\mathrm{rad}R$, and let $m$ be such that $J^m=0$. Through the chain $$M\supset JM\supset J^2M\supset J^mM=0,$$ we can provide a composition series for $M$ by showing that each of the factors $J^iM/J^{i+1}M$ has a composition series. $J^iM/J^{i+1}M$ is either noetherian or artinian (depends on our assumption for $M$) as left module over  $R/J$. Since $R/J$  is semisimple, by Theorem \ref{semisimple rings}, $J^iM/J^{i+1}M$ is a direct sum of simple $R/J$-modules. Either chain condition on $J^iM/J^{i+1}M$ makes this a finite sum, giving rise to the sought composition series.\end{par}
\end{proof}

\begin{theorem}\label{nakayama}\textbf{(Nakayama's Lemma)}. For any left ideal $I\subset R$, the following are equivalent:
\begin{enumerate}
        \item[$(1)$] $I\subset \mathrm{rad}R.$
        \item[$(2)$] For any finitely generated left $R$-module $M$, if $I\cdot M=M$ then $M=0$.
        \item[$(3)$] For any left $R$-modules $N\subset M$ such that $M/N$ is finitely generated, if $N+I\cdot M=M$ then  $N=M$.
    \end{enumerate}

\end{theorem}

\begin{proposition}\label{19.27} Let $I\subset\mathrm{rad}R$ be an ideal, and set $\overline{R}=R/I$. Let $P,Q$ be finitely generated projective left $R$-modules. Then $P\cong Q$ as $R$-modules if and only if $P/IP\cong Q/IQ$ as $\overline{R}$-modules.
\end{proposition}
\begin{proof}The ``only if'' part is direct. For the ``if'' part, let $\overline{f}:P/IP\rightarrow Q/IQ$ be an isomorphism, and consider the following diagram
\begin{displaymath}
\xymatrix{ P \ar@{.>}[d]_{f}\ar[r]  & P/IP \ar[d]^{\overline{f}} \\
             Q\ar[r]  & Q/IQ  }
\end{displaymath}
By the projectivity of $P$, since $\overline{f}$ is surjective, there is an $R$-homomorphism $f:P\rightarrow Q$ (dotted in the diagram) such that the diagram is commutative. Again, since $\overline{f}$ is surjective, $\mathrm{Im}f+IQ=Q$. Since $Q$ is finitely generated, Nakayama's Lemma implies that $\mathrm{Im}f=Q$ and $f$ is surjective. Now, the projectivity of $Q$ implies, by Proposition \ref{projinj}, that the short exact sequence $0\rightarrow \mathrm{Ker}f\overset{\iota}{\rightarrow} P \overset{f}{\rightarrow} Q\rightarrow0$ splits, and $P\cong \mathrm{Ker}f\oplus Q$. Set $M=\mathrm{Ker}f$, and calculate that,
$$P/IP\cong M/IM \oplus Q/IQ.$$
Now, since $\overline{f}$ is an isomorphism, $ M/IM=0$ and $M=IM$. Also $M$ is finitely generated, being a submodule of $P$. Then, in the use of Nakayama's Lemma, we get $ \mathrm{Ker}f=M=0$.\end{proof}

\begin{definition}The Jacobson radical of a ring was defined to be the intersection of all maximal left ideals. Extending this notion, define $\mathrm{rad} M$, the radical of an $R$-module $_RM$, to be the intersection of all maximal  submodules.\end{definition} The next  lemma is  to be used later.
\begin{proposition}\label{radmod} The following are true.

\begin{enumerate}

\item[$(1)$]$\mathrm{rad} M=\underset{\varphi,U}{\bigcap}\{\mathrm{Ker}\varphi\;| \;\varphi:{_RM}\rightarrow {_RU}\;\;\text{is a homomorphism}, {_RU} \;\;\text{simple} \}$.
\item [$(2)$]If $\{M_\alpha\}$ is a family of $R$-modules, then $\mathrm{rad} (\bigoplus M_\alpha)=\bigoplus \mathrm{rad} M_\alpha$.
\item[$(3)$]If $f:M\rightarrow N$ is an $R$-module isomorphism, then $\mathrm{rad} N=f(\mathrm{rad} M)$.
\item[$(4)$]For any left $R$-module $M$, $(\mathrm{rad}R) M\subset\mathrm{rad} M$.
\item[$(5)$]$\mathrm{rad} F=(\mathrm{rad} R) F$ for any free $R$-module $F$.
\item[$(6)$]$\mathrm{rad} P=(\mathrm{rad} R) P$ for any nonzero projective $R$-module $P$.

\end{enumerate}\end{proposition}

\begin{proof}\begin{enumerate}
\item[$(1)$] If $U$ is a simple left $R$-module, and $\varphi:M\rightarrow U$ is a nonzero homomorphism, then $M/\mathrm{Ker}\varphi\cong U$ is simple and hence $\mathrm{Ker}\varphi$ is a maximal submodule. Conversely, if $M_1$ is a maximal submodule of $M$, then the natural epimorphism $M\rightarrow M/M_1$ has its kernel equal to $M_1$.
\item[$(2)$] Let $\pi_\alpha:M\rightarrow M_\alpha$ be the natural projection, and $\iota_\alpha: M_\alpha\rightarrow M$ be the natural inclusion. If $m\in \mathrm{rad}M$, and $\pi_\alpha(m)=m_\alpha$, we show that $m_\alpha\in\mathrm{rad}M_\alpha$. Let $\psi:M_\alpha\rightarrow U$ be a homomorphism into a simple module $U$. Then $\psi\pi_\alpha: M\rightarrow U$ satisfies $\psi(m_\alpha)=\psi\pi_\alpha(m)=0$.  Thus, $m_\alpha\in\mathrm{rad}M_\alpha $. Conversely, let $m\in\bigoplus \mathrm{rad} M_\alpha$, and let $\varphi:M\rightarrow U$ be a homomorphism. Notice that $\varphi(m)=\underset{\alpha}{\sum} \varphi\iota_\alpha \pi_\alpha(m) $, and that  $\varphi_\alpha=\varphi\iota_\alpha$ is a homomorphism from $M_\alpha$ into $U$, hence $\varphi(m)=\underset{\alpha}{\sum} \varphi_\alpha(m_\alpha)=0$ since $m_\alpha\in\mathrm{rad}M_\alpha$ for each $\alpha$.

\item[$(3)$] If $m\in\mathrm{rad}M$, and $\varphi:N\rightarrow U$ is a homomorphism into a simple $R$-module $U$, then $\varphi f$ is a homomorphism from $M$ into $U$, and hence $\varphi(f(m))=0$. This shows that $f(m)\in \mathrm{rad}N$. Also, if $f(m)\in\mathrm{rad}N$, and $\psi:M\rightarrow U$ is any homomorphism into a simple $R$-module $U$, then $\psi f^{-1}:N\rightarrow M$ must satisfy $0=\psi f^{-1}(f(m))=\psi(m)$.
\item[$(4)$] First, notice that the first part in  the proof of (3) above shows, in particular, that if $f:M\rightarrow N$ is a homomorphism, then $f(\mathrm{rad}M)\subset \mathrm{rad}N$. Now, for statement (4), if $m\in M$, define $f:R\rightarrow M$ by $r\mapsto rm$. Then $(\mathrm{rad}R)m=f(\mathrm{rad}R)\subset\mathrm{rad}M$, giving the desired result.

\item[$(5)$] Suppose that $X$ is a basis of $F$, then we know that $F=\underset{x\in X}{\oplus}Rx$, and that, for each $x\in X$, the map $r\mapsto rx$ is an $R$-isomorphism between $R$ and $ Rx$. The result now follows from (2), (3) and (4).

\item[$(6)$] Being projective, $P$ is a direct summand of a free $R$-module $F$, say $F=P\oplus Q$, for some $Q$. Then, by (2), $(\mathrm{rad}R)F=\mathrm{rad}F=\mathrm{rad}P\oplus\mathrm{rad}Q$. Besides, we have $(\mathrm{rad}R)F\subset(\mathrm{rad}R)P\oplus(\mathrm{rad}R)Q$, and, by (4), $(\mathrm{rad}R)P\oplus(\mathrm{rad}R)Q\subset(\mathrm{rad}R)F$, hence $(\mathrm{rad}R)F=(\mathrm{rad}R)P\oplus(\mathrm{rad}R)Q$. Now, It follows  that $\mathrm{rad}P=(\mathrm{rad}R)P$  (since $(\mathrm{rad}R)P\subset P$ and $(\mathrm{rad}R)Q\subset Q$).
\end{enumerate}

\end{proof}

\section{Theory of Idempotents}
\subsection{Local, Semilocal, and Semiperfect Rings}
\begin{par}The notion of a local ring originally appears in commutative algebra,  referring to those nonzero commutative rings with a unique maximal ideal. In the noncommutative case, however, the notion was naturally generalized: a nonzero ring $R$ is \emph{local} if $R$ has a unique maximal left ideal. One may wonder now, why such ring is not called \emph{left} local instead? The answer is that having a unique maximal left ideal turns out to be equivalent to having a unique maximal right ideal; moreover these two one-sided ideals coincide, and that unique `ideal' is $\mathrm{rad}R$ (\cite{Lamnoncom}, Theorem 19.1). Thus, being local is equivalent to that $R/\mathrm{rad}R$ is a division ring. In view of the Wedderburn theorem, a probably good generalization would be the class of  rings $R$ with the property that $R/\mathrm{rad}R$ is semisimple. Indeed, this becomes more convincing when we find that, in the commutative case, those are exactly the rings with finitely many maximal ideals (Proposition \ref{prop20.2lam} below). This is the class of \emph{semilocal} rings.
\end{par}

\begin{definition}(Semilocal Rings). A ring $R$ is said to be \emph{semilocal} if $R/\mathrm{rad}R$ is a semisimple ring, or, equivalently\footnote{Theorem \ref{4.14lam}.}, if $R/\mathrm{rad}R$ is left artinian.\end{definition}

\begin{proposition}\label{prop20.2lam}If a ring $R$ has finitely many maximal left ideals, then $R$ is semilocal. The converse holds if $R/\mathrm{rad}R$ is commutative. \end{proposition}

\begin{proof}Upon reflection, $R$ is semilocal if and only if $R/\mathrm{rad}R$ is semilocal, also, $R$ has finitely many maximal left ideals if and only if $R/\mathrm{rad}R$ has finitely many maximal left ideals. Thus, we may assume that $\mathrm{rad}R=0$ in proving the two claims in the proposition.\\ Let $\mathfrak{m}_1,\ldots,\mathfrak{m}_n$ be the maximal left ideals of $R$. The map $R\rightarrow \overset{n}{\underset{i=1}{\bigoplus}} R/\mathfrak{m}_i$, defined by $r\mapsto (r+\mathfrak{m}_1,\ldots,r+\mathfrak{m}_n)$  is an injection of left $R$-modules due to the assumption  $0=\mathrm{rad}R=\bigcap^n_{i=1}\mathfrak{m}_i$. The module $\overset{n}{\underset{i=1}{\bigoplus}} R/\mathfrak{m}_i$ has a composition series and hence, by the previous, so does $R$. Thus, $R$ is artinian and hence is semilocal. \\
Conversely, assuming that $R\neq0$ is artinian and commutative, we can use the Wedderburn theorem to show that $R$ is a direct product of a finite number of fields\footnote{Since $\mathrm{rad}R=0$, and $R$ is artinian, the Wedderburn theorem says that $R\cong\mathbb{M}_{n_1}(D_1)\times\cdots\times\mathbb{M}_{n_k}(D_k)$, where the $D_i$'s are division rings. Now, the assumed commutativity forces the size of each matrix ring to be 1, lest we have non-commuting elements like $\left( \begin{smallmatrix} 1&0 \\ 0&0 \end{smallmatrix}\right) ,\left( \begin{smallmatrix} 1&0 \\ 1&0 \end{smallmatrix}\right) $. Then, further, the division rings should be fields.}. Then the number of maximal ideals is the finite number of  these fields (a maximal ideal is obtained by fixing one and only one component in the prescribed decomposition equal to zero).\\
\end{proof}

Noticing that left or right artinian rings are semilocal, we see that semilocal rings generalizes one-sided artinian rings. If $R$ is left (right) artinian, we know that $\mathrm{rad}R$ is nilpotent (Theorem \ref{artradnilp}). So it would be reasonable to consider the class of those semilocal rings with $\mathrm{rad}R$ nilpotent, namely, \emph{semiprimary} rings. This clearly generalizes, again, one-sided artinian rings.\\

\begin{definition}(Semiprimary Rings). A ring $R$ is \emph{semiprimary} if $R/\mathrm{rad}R$ is semisimple and $\mathrm{rad}R$ is nilpotent.\end{definition}
 Once again in this chart of generalizations, we know that if $I$ is a nil  ideal then idempotents lift modulo $I$, that is, if $\overline{e}\in\overline{R}=R/I$ is an idempotent then there is an idempotent $x\in R$ such that $\overline{x}=\overline{e}$. We may then generalize semiprimary rings by considering  those semilocal rings for which idempotents lift modulo $\mathrm{rad}R$, these are \emph{semiperfect} rings. In this text, however, we shall not consider  \emph{perfect} rings, lest a digression would occur. Refer to \cite{Lamnoncom} for further readings.

 \begin{definition}\label{semiper}(Semiperfect Rings). A ring $R$ is \emph{semiperfect} if $R$ is semilocal, and idempotents lift modulo $\mathrm{rad}R$. (The last statement is sometimes said as: idempotents of $R/\mathrm{rad}R$ can be lifted to $R$.)

 \end{definition}
\begin{par}
We have the following scheme (\cite{Lamnoncom}, p.345):
\begin{center}$\{\text{one-sided artinian rings}\}$\\$\cap$\\$\{\text{semiprimary rings}\}$\\$\cap$\\$\quad\quad\{\text{local rings}\}\subset
\{\text{semiperfect rings}\}\subset\{\text{semilocal rings}\}.$\end{center}\end{par}

Indeed, any local ring $R$ is semiperfect since $R/\mathrm{rad}R$ contains only the trivial idempotents, being a division ring. Also note that the class of semiperfect rings is proper in the class of semilocal rings. To see this, let $R$ be a  semilocal domain (commutative) with two maximal ideals $\mathfrak{m}_1,\mathfrak{m}_2$. Then $R/\mathrm{rad}R\cong R/\mathfrak{m}_1\times R/\mathfrak{m}_2$ and hence has two nontrivial idempotents whilst $R$ lacks any nontrivial idempotents.

\subsection{Idempotents}
In his lithograph, \emph{Linear Associative Algebra} (1872), Benjamin Peirce introduced the notions of an idempotent and a nilpotent element and the famous Peirce decompositions. A modified version of this, including notes of his son, Charles Sanders Peirce, was published in the \emph{American Journal of Mathematics}, volume 4, in 1881. For any idempotent $e$, we have the three Peirce decompositions: \begin{enumerate}
\item[$(1)$] $R=Re\oplus Rf$,
\item[$(2)$] $R=eR\oplus fR$,
\item[$(3)$] $R=eRe\oplus  eRf\oplus fRe\oplus fRf$,
\end{enumerate}
where $f=1-e$ is the ``complementary'' idempotent to $e$. Note that the first two are decompositions into one-sided ideals, while  the third is a decomposition into  subgroups, the summands, in fact, are rings. The ring $eRe$ has the identity $e$.

\begin{proposition}\label{nillift} Idempotents lift modulo every nil ideal $I$ $\mathrm{(}I\subset\mathrm{rad}R$, since the radical contains every nil ideal $\mathrm{)}$.
\end{proposition}
\begin{proof}Let $a\in R$  be such that $ \overline{a}\in \overline{R}=R/I$ is an idempotent. Set $b=1-a$, then $ab=ba=a-a^2\in I$, so there exists $m$ such that $(ab)^m=0$. By the binomial theorem,
\begin{align*}
 1&=(a+b)^{2m}\\&=a^{2m}+r_1 a^{2m-1} b+\cdots+r_m a^m b^m+r_{m+1}a^{m-1}b^{m+1}+\cdots+b^{2m},
\end{align*}
the $r_i$'s being integers. Now, let
\begin{align*}
 e&=a^{2m}+r_1 a^{2m-1} b+\cdots+r_m a^m b^m,\quad\text{and}\\f&=r_{m+1}a^{m-1}b^{m+1}+\cdots+b^{2m}.
\end{align*}
$ef=0$ since $a^mb^m=b^ma^m=0$, and so $e=e(e+f)=e^2$. Finally, $ab\in I$ gives that $e\equiv a^{2m}\equiv a \;$(mod $I$).

\end{proof}

We now proceed to define some kinds of idempotents that will be used in the coming sections, and in order to define these properly, we shall have some idempotent-tools in hand.

\begin{proposition}\label{21.6}
Let $e,e'$ be idempotents, and $M$ be a right $R$-module. Then, as groups, $\mathrm{Hom}_R(eR,M)\cong Me$. In particular, $\mathrm{Hom}_R(eR,e'R)\cong e'Re$.

\end{proposition}

\begin{proof} Consider any $R$-homomorphism $\eta:eR\rightarrow M$, and set $m=\eta(e)$. Then $$me=\eta(e)e=\eta(e^2)=\eta(e)=m.$$Thus, $m\in Me$. Define $\lambda:\mathrm{Hom}_R(eR,M)\rightarrow Me$ by $\eta\mapsto \eta(e)$. It is clear that $\lambda$ is an injective group homomorphism. Now we show $\lambda$ is surjective. Let $m$ be any element in $Me$ and define $\eta:eR\rightarrow M$ by $\eta(er)=mr,\; r\in R$. $\eta$ is a well defined $R$-homomorphism since $er=0$ implies that $mr\in Mer=0$. Finally, $\lambda(\eta)=\eta(e)=m$, and $\lambda$ is surjective.

\end{proof}

\begin{corollary}\label{21.7} If $e\in R$ is any idempotent, then, $\mathrm{End}_R(eR)\cong eRe$, as rings.\end{corollary}

\begin{proof}First, by the previous proposition, setting $e'=e$, we have a group isomorphism $\lambda:\mathrm{End}_R(eR)\rightarrow eRe$. Further,  this is a ring isomorphism. Let $\eta,\eta'\in \mathrm{End}_R(eR)$ and notice that $\eta'(e)\in eR$. Then
$$\lambda(\eta\eta')=\eta\eta'(e)=\eta(\eta'(e))=\eta(e\eta'(e))=\eta(e)\eta'(e)=\lambda(\eta)\lambda(\eta').$$\end{proof}

Two idempotents $\alpha,\beta\in R$ are \emph{orthogonal} if $\alpha\beta=\beta\alpha=0$. To proceed, it is also necessary to make the following definitions.
\begin{definition} Let $R$ be any ring, and $M\neq0$ be any left (right) $R$-module.
\begin{enumerate}
\item[$(1)$] $M$ is said to be \emph{indecomposable} if $M$ cannot be written as the direct sum of two nonzero submodules.
\item[$(2)$] $M$ is said to be \emph{strongly indecomposable} if $\mathrm{End}_R(M)$ is a local ring.
\end{enumerate}
\end{definition}

The definition of strongly indecomposable modules may seem unnatural in this context. However, once we check easily that, $M$ is indecomposable if and only if the ring $\mathrm{End}_R(M)$ has no nontrivial idempotents, the definition is more than justified, since, in particular, local rings have no nontrivial idempotents.

\begin{proposition}\label{21.8}\textbf{(Primitive Idempotents)}. For any idempotent $e\in R$, the following are equivalent:
\begin{enumerate}
\item[$(1)$] $eR$ is an  indecomposable right $R$-module.
\item[$(1')$] $Re$ is an indecomposable left $R$-module.
\item[$(2)$] The ring $eRe$ has no nontrivial idempotents.
\item[$(3)$] $e$ has no decomposition into $\alpha+\beta$ where $\alpha,\beta$ are nonzero orthogonal idempotents in $R$.
\end{enumerate}

If the idempotent $e\neq0$ satisfies any of these conditions, $e$ is  said to be a primitive idempotent of $R$.
\end{proposition}

\begin{proof} The last two conditions are left-right symmetric, so it is enough to show that $(1)\Longleftrightarrow(2)\Longleftrightarrow(3)$. The first equivalence is apparent from Corollary \ref{21.7}, since $eR$ is indecomposable if and only if $\mathrm{End}_R(eR)$ has no nontrivial idempotents.\\
$(3)\Longrightarrow(2)$. If $eRe$ has a nontrivial idempotent $\alpha$, then the complementary idempotent to $\alpha$ in the ring $eRe$ is $\beta=e-\alpha$ and we have the orthogonal decomposition  $e=\alpha+\beta$, contradiction.\\
$(2)\Longrightarrow(3)$. Suppose we have the orthogonal decomposition $e=\alpha+\beta$ where $\alpha,\beta$ are nonzero. Then $e\alpha=\alpha^2+\beta\alpha\alpha=\alpha$ and $\alpha e=\alpha^2+\alpha\beta=\alpha$. Thus, $\alpha=\alpha e=(e\alpha)e\in eRe$, and of course $\alpha$ is nontrivial ($\alpha\neq0$ and $\alpha\beta=0$), contradiction.

\end{proof}

\begin{proposition}\label{21.9}\textbf{(Local Idempotents)}. For any idempotent $e\in R$, the following are equivalent:
\begin{enumerate}
\item[$(1)$] $eR$ is a strongly indecomposable right $R$-module.
\item[$(1')$] $Re$ is a strongly indecomposable left $R$-module.
\item[$(2)$]  $eRe$ is a local ring.
\end{enumerate}

If the idempotent $e\neq0$ satisfies any of these conditions, it is called a local idempotent of $R$ (clearly, a local idempotent is primitive).
\end{proposition}

\begin{proof}
$(1)\Longleftrightarrow(2)$ follows from Corollary \ref{21.7}.  $(1')\Longleftrightarrow(2)$ then  follows from left-right symmetry.
\end{proof}

\begin{theorem}\label{21.10} Let $J=\mathrm{rad}R$ and $\overline{R}=R/J$. Then for any idempotent $e\in R$, $\mathrm{rad}(eRe)=J\cap(eRe)=eJe$. Moreover, $eRe/\mathrm{rad}(eRe)\cong \overline{e}\;\overline{R}\;\overline{e}$.
\end{theorem}

\begin{proof}Suppose that $r\in \mathrm{rad}(eRe)$. We show that $r\in J$ by showing that, for any $x\in R$, $1-xr$ has a left inverse in $R$ (Lemma \ref{4.1lam}). In the ring $eRe$, the element $(e-exe \cdot r)$ has a left inverse, $b$, say. Thus, $e=b(e-exe \cdot r)=be(1-xe\cdot r)=b(1-xr)$. Then, $$xrb(1-xr)=xre=xr.$$ Adding $1-xr$, we get $(1+xrb)(1-xr)=1$. Therefore, $r\in J\cap eRe$. Moreover, $r=ere\in eJe$.

Conversely, if $r\in eJe$, we show that for any $y\in eRe$, $e-yr$ is left invertible in $eRe$. Since $r\in eJe\subset J$, there is $x\in R$ such that $x(1-yr)=1$. then $$e=ex(1-yr)e=ex(e-yr)=exe\cdot(e-yr),$$ which is the desired result. Thus, $\mathrm{rad}(eRe)=J\cap(eRe)=eJe$.
For the other claim, the map $ere\mapsto \overline{e}\;\overline{r}\;\overline{e} $ is a well defined ring  epimorphism from $eRe$ onto $\overline{e}\;\overline{R}\;\overline{e}$. Clearly, the map vanishes on $eJe$. Also, if $\overline{e}\;\overline{r}\;\overline{e}=0$, this implies that $ere\in J\cap eRe=eJe$. Thus, the kernel is exactly $eJe$, and the map induces an isomorphism $eRe/eJe\cong \overline{e}\;\overline{R}\;\overline{e}$.

 \end{proof}

\begin{proposition}\label{21.18}Adapting the same notation as above, the following are equivalent.
\begin{enumerate}
\item[$(1)$] $e$ is a local idempotent.
\item[$(2)$] $\overline{e}\overline{R}$ is a minimal right ideal in $\overline{R}$  (simple $\overline{R}$-module).
\item[$(3)$] $eR/eJ$ is a simple right $R$-module.
\end{enumerate}
\end{proposition}

\begin{proof}
First, $\overline{e}\overline{R}$ is a simple right  $\overline{R}$-module if and only if $\mathrm{End}_{\overline{R}\;}(\overline{e}\overline{R})$ is a division ring, which is equivalent, by Corollary \ref{21.7}, to that $\overline{e}\;\overline{R}\;\overline{e}$ is a division ring. Now, from Theorem \ref{21.10}, $$eRe/\mathrm{rad}(eRe)\cong \overline{e}\;\overline{R}\;\overline{e},$$ hence $ \overline{e}\;\overline{R}\;\overline{e}$ is a division ring if and only if $eRe$ is a local ring. This gives $(1)\Longleftrightarrow(2)$.\\
For $(2)\Longleftrightarrow(3)$, we have an isomorphism  $\lambda:eR/eJ\rightarrow \overline{e}\overline{R}$ of right $\overline{R}$-modules.

\end{proof}

\begin{proposition}\label{21.20}\textbf{(Isomorphic Idempotents)}. The following are equivalent for any two idempotents $e,f\in R$.
\begin{enumerate}
\item[$(1)$] $eR\cong fR$ as right $R$-modules.
\item[$(1')$] $Re\cong Rf$ as left $R$-modules.
\item[$(2)$] There exist   $a\in eRf$ and $b\in fRe$ such that $e=ab$ and $f=ba$.
\item[$(3)$]  There exist   $a,b\in R$ such that $e=ab$ and $f=ba$.
\end{enumerate}
If two idempotents $e$ and $f$ satisfy any of these conditions, we say that they are isomorphic idempotents, written $e \cong f$.
\end{proposition}

\begin{proof}The last two conditions are left-right  symmetric, so it is enough to prove that $(1)\Longrightarrow(2)\Longrightarrow(3)\Longrightarrow(1)$.\\
$(1)\Longrightarrow(2)$. Let $\theta: eR\rightarrow fR$ be an isomorphism. Taking $b=\theta(e)=\theta(e)e\in fRe$ and $a=\theta^{-1}(f)=\theta^{-1}(f)f\in eRf$, we have $$ab=\theta^{-1}(f)b=\theta^{-1}(fb)=\theta^{-1}(b)=\theta^{-1}(\theta(e))=e,$$ and similarly, $ba=f$.\\
$(2)\Longrightarrow(3)$. Nothing to prove.\\
$(3)\Longrightarrow(1)$. Given such $a$ an $b$, define $\theta: eR\rightarrow fR$ and $\theta': fR\rightarrow eR$ by $\theta(er)=ber =(ba)br=fbr\in fR$, and $\theta'(fr)=afr=(ab)ar=ear\in eR$. Then
\begin{align*}
\theta'\theta(e)&=\theta'(be)=abe=e^2=e,\\ \theta\theta'(f)&=\theta(af)=baf=f^2=f.
\end{align*}
Hence, $\theta'\theta=1$ and $\theta\theta'=1$, and $\theta$ is an R-isomorphism.

\end{proof}

Notice that, in a commutative ring, $e\cong f$ if and only if $e=f$. Another thing is the following. If $e'=1-e$ is the complementary idempotent to the idempotent $e$, we have the decomposition $R=Re\oplus Re'$, so $P=Re$ is a projective left $R$-module. For any ideal $I\subset R$, putting $\overline{R}=R/I$, it is easily verified that $$P/IP=Re/Ie\cong \overline{R}\;\overline{e}$$ as  $\overline{R}$-modules. Now, Proposition \ref{19.27} results in the following.
\begin{proposition}\label{21.21} Let $I\subset\mathrm{rad}R$ be an ideal. Then for any two idempotents $e,f\in R$, $e\cong f$ if and only if $\;\overline{e}\cong\overline{f}$ in $\overline{R}$.\end{proposition}
The next proposition  treats the problem of lifting primitive idempotents, but before getting into it, the following  lemma is needed.

\begin{lemma}\label{neednow}Let $I\subset\mathrm{rad}R$ be an ideal and $\alpha,\beta$ be idempotents in $R$ such that $\alpha\beta\equiv\beta\alpha\equiv0\; \mathrm{(mod\;} I\mathrm{)}$ (i.e. $\overline{\alpha}$ and $\overline{\beta}$ are orthogonal in $\overline{R}$). There exists an idempotent $\beta'\in R$ orthogonal to $\alpha$ such that $\beta'\equiv\beta\; \mathrm{(mod\;} I\mathrm{)}$. \end{lemma}

\begin{proof}First, $\beta\alpha\in I\subset\mathrm{rad}R$ implies that $1-\beta\alpha$ is a unit. Consider the idempotent $$\beta_0=(1-\beta\alpha)^{-1}\beta(1-\beta\alpha),$$
and check easily that $\overline{\beta}_0=\overline{\beta}$, and$$\beta_0\alpha=(1-\beta\alpha)^{-1}\beta(\alpha-\beta\alpha)=0.$$
The only trouble is that $\alpha\beta_0$ may not be zero. To overcome this, let $\beta'=(1-\alpha)\beta_0$. Since $\overline{\alpha\beta_0}=\overline{\alpha}\overline{\beta}=0$, we have $\overline{\beta'}=\overline{\beta}_0=\overline{\beta}$. Then, $\beta'\alpha=(1-\alpha)\beta_0\alpha=0$, and $\alpha\beta'=\alpha(1-\alpha)\beta_0=0$. Finally, $\beta'$ is an idempotent since $\beta^{'2}=(1-\alpha)\beta_0(1-\alpha)\beta_0=(1-\alpha)\beta_0^2=\beta'$.
 \end{proof}

\begin{proposition}\label{21.22}Let $I\subset\mathrm{rad}R$ be an ideal and $e\in R$ be an  idempotent. If $\overline{e}$ is primitive in $\overline{R}=R/I$, then $e$ is primitive. The converse holds if idempotents of $\overline{R}$ can be lifted to $R$.\end{proposition}

\begin{proof}First of all, we have to pay attention to a little remark, that the only idempotent in $\mathrm{rad}R$ is $0$. To see this, suppose that $\alpha\in\mathrm{rad}R$ is an idempotent, then its complementary idempotent $1-\alpha$ is a unit, giving $1-\alpha=1$. Now we prove the proposition. Assume that $\overline{e}$ is primitive, and that  $e=\alpha+\beta$ is a decomposition of $e$ into nonzero  orthogonal idempotents. By the previous remark, $\overline{\alpha},\overline{\beta}$ are nonzero in $\overline{R}$. Thus, $\overline{e}=\overline{\alpha}+\overline{\beta}$ is a nontrivial orthogonal decomposition of $\overline{e}$, contradicting $\overline{e}$ is primitive.

Conversely, suppose that $\overline{e}=\overline{x}+\overline{y}$ is a nontrivial decomposition of $\overline{e}$ into orthogonal idempotents $\overline{x},\overline{y}$ in $\overline{R}$. Assuming that idempotents lift modulo $I$, let $\alpha$ and $\beta$ be idempotents, in $R$, lifting $\overline{x}$ and $\overline{y}$, respectively, that is, $\overline{\alpha}=\overline{x}$ and $\overline{\beta}=\overline{y}$. We then have $\alpha\beta\equiv\beta\alpha\equiv0\; \mathrm{(mod\;} I\mathrm{)}$, and Lemma \ref{neednow} provides a nonzero\footnote{Were it equal to zero, we would have had $\beta\in I$ and $\overline{y}=\overline{\beta}=0$, contradiction.} idempotent $\beta'\in R$ orthogonal to $\alpha$ such that $\overline{\beta'}=\overline{\beta}$. To complete the proof, define the idempotent $e'=\alpha+\beta'$, which is not primitive since $\alpha, \beta'\neq 0$. Besides, $\overline{e'}=\overline{\alpha}+\overline{\beta'}=\overline{\alpha}+\overline{\beta}=\overline{e}\;\; \text{in}\;\; \overline{R},$ and, by Proposition \ref{21.21}, $e'\cong e$ in $R$, thus $e$ is not primitive since $e'$ is not.\end{proof}

We conclude this section with the following proposition. In fact, most of the results  developed  hitherto were introduced to gently prove this  proposition without referring the reader to an external text.

\begin{proposition}\label{23.5} In a semiperfect ring $R$, any primitive idempotent is local.\end{proposition}

\begin{proof}Let $e$ be a primitive idempotent in $R$. By Proposition \ref{nillift}, idempotents lift modulo $\mathrm{rad}R$, and then Proposition \ref{21.22} says that $\overline{e}$ is primitive in $\overline{R}=R/\mathrm{rad}R$. Since $\overline{R}$ is semisimple, every $\overline{R}$-module is semisimple, and hence any indecomposable $\overline{R}$-module is simple. In particular, $\overline{e}\overline{R}$ is simple, and $e$ is local by Proposition \ref{21.18}.   \end{proof}

\section{Characters}\label{characters}

\begin{par} This section  is devoted to develop some elements from the theory of group characters. These objects, characters, shall turn out to be very crucial when plunging into our main problem. Here we are mainly following \cite{r7} and \cite{r3}.\end{par}
\subsection{Characters of Finite Abelian Groups}
\begin{par}Let $G$ be a finite abelian group. A group homomorphism $\pi:G\rightarrow\mathbb{C}^\times$ into the multiplicative group of complex numbers is called a  character  of $G$. The set of all characters of $G$ will be denoted $\widehat{G}$, thus $\widehat{G}=\mathrm{Hom}_\mathbb{Z}(G,\mathbb{C}^\times)$. $\widehat{G}$ is a finite abelian group under the operation $(\pi\theta)(x)=\pi(x) \theta(x)$, the inverses are obtained through complex conjugation: $\pi^{-1}(x)={\pi(x)}^{-1}=\overline{\pi(x)}/|\pi(x)|^2$. \end{par}

\begin{proposition}\label{char}Let $G$ be a finite abelian group with character group $\widehat{G}$. Then
\begin{enumerate}
\item[1.] $\widehat{G}\cong G$,
\item[2.] $G\cong \widehat{\widehat{G}}$,
\item[3.] $|\widehat{G}|=| G|$,
\item[4.] $(G_1\times G_2\widehat{)}\cong\widehat{G}_1\times\widehat{G}_2$;
\item[5.] $\sum_{x\in G}\pi(x)=\left\{\begin{matrix}
 |G|& \quad \text{if } \pi=1 \\
  \quad0 & \quad \text{if } \pi\neq 1
 \end{matrix}
     \right.$,

\item[6.]$\sum_{\pi\in \widehat{G}}\pi(x)=\left\{\begin{matrix}
 |G|& \quad \text{if } x=0 \\
  \quad0 & \quad \text{if } x\neq 0
 \end{matrix}
     \right.$.
\end{enumerate}
\end{proposition}

\begin{proof}
1. Case1: Assume that  $G$ is cyclic, say $G=\langle g\rangle$ for some $g\in G$. Let $|G|=n$. Since the group $U_n$ of $n^\text{th}$ roots of unity in $\mathbb{C}$ is also cyclic of order $n$, we must have $G\cong U_n$ and it follows that  the only characters of $G$ are defined by $\pi_i:g\mapsto u_i$, where $i=1,\ldots,n$ and $U_n=\{u_1,\ldots,u_n\}$. Thereby,  $\widehat{G}\cong U_n\cong G$.\\
Case2: If $G$ is not cyclic, then it is well known that $G$ is isomorphic to a finite direct product of cyclic groups, $G\cong G_1\times\cdots\times G_k$, say. By part (4) and induction, $(G_1\times\cdots\times G_k\widehat{)}\cong \widehat{G}_1\times\cdots\times \widehat{G}_k$. Now, case1 applied to each $G_i$, we get $$\widehat{G}\cong(G_1\times\cdots\times G_k\widehat{)}\cong \widehat{G}_1\times\cdots\times \widehat{G}_k\cong G_1\times\cdots\times G_k \cong G.\\$$
2. By part (1). Also, there is a natural map from $G$ to $\widehat{\widehat{G}}$ defined by $$g\mapsto (\widehat{g}:\pi\mapsto \pi(g)).$$ This map is in fact an isomorphism.\\[0.5cm]
3. By part (1).\\[0.5cm]
4. Let $\phi$ be a character of $G_1\times G_2$ and define characters $\phi_1$ and $\phi_2$ of $G_1$ and $G_2$ respectively as follows: $\phi_1(g_1)=\phi(g_1,0)$ and $\phi_2(g_2)=\phi(0,g_2)$ for $g_1\in G_1$ and $g_2\in G_2$. Now,
\begin{align*}
 \phi(g_1,g_2)&=\phi((g_1,0)+(0,g_2))= \phi((g_1,0))\phi((0,g_2))\\
 &=\phi_1(g_1)\phi_2(g_2).
\end{align*} The reader may check that the map $\phi\mapsto(\phi_1,\phi_2)$  is an isomorphism.\\[0.5cm]
5. If $\pi=1$, then $\sum_{x\in G}\pi(x)=\sum_{x\in G}1=|G|.$ If $\pi\neq1$, then there is $g\in G$ such that $\pi(g)\neq1$. Set $S=\sum_{x\in G}\pi(x)$, then \begin{align*}
 \pi(g)S&=\pi(g)\sum_{x\in G}\pi(x)=\sum_{x\in G}\pi(x+g)\\
 &=\sum_{y\in G}\pi(y)=S.
\end{align*} Since $\pi(g)\neq1$, it follows that $S=0$.\\[0.5cm]
6. If $x=0$, then $\sum_{\pi\in \widehat{G}}\pi(0)=\sum_{\pi\in \widehat{G}}1=|\widehat{G}|=|G|$, part (3) used in the last equality. If $x\neq0$, then there exists \footnote{Recall  the natural isomorphism indicated above in (2).} $\delta\in \widehat{G}$ with $\delta(x)\neq1$. Set $S=\sum_{\pi\in \widehat{G}}\pi(x)$, then  \begin{align*}
\delta(x)S&=\delta(x)\sum_{\pi\in \widehat{G}}\pi(x)=\sum_{\pi\in \widehat{G}}\delta(x)\pi(x)\\
 &=\sum_{\pi\in \widehat{G}}\delta\pi(x)=\sum_{\theta\in \widehat{G}}\theta(x)=S.
\end{align*} Now, since $\delta(x)\neq1$, it follows that $S=0$.
\end{proof}

\begin{par}Given a finite abelian group $G$, define its dual group to be the group $G^*=\langle\mathrm{Hom}_\mathbb{Z}(G,\mathbb{Q/Z}),+\rangle$. A quite similar proof to that of (4) in the last proposition applies to show that $(G_1\times G_2)^*\cong G^*_1\times G^*_2$. Now, we show that $|G^*|=|G|$. By the previous observation we need only consider the case when $G$ is cyclic. If $G=\langle g\rangle$ and $|G|=m$ then for each $0\leq k\leq m-1$, the action $g\mapsto \frac{k}{m}+\mathbb{Z}\;$ defines one and only one homomorphism in $ \mathrm{Hom}_\mathbb{Z}(G,\mathbb{Q/Z})$. This shows that $|G^*|=|G|$. \\
The map $\mathbb{Q/Z}\rightarrow \mathbb{C}^\times$ defined by the exponential $x\mapsto\exp^{2\pi ix}$ is a group monomorphism. This monomorphism induces a monomorphism from the dual group into the character group, and by the above argument, this makes the two groups in fact isomorphic to each other. Thus we may refer to $\mathrm{Hom}_\mathbb{Z}(G,\mathbb{Q/Z})$ as the character group when an additive form for characters is preferred.\end{par}
\begin{par}We will use the following convention concerning character notation: multiplicative characters in $\mathrm{Hom}_\mathbb{Z}(-,\mathbb{C}^\times)$ will be denoted by the standard Greek letters $\pi,\theta,\phi$ and $\rho$, while those corresponding additive characters in\\ $\mathrm{Hom}_\mathbb{Z}(-,\mathbb{Q/Z})$ are to be denoted by variant Greek letters such as $\varpi,\vartheta,\varphi$ and $\varrho$ (respectively). Thus, for example, $\theta=\exp^{2\pi i\vartheta}$.\end{par}

\begin{par}We say that the characters $\chi_1, \ldots,\chi_n:G\rightarrow \mathbb{C}^\times$ are linearly independent whenever $c_1\chi_1(g)+\cdots+c_n\chi_n(g)=0$, for every $g\in G$, implies that all the constants $c_i\in\mathbb{C}$ are equal to zero.\end{par}

\begin{theorem}\label{linearindep}If $G$ is a finite abelian group, then every subset of  $\widehat{G}$ of distinct characters is linearly independent.\end{theorem}
\begin{proof}The proof is by induction on the finite number $n$ of characters under consideration. For $n=1$, the result is trivial. Now, assume the statement is true for any set of $n-1$ distinct  characters. Let  $\chi_1, \ldots,\chi_n:G\rightarrow \mathbb{C}^\times$ be distinct characters of $G$ satisfying\begin{equation}\label{eq1.1charlin}c_1\chi_1(g)+\cdots+c_n\chi_n(g)=0\end{equation} for $c_i\in\mathbb{C}$ and for all $g\in G$. Since $\chi_1\neq\chi_n$, there exists $h\in G$ such that $\chi_1(h)\neq\chi_n(h)$. The equation above is true for every $g\in G$, hence, in particular, it works for $hg$, $$c_1\chi_1(h+g)+\cdots+c_n\chi_n(h+g)=c_1\chi_1(h)\chi_1(g)+\cdots+c_n\chi_n(h)\chi_n(g)=0.$$ Multiplying equation \ref{eq1.1charlin} by $\chi_n(h)$ and subtracting it from the previous equation one gets$$c_1(\chi_1(h)-\chi_n(h))\chi_1(g)+\cdots+c_n(\chi_{n-1}(h)-\chi_n(h))\chi_{n-1}(g)=0,$$ for every $g\in G$.  We have, by the induction hypothesis, $c_i(\chi_i(h)-\chi_n(h))=0$; $i=1,\ldots,n$. But $\chi_1(h)-\chi_n(h)\neq0$, hence $c_1=0$. Similarly, we prove that $c_i=0$ for $i=1,\ldots,n$.
\end{proof}

\begin{definition}(Annihilator)\label{annchar}. Let $H$ be a subgroup of $G$, we define the group $$(\widehat{G}:H):=\{\varpi\in\widehat{G}:\varpi(h)=0 \quad\quad \text{for\; all}h\in H\}$$ to be the annihilator of $H$.  The annihilator $(\widehat{G}:H)$ is isomorphic to $(G/H\widehat{)}$ and hence $|(\widehat{G}:H)|=|G|/|H|.$\end{definition}

\begin{proposition}\label{(g:h)}Let $H$ be a subgroup of $G$ such that $H\subset \mathrm{Ker}\varpi$ for every $\varpi\in \widehat{G}$. Then $H=0.$\end{proposition}
\begin{proof}Since $H\subset \mathrm{Ker}\varpi$ for every $\varpi\in \widehat{G}$, then $(\widehat{G}:H)=\widehat{G}$. Thus $|H|=|G|/|(\widehat{G}:H)|=|G|/|\widehat{G}|=1$.\end{proof} Actually, this proposition can be proved in view of the footnote in page 16.\\

\subsection{The Character Functor}

In the coming chapters we are always working on finite modules over finite rings. Let $R$ be a finite ring with unity. Regardless the ring action, a finite  $R$-module $M$ is, in the first place, a finite abelian group. Hence, we can consider the character group  $\widehat{M}$ of $M$.\\ If $M$ is a left $R$-module, then $\widehat{M}$  can be equipped with a right $R$-module structure given by:\\

(additive form)  $\quad(\varpi r)(m):=\varpi(rm), \qquad \varpi\in \widehat{M},\quad r\in R, \quad m\in M,$\\
(multiplicative form)  $\quad\pi^r(m):=\pi(rm),\qquad \pi\in \widehat{M},\quad r\in R, \quad m\in M$.\\

Reversely, if $M$ is a right $R$-module, then $\widehat{M}$ is equipped with a left $R$-module structure given by:\\

(additive form)  $\quad(r\varpi )(m):=\varpi(mr), \qquad \varpi\in \widehat{M},\quad r\in R, \quad m\in M,$\\
(multiplicative form)  $\quad^r\pi(m):=\pi(mr),\qquad \pi\in \widehat{M},\quad r\in R, \quad m\in M$.

\begin{lemma}\label{111111}Let $R$ be a finite ring, with $\widehat{R}$ its character bimodule. If $r\widehat{R}=0$, then $r=0$.\end{lemma}
\begin{proof}If $r\widehat{R}=0$, then for every $\varpi\in\widehat{R} $ and every $x\in R$, we have that $0=r\varpi(x)=\varpi(xr)$. Thus, $Rr \subset \mathrm{Ker}\varpi$ for every $\varpi\in\widehat{R} $. By Proposition \ref{(g:h)}  $Rr=0$ and hence $r=0$.  \end{proof}

\begin{proposition} The mapping $\;\widehat{ }\;$ taking the left $R$-module $M$ to $\widehat{M}$ is a contravariant  functor from the category of finitely generated left (right, respectively) $R$-modules to the category of finitely generated right (left, respectively) $R$-modules.\end{proposition}

\begin{proof}We only have to show that $\mathrm{Hom}_\mathbb{Z}(-,\mathbb{Q/Z})$ is a contravariant functor.
\begin{par}Let $A,B$ be two left $R$-modules. For every $R$-module homomorphism $\phi:A\rightarrow B$, define $\phi^*:\mathrm{Hom}_\mathbb{Z}(B,\mathbb{Q/Z})\rightarrow\mathrm{Hom}_\mathbb{Z}(A,\mathbb{Q/Z})$ by $\phi^*:g\mapsto g\circ\phi $. Then $\phi^*$ is an $R$-module homomorphism. Now,
\begin{enumerate}
\item[$(i)$] $\widehat{A}=\mathrm{Hom}_\mathbb{Z}(A,\mathbb{Q/Z})$ is a right $R$-module.
\item[$(ii)$]For the homomorphism $1_A:A\rightarrow A$, $1^*_A(g)=g\circ 1_A=g$ for every group homomorphism $g\in \mathrm{Hom}_\mathbb{Z}(A,\mathbb{Q/Z})$. Thus $1^*_A=1_{\mathrm{Hom}_\mathbb{Z}(A,\mathbb{Q/Z})}$.
\item[$(iii)$]Suppose that $A\overset{\phi}{\rightarrow} B\overset{\psi}{\rightarrow} C$ is a sequence of $R$-module homomorphisms. Then, at the action of the $\;\widehat{ }\;$ (claimed) functor, we get $$\mathrm{Hom}_\mathbb{Z}(A,\mathbb{Q/Z})\overset{\phi^*}{\leftarrow}\mathrm{Hom}_\mathbb{Z}(B,\mathbb{Q/Z})
    \overset{\psi^*}{\leftarrow}\mathrm{Hom}_\mathbb{Z}(C,\mathbb{Q/Z}).$$

Now, for every $g\in\mathrm{Hom}_\mathbb{Z}(C,\mathbb{Q/Z})$, \begin{align*}(\psi\circ\phi)^*(g)&=g\circ(\psi\circ\phi)=(g\circ\psi)\circ\phi\\
&=\psi^*(g)\circ\phi=\phi^*(\psi^*(g))\\
&=(\phi^*\circ\psi^*)(g).\end{align*}
\end{enumerate}\end{par}

Thus, $\;\widehat{ }\;$ is a contravariant functor. \end{proof}

In the next we prove further properties of the character functor.\\

\begin{lemma} $\mathbb{Q/Z}$ is an injective $\mathbb{Z}$-module.\end{lemma}
\begin{proof}It is easily checked that the abelian group $\mathbb{Q/Z}$ is divisible. The result now follows by Lemma IV.3.11 \cite{hungr}. \end{proof}

\begin{proposition}\label{ex}The character functor $\;\widehat{ }\;$ is an exact functor.\end{proposition}
\begin{proof} By the previous lemma, the result follows by Proposition IV.4.6 \cite{hungr}.\end{proof}

\begin{corollary}\label{Rhatinj}$\widehat{R}$ is an injective $R$-module.\end{corollary}
\begin{proof} An exact functor takes a projective module to an  injective module. \end{proof}

\begin{proposition}\label{hynf3}If $ M$ is a finite right $R$-module, then $$(\widehat{M}:M\mathrm{rad}R)=\mathrm{soc}(\widehat{M}),\text{\qquad and\qquad} (M/M\mathrm{rad}R\widehat{)}\cong\mathrm{soc}(\widehat{M}).$$\end{proposition}

\begin{proof}\begin{par}Before proceeding in the proof, just remember that $\widehat{\widehat{M}}\cong M$ as groups, and easily check they are isomorphic as right $R$-modules. This, together with the established exactness, give that the character functor takes simple modules to simple modules. Also, the exactness  gives that direct sums are taken to direct sums (Theorem IV.1.18 \cite{hungr}). Another fact is that whenever $N_R\subset M_R$ is a submodule then $(\widehat{M}:N)$ is a left  submodule of $_R\widehat{M}$. We plunge into the proof with the following isomorphism of left $R$-modules.
\begin{align*}(M/M\mathrm{rad}R\widehat{)}&=\mathrm{Hom}_\mathbb{Z}(M/M\mathrm{rad}R,\mathbb{Q/Z})\\
                                                                     &\cong \{f:M \xrightarrow[\mbox{\tiny{hom.}}]{\mbox{\tiny{group}}}\mathbb{Q/Z} \;\;| f(M\mathrm{rad}R)=0\}\\
                                                                     &=(\widehat{M}:M\mathrm{rad}R).\end{align*}\end{par}
Applying the character functor to the short exact sequence  $$0\rightarrow M\mathrm{rad}R\rightarrow M\rightarrow M/M\mathrm{rad}R\rightarrow0,$$we obtain the following short exact sequence of left $R$-modules $$0\rightarrow(\widehat{M}:M\mathrm{rad}R)\rightarrow \widehat{M} \rightarrow(M\mathrm{rad}R\widehat{)}\rightarrow 0.$$
Now, since $\overline{R}=R/\mathrm{rad}R$ is a semisimple ring, all $\overline{R}$-modules are semisimple by Theorem \ref{semisimple rings}. In particular, $(M/M\mathrm{rad}R)_{\overline{R}}$ is semisimple as a right $\overline{R}$-module and is therefore the sum of simple $\overline{R}$-modules.  But Proposition \ref{rad andsim} shows that $R$  and $\overline{R}$ have the same simple modules, hence $(M/M\mathrm{rad}R)_R$ is a semisimple $R$-module, and is the finite direct sum of simple $R$-submodules. By the observations in the beginning of the proof,  $(\widehat{M}:M\mathrm{rad}R)$ is semisimple (being isomorphic to $(M/M\mathrm{rad}R\widehat{)}$, the character module of  a semisimple module). Thus $(\widehat{M}:M\mathrm{rad}R)\subset\mathrm{soc}(\widehat{M}).$\\
Conversely, $\mathrm{rad}R$ annihilates simple modules, hence $\mathrm{rad}R\;\mathrm{soc}(\widehat{M})=0$. This gives that $\mathrm{soc}(\widehat{M})\subset(\widehat{M}:M\mathrm{rad}R)$. The proof is complete.

\end{proof}

\subsection{The Character Module of a Socle}

In this part (and most of the coming parts, actually) let $R$ be a finite  ring (hence artinian), and $_RA$ be a finite left $R$-module. By the Wedderburn-Artin theorem and Wedderburn's little theorem we have an isomorphism $ \varphi$ of rings
\begin{equation}\label{mu}R/\mathrm{rad}R\overset{\varphi}{\cong }\mathbb{M}_{\mu_1}(\mathbb{F}_{q_1})\times\cdots\times\mathbb{M}_{\mu_n}(\mathbb{F}_{q_n}),\end{equation}where  $\mathbb{F}_{q_i}$ is a finite  field of order $q_i$ and $\mu_i$ is an integer.

Recall that $R$ and $R/\mathrm{rad}R$ have the same simple modules. Then, as mentioned in section \ref{wed}, after equation (\ref{action}), the modules $_RT_i$ form a complete list of simple left $R$-modules, and the modules ${T'_i}_R$ form a complete list of simple right $R$-modules. In fact, as $R$-modules,
 \begin{equation}\label{wedec}_R(R/\mathrm{rad}R)\cong \mu_1T_1\oplus\cdots\oplus \mu_nT_n\;\;\text{and}\;\;(R/\mathrm{rad}R)_R\cong \mu_1T'_1\oplus\cdots\oplus \mu_nT'_n.\end{equation}
The socle of the module $A$ is hence decomposed as  $$\mathrm{soc}(A)\cong s_1T_1\oplus\cdots\oplus s_nT_n,$$for some multiplicities $s_1,\ldots,s_n$.

\begin{lemma}\label{simui}Let $R$ be a finite ring and $A$ be a left $R$-module.  $\mathrm{soc}(A)$ is cyclic if and only if $s_i\leq\mu_i$ for $i=1,\ldots,n\;$, where the $\mu_i$'s are as in equation (\ref{mu}), and the $s_i$'s as above.\end{lemma}

\begin{proof}First, just recall that $_RM$ is cyclic if and only if there is an epimorphism $_RR\rightarrow{}_RM$. If, for each $i$, $s_i\leq\mu_i$, then there is clearly an epimorphism  $R/\mathrm{rad}R\rightarrow \mathrm{soc}(A)$ of left $R$-modules. Then, we have an epimorphism $R\rightarrow R/\mathrm{rad}R\rightarrow \mathrm{soc}(A)$ and we  are done.

Conversely, assume that $\mathrm{soc}(A)=Ra$, for some $a\in \mathrm{soc}(A)$. The map $\theta:r\mapsto ra$ is an epimorphism from $R$ onto $\mathrm{soc}(A)$. Now, $$\mathrm{soc}(A)\cong R/\mathrm{Ker} \theta\cong(R/\mathrm{rad}R)/(\mathrm{Ker}\theta/\mathrm{rad}R),$$ where, of course, $\mathrm{rad}R\subset\mathrm{Ker}\theta$ since $\mathrm{soc}(A)$ is semisimple. But $R/\mathrm{rad}R$ is a semisimple left $R$-module, and  $\mathrm{Ker}\theta/\mathrm{rad}R$ is then a direct summand of $R/\mathrm{rad}R$. This gives a monomorphism $\mathrm{soc}(A)\rightarrow R/\mathrm{rad}R$, and the desired condition on the multiplicities holds indeed.

\end{proof}
Actually, the proof works for artinian rings. Also, we could have used Proposition \ref{socrad} directly.

\begin{lemma}\label{Mhat} Let $S$ be the ring $\mathbb{M}_n(\mathbb{F}_q)$, where $\mathbb{F}_q$ is a finite field. Let $M_S$ be the unique simple  right $S$-module ($M\cong \mathbb{F}_q^n$ as right $S$-modules), and let $_SN$ be the unique simple left $S$-module ($N\cong \mathbb{F}_q^n$ as left $S$-modules). Then \begin{center}$\widehat{M}\cong N \;\;\text{and}\;\; \widehat{N}\cong M.$\end{center}\end{lemma}

\begin{proof} $\widehat{\mathbb{F}}_q$ is a 1-dimensional vector space over $\mathbb{F}_q$, let $\chi\in \widehat{\mathbb{F}}_q$ be a basis (any nonzero element). Viewing the elements of $M$ and $N$, respectively, as row and column vectors, the map $N\rightarrow \widehat{M}$, given by $v\mapsto \pi_v$, where $\pi_v(m)=\chi(mv), m\in M$, is an isomorphism of left $S$-modules. \end{proof}

This yields immediately the next corollary. ($T_i,T'_i$ are as in the beginning of the section.)
\begin{corollary}If $R$ is  a finite ring then $$\widehat{T'}_i\cong T_i \;\; \text{ and }\;\;  \widehat{T}_i\cong T'_i\; ,\quad i=1,\ldots,n.$$\end{corollary}

\begin{proposition}\label{Rhat}For any finite left $R$-module, $\mathrm{soc}(A)$ is cyclic if and only if $A$ can be embedded into $_R\widehat{R}$. \end{proposition}

\begin{proof}Applying Proposition \ref{hynf3} to $R_R$, and that the exactness of the character functor (Proposition \ref{ex}), we get
\begin{align*}\mathrm{soc}(_R\widehat{R})\cong(R/\mathrm{rad}R\widehat{)_R}\cong (\overset{n}{\underset{i=1}{\oplus}}\mu_iT'_i\widehat{)}\cong \overset{n}{\underset{i=1}{\oplus}}\mu_i \widehat{T'}_i\cong \overset{n}{\underset{i=1}{\oplus}}\mu_iT_i.\end{align*}
If $A$ can be embedded into $_R\widehat{R}$, then $\mathrm{soc}(A)$ is isomorphic to a submodule of $\mathrm{soc}(_R\widehat{R})$, implying that $s_i\leq\mu_i$ for all $i$, and so $\mathrm{soc}(A)$ is cyclic.

Conversely, if $\mathrm{soc}(A)$ is cyclic, then from Lemma \ref{simui}, there is an embedding (monomorphism) $f:\mathrm{soc}(A)\rightarrow\mathrm{soc}(_R\widehat{R})\subset{ _R\widehat{R}}$.  But $_R\widehat{R}$ is injective by Corollary \ref{Rhatinj}, and hence the homomorphism (monomorphism, actually) $f$ extends to a homomorphism $F:A\rightarrow {_R\widehat{R}}$. We show that, further, $F$ is also a monomorphism.

Notice that $\mathrm{soc}(\mathrm{Ker}F)=\mathrm{Ker}F\cap\mathrm{soc}(A)=\mathrm{Ker} f=0$. Since all the modules here are finite, this gives at once that $\;\mathrm{Ker}F=0$.

\end{proof}

\chapter{Extension Theorem: \\Hamming Weight Setting}

\begin{par}The Extension Property (EP), with all its variations, stands as the main generator for the whole thesis. To start, we retract to the classical case in which codes are vector spaces over finite fields, where the early theorems of Florence J. MacWilliams hold. Two codes $C_1$ and $C_2$ may be thought referring to the same object in a couple of situations: (1) First, if $C_1$ and $C_2$ are isomorphic vector spaces via an isomorphism preserving Hamming weight, or, (2) if the two codes are monomially equivalent.\end{par}
\begin{par}In 1962, in her doctoral dissertation, MacWilliams showed the two notions are equivalent within the classical context over finite fields. Before proceeding, let's rephrase what is meant by saying the alphabet $_RA$ is MacWilliams (i.e. satisfying the Hamming weight EP):
$_RA$ is MacWilliams if any monomorphism $f:C\rightarrow A^n$ preserving Hamming weight extends to a monomial transformation. In this chapter we prove MacWilliams theorem following the character-theoretic proof introduced by H. Ward and J. Wood in \cite{r1996}. This kind of proofs turned out to be more dynamic, as we shall see.\end{par}
\section{Algebraic Coding}
\begin{par}Although the results will appear so abstract, the applied origin from coding theory is worth mention. It may now be the time to answer a question that probably haunted the reader: Why the objects are referred to as \emph{codes}? The few elements from coding theory quoted below provide a fine answer, and see \cite{Lint}, \cite{Mac book} for more.\end{par}
\begin{par} Imagine a message being sent as a sequence of $0|1$ bits from a transmitter, through some sort of channel (medium). Due to any potential disturbance, or a channel malfunction, an error may occur, causing, say, a 0-bit to be received as a 1-bit. The study of how messages (constituted as sequences from certain alphabets) are \emph{encoded} as elements of a special set, namely \emph{code}, in such a way that, when \emph{decoded} by the receiver, a correct message is delivered, is the subject of Coding Theory. A code, being a set, may have an algebraic structure, which codes are subject to our investigations here in the thesis. \end{par}

\begin{definition}(Linear Codes over $\mathbb{F}_q$). A \emph{linear code} $C$ of length $n$ over the finite field $\mathbb{F}_q$ is a subspace of the vector space $\mathbb{F}^n_q$. An element of $C$ is called a \emph{codeword}. If $C$ is a $k$-dimensional subspace of $\mathbb{F}^n_q$, we say that $C$ is an $[n,k]$ linear code over $\mathbb{F}_q$.\end{definition}

The following two definitions are used for representing an $[n,k]$ linear code by a matrix.
\begin{definition}(Generator Matrix). Let  $C\subset\mathbb{F}^n_q$ be an $[n,k]$ linear code. A $k\times n$ matrix $G$ over  $\mathbb{F}_q$ whose row space coincides with $C$ is called a generator matrix for $C$.\end{definition}
If $G$ is a generator matrix for $C$, then $C=\{xG|x\in \mathbb{F}^k_q\}$. A code may have more than one generator matrix. A generator matrix may be reduced to the form $G'=[I_k|A]$, where $I_k$ is the $k\times k$ identity matrix. In this case we say that $G'$ is in the standard form. If $G$ is in standard form, then the first $k$ entries of a codeword are called the information symbols. The remaining $n-k$ symbols are often called parity check symbols.
\begin{example}\label{eeee}Define the $[7,4]$ code $C$ over  $\mathbb{F}_2$ through the generator matrix
$$G=\begin{bmatrix}
       1&0&0&0&1&1&1 \\[0.3em]
       0&1&0&0&0&1&1\\[0.3em]
       0&0&1&0&1&0&1 \\[0.3em]
       0&0&0&1&1&1&0
     \end{bmatrix}.$$\end{example}
Thus, $G$ is in the standard form, and the coding scheme is the linear map
$\mathbb{F}^4_2\rightarrow\mathbb{F}^7_2$ given by$$(c_1,c_2,c_3,c_4)\mapsto(c_1,c_2,c_3,c_4,c_1+c_3+c_4,c_1+c_2+c_4,c_1+c_2+c_3).$$
\begin{definition}(Parity Check Matrix). Let $C\subset\mathbb{F}^n_q$ be an $[n,k]$ linear code. An $(n-k)\times n$ matrix $H$ over  $\mathbb{F}_q$, of rank $n-k$, is a \emph{parity check matrix} of $C$ if $C=\{c\;|\; Hc^T=0\}$. A code may have more than one parity check matrix.
\end{definition}

\begin{theorem}If $G=[I_k|A]$ is a generator matrix, in the standard form, for the $[n,k]$ linear code $C$, then $H=[-A^T|I_{n-k}]$ is a parity check matrix of $C$.\end{theorem}
\begin{proof} We clearly have $HG^T=-A^T+A^T=0$. Thus, $C$ is in the kernel of $x\mapsto Hx^T$. Since $H$ has rank $n-k$, the kernel of this map has dimension $k$, which is the same as the dimension of $C$. Hence, $C=\{c\;|\; Hc^T=0\}$. \end{proof}

\begin{example} For the code in Example \ref{eeee}, a parity check matrix can be found using the previous theorem:
$$H=\begin{bmatrix}
       1&0&1&1&1&0&0 \\[0.3em]
       1&1&0&1&0&1&0\\[0.3em]
       1&1&1&0&0&0&1 \\[0.3em]

     \end{bmatrix}.$$\end{example}

\begin{definition}(Hamming Weight). Let $C\subset\mathbb{F}^n_q$ be a linear code of length $n$. The \emph{Hamming weight} of a codeword $c=c_1,\ldots,c_n\in C$, is the number of nonzero components in the vector $c$.
\end{definition}

In many problems in algebraic coding theory, we need to have a detailed knowledge of the weight distribution of a code, i.e. knowing how many codewords have one weight or another. This is what weight enumerators are defined for.
\begin{definition}(Hamming Weight Enumerator). The \emph{Hamming weight enumerator} of a linear code $C$ of length $n$ is given by the polynomial $$W_C(X)=\underset{n}{\overset{j=0}{\sum}}A_jX^j,$$where $A_j$ denotes the number of codewords with Hamming weight $j$. By substituting $Y/X$ for $X$ and multiplying by $X^n$, $W_C(X)$ is converted into the two-variable weight enumerator $$W_C(X,Y)=\underset{n}{\overset{j=0}{\sum}}A_jX^{n-j}Y^j.$$\end{definition}

Among the fundamental results in coding theory are the  MacWilliams identities, proved in 1962 \cite{Mac}. The identities relate the Hamming weight enumerators of a code to those of its dual code $C^\bot$, generated by a parity check matrix of $C$.

\begin{theorem}(MacWilliams Identities). Let $C$ be a linear code of length $n$ over a finite field $\mathbb{F}_q$. The Hamming weight enumerators of $C$ and $C^\bot$ satisfy the following,
\begin{center}$W_{C^\bot}(X)=\textstyle{\tfrac{(1+(q-1)X)^n}{|C|}}W_C(\textstyle{\tfrac{1-X}{1+(q-1)X}})$,\\
$W_{C^\bot}(X,Y)=\textstyle{\tfrac{1}{|C|}} W_C(X+(q-1)Y, X-Y)$.\end{center}\end{theorem}

\section{ MacWilliams Theorem for Codes over Finite Fields }
In this classical context, where the codewords have entries from a finite field $\mathbb{F}_q$, the extension problem arose when two notions of equivalence of codes were posed.

\begin{definition}(Weight). A \emph{weight} defined on the finite field $\mathbb{F}_q$ is a function $w: \mathbb{F}_q\rightarrow\mathbb{Q}$ with $w(0)=0$. This function is then extended naturally to a weight on $\mathbb{F}^n_q$ by defining $w(x_1,\ldots,x_n)=\overset{n}{\underset{i=1}{\sum}}w(x_i)$. In particular, the Hamming weight, $\mathrm{wt}$, may be defined by $\mathrm{wt}(a)=\left\{ \begin{smallmatrix}1&,\;a\neq0\\0&,\;a=0\end{smallmatrix}\right. $, for all $a\in\mathbb{F}_q$.
\end{definition}

\begin{definition}(Preservation of Weight). A function $f: \mathbb{F}^n_q\rightarrow\mathbb{F}^n_q$ is said to \emph{preserve the weight} $w$ if  $w(x)=w(f(x))$ for every $x\in\mathbb{F}^n_q$.
\end{definition}

\begin{definition}(Isometries).  Let $C_1$ and $C_2$ be two codes of length $n$ over  $ \mathbb{F}^n_q$. We say that $\phi:C_1\rightarrow C_2$ is an \emph{isometry} if  $\phi$ is a vector space isomorphism that preserves Hamming weight, and then we say $C_1$ and $C_2$ are\emph{ isometric}.
\end{definition}

Isometries give a notion for the equivalence of codes, another notion is that carried by \emph{monomial equivalence}.

\begin{definition}(Monomial Transformation). A linear map $T: \mathbb{F}^n_q\rightarrow\mathbb{F}^n_q$ is called a \emph{monomial transformation} if $$T(x_1,\ldots,x_n)=(\lambda_1(x_{\sigma(1)}),\ldots,\lambda_n(x_{\sigma(n)})),$$ where $\sigma$ is a permutation of $\{1,\ldots,n\}$, and the $\lambda_i\in\mathbb{F}_q$ and  $(x_1,\ldots,x_n)\in\mathbb{F}^n_q$. $C_1$ and $C_2$ are said to be \emph{monomially equivalent} if there is a monomial transformation $T$ such that $T(C_1)=C_2.$
\end{definition}

Apparently enough, any two monomially equivalent codes over $\mathbb{F}_q$ are also isometric, but is the converse true? In her doctoral thesis \cite{Mac}, MacWilliams showed the two notions are equivalent for codes over finite fields. The result is known by ``MacWilliams Extension Theorem'' or ``MacWilliams Equivalence Theorem''.

\begin{theorem}\label{mac}(MacWilliams Extension Theorem, 1962). Two linear codes of length $n$ over a finite field $\mathbb{F}_q$ are isometric if and only if they are monomially equivalent.\end{theorem}

In 1978, another elementary proof of the theorem was presented by Bogart, Goldberg and Gordon in \cite{bogart}. The proof displayed below is due to J. Wood and H. Ward \cite{r1996}. This character theoretic proof was the generator for more proofs of the extension property in the context of general  module alphabets.\\

\setlength{\parindent}{0cm}\textbf{\emph{Proof of MacWilliams Extension Theorem.}}
\setlength{\parindent}{0.6cm}

First recall that if $(G,+)$ is a finite abelian group, a \emph{character} on $G$ is a group homomorphism $\pi:(G,+)\rightarrow(\mathbb{C}^\times,\cdot)$, where the latter is the multiplicative group of nonzero complex numbers. The set of all characters on $G$ forms another finite abelian group $\hat{G}$, with $|\hat{G}|=|G|$. From Proposition \ref{char}, if $\pi\in \hat{G}$ we know that

\[\underset{g\in G}{\sum}\pi(g)= \left\{
  \begin{array}{lr}
    |G|,& \pi=1,\\
    0, & \pi\neq1.
  \end{array}
\right.
\]

If $G=\mathbb{F}_q$, then $\hat{\mathbb{F}}_q$ has the structure of a $\mathbb{F}_q$-vector space, with scalar multiplication defined as $\pi^a(x)=\pi(ax), \;a,x \in \mathbb{F}_q$. Obviously, since $|\hat{\mathbb{F}}_q|=|\mathbb{F}_q|$, $\hat{\mathbb{F}}_q$ has dimension 1 over $\mathbb{F}_q$, and any nontrivial character forms a basis for $\hat{\mathbb{F}}_q$. For example, if char$(\mathbb{F}_q)=p$, and $\xi\in\mathbb{C}$ is a primitive $p^{\text{th}}$ root of unity, then the character $\exp x=\xi^{\text{tr}\;x} $ (with $\text{tr}:\mathbb{F}_q\rightarrow \mathbb{F}_p$ being the trace function $\alpha\mapsto\sum^{m-1}_{i=0}\alpha^i$ from $\mathbb{F}_q$ to its prime subfield $\mathbb{F}_p$, where $q=p^m$) is a nontrivial character. Thus, any other character on $\mathbb{F}_q$ has the form $\exp ax= \xi^{\text{tr}\;ax}$ for some $a\in\mathbb{F}_q$. Now we begin.

Given two linear codes $C_1,C_2\subset\mathbb{F}^n_q$, let $f:C_1\rightarrow C_2$be an $\mathbb{F}_q$-linear isomorphism that preserves Hamming weight. Assuming that $\mathrm{dim}_{\mathbb{F}_q}C_1=k=\mathrm{dim}_{\mathbb{F}_q}C_2$. Let $V$ be a $k$-dimensional vector space over $\mathbb{F}_q$ (a unified underlying module), then there exist linear embeddings $\Lambda=(\lambda_1,\lambda_2,\ldots,\lambda_n):V\rightarrow\mathbb{F}^n_q$ and $\Psi=(\psi_1,\psi_2,\ldots,\psi_n):V\rightarrow\mathbb{F}^n_q$ such that $\Lambda(V)=C_1$ and $\Psi(V)=C_2$, with $f\circ\Lambda=\Psi$. Since $f$ preserves Hamming weight, the following holds for every  $v\in V$, $$|\{i|\lambda_i(v)\neq 0\}|=|\{j|\psi_j(v)\neq 0\}|.$$
We use the equation mentioned earlier about the sum of character actions to write the above equation, for any $v\in V$, as
$$\underset{i=1}{\overset{n}{\sum}}\underset{a\in\mathbb{F}_q }{\sum} \xi^{\text{tr}\;a\lambda_i(v)}=\underset{j=1}{\overset{n}{\sum}}\underset{b\in\mathbb{F}_q }{\sum}\xi^{\text{tr}\;b\psi_j(v)},$$ giving rise to the following equation of characters$$\underset{i=1}{\overset{n}{\sum}}\underset{a\in\mathbb{F}_q }{\sum} \exp\;a\lambda_i=\underset{j=1}{\overset{n}{\sum}}\underset{b\in\mathbb{F}_q }{\sum}\exp\;b\psi_j.$$
Now, this is easily refined to \begin{equation}\label{here}\underset{i=1}{\overset{n}{\sum}}\underset{a\in\mathbb{F}^\times_q }{\sum} \exp\;a\lambda_i=\underset{j=1}{\overset{n}{\sum}}\underset{b\in\mathbb{F}^\times_q }{\sum}\exp\;b\psi_j,\end{equation} which is done by canceling $n$ times the function 1 on $V$ from both sides (the terms corresponding to $a=b=0$).

Since a group characters are linearly independent, it follows that, for $j=1$, $b=1$, there must be some $a_1\in\mathbb{F}^\times_q $ and $i_1$, such that $\exp\psi_1=\exp a_1\lambda_{i_1}$. Since the map from $V^\sharp$ (the $\mathbb{F}_q $-dual  of $V$) to $\widehat{V}$ given by $\lambda\mapsto \exp\lambda$ is an isomorphism, we get $\psi_1=a_1\lambda_{i_1}$.

Now, when multiplying the equation $\psi_1=a_1\lambda_{i_1}$ by $b\in\mathbb{F}^\times_q$, we get $b\psi_1=b\tau_1\lambda_{\sigma(1)}$, where $\tau_1$ is the automorphism $c\mapsto a_1c$ of $\mathbb{F}_q$, and $\sigma(1)=i_1$. Thus, $$\underset{b\in\mathbb{F}_q }{\sum}\exp\;b\psi_1=\underset{a\in\mathbb{F}_q }{\sum}\exp\;a_1\lambda_{\sigma(1)}.$$ Canceling this part from equation \ref{here}, the outer summation in the equation is reduced by 1. Proceeding in that way, we obtain a finite family of automorphisms $\tau_i$'s and a permutation $\sigma$ such that $\psi_i=\tau_i\lambda_{\sigma(i)}$.$\qquad\qquad\qquad\qquad\square$\\

To this end, a natural question arises: To what extent can this proof be generalized? Does it work for arbitrary rings? The answers of these questions are introduced in the coming sections.
\section{Codes over Finite Frobenius Rings\\ and Bi-modules}

\begin{par}A rich analysis of the extension problems for codes was carried out in many papers, by many authors. Among this study,  the deepest impact is, perhaps, accredited to the previous character theoretic proof of MacWilliams' theorem. This argument managed to survive in more general contexts, and appealing to it, much results began to appear.

One of the revealed facts was that, the extension problem is, so to speak, ``fastened'' to finite Frobenius rings, and to what these rings suggest for a general module for having similar properties. This section is devoted for this material. We start with the original definitions of T. Nakayama \cite{eniusean}, and then we follow \cite{Honold}, \cite{duality} to obtain (and use)  that useful characterization of finite Frobenius rings: A finite ring $R$ is Frobenius  $\Longleftrightarrow$ $_R\widehat{R}\cong {_RR} \Longleftrightarrow\widehat{R}_R\cong R_R$. As its weakest consequence, this characterization provided another answer to the question of Claasen and Goldbach \cite{Claasen} whether $_R\widehat{R}$ being cyclic implies its right counterpart for a finite ring (there is, however, an elementary proof of this).\end{par}

\begin{par} As T. Honold points out in \cite{Honold}, E. Lamprecht, in the 1950's, introduced the character module of a finite ring when he was working on Gaussian sums over finite rings  \cite{lmprcht1, lmprcht2, lmprcht3, lmprcht4}. However, it seems very clear that his papers are not well known among ring theorists. For example, the left-right symmetry question of Claasen and Goldbach, posed in 1992, has its elementary answer in $\mathrm{\cite[\S3.2, Lemma 1]{lmprcht2}}$  by Lamprecht (Proposition \ref{Lmprcht} below).  Moreover, Lamprecht  showed that a finite local ring has a generating character if and only if it is a local Frobenius ring \cite{lmprcht2}, and that a finite commutative ring has a generating character if  and only if it is Frobenius.

Unaware of these, finite rings with left generating character ($_R\widehat{R}$ is cyclic) were called \emph{right admissible} by Claasen and Goldbach \cite{Claasen}. They re-produced some of Lamprecht's results and determined more classes of left (right) admissible rings. More to say about expressions, the phrase \emph{generating character} is  due to Klemm \cite{Klemm}. Klemm also showed that the MacWilliams identities hold for finite commutative Frobenius rings \cite{Klemm23}. \end{par}
\begin{par}This prelude aimed to form, indirectly, a conception of how Frobenius rings were plunged into the world of coding theory. Lest digression, we conclude with a justified outline of the section. In \cite{duality}, the aforementioned chracterization of finite Frobenius rings is proved using Morita duality. Were we to follow this, we would have only stated (without proofs) the strong theorems to be used, violating our self-containedness policy. Instead, we shall first follow \cite{Honold} to prove that, for a finite ring $R$
\begin{align*}R\;\; \text{is Frobenius}&\Longleftrightarrow{_R(R/\mathrm{rad}R)}\cong\mathrm{soc}(_RR) \\ &\Longleftrightarrow (R/\mathrm{rad}R)_R\cong\mathrm{soc}(R_R) \Longleftrightarrow {_R\widehat{R}} \;\;\text{is cyclic}. \end{align*}\end{par}
Then we are back to the track of \cite{duality}.

\subsection{Frobenius and Quasi-Frobenius Rings}

\begin{par}Let $R$ be an artinian ring (not necessarily finite). By the Krull-Schmidt theorem, the left regular module $_RR$  has a  decomposition $$_RR=A_{11}\oplus\cdots\oplus A_{1k_1}\oplus\cdots\oplus A_{m1}\oplus\cdots\oplus A_{mk_m},$$ into indecomposable submodules (left ideals) $A_{ij}$, where we assume that $A_{ij}\cong A_{kl}$ if and only if $i=k$. Accordingly, there corresponds a decomposition of unity $1=e_{11}+\cdots+e_{1k_1}+\cdots+e_{m1}+\cdots+e_{mk_m}$, say. It is easily checked\footnote{Multiply the decomposition of unity by $e_{ij}$ from the left, and use the uniqueness of expressions in a direct sum.} that $\{e_{ij}\}_{i,j}$ are orthogonal idempotents and that $A_{ij}=Re_{ij}$. Further, these are primitive idempotents since the $Re_{ij}$'s are indecomposable.
Setting $e_i=e_{i1}$, write \begin{equation}\label{dec} _RR\cong \overset{m}{\underset{i=1}{\bigoplus}}k_i Re_i,\end{equation} where $k_i$ is the  multiplicity of the isomorphism type $Re_i$.\end{par}

\begin{par}Now,  recalling that artinian rings are semiperfect (Definition \ref{semiper}), the $e_{ij}$'s are also local idempotents by Proposition \ref{23.5}.   Then, by Proposition \ref{21.18}, each indecomposable $Re_i$ has a unique simple ``top quotient\footnote{In this we follow \cite{duality}.}''  $T(Re_i)=Re_i/\mathrm{rad} R \;e_i$.\end{par}

\begin{par}The parallel ``right'' argument can also  be established. From the decomposition $1=e_{11}+\cdots+e_{1k_1}+\cdots+e_{m1}+\cdots+e_{mk_m}$ of unity, we obtain $R_R=e_{11}R+\cdots+e_{1k_1}R+\cdots+e_{m1}R+\cdots+e_{mk_m}R$. This is also a direct sum\footnote{Without loss of generality, if $e_{11}r\in e_{12}R+\cdots+e_{1k_1}R+\cdots+e_{m1}R+\cdots+e_{mk_m}R $, then, once we multiply from the left by $e_{11}$, we get $e_{11}r=0$.}.
Besides, since for any two primitive idempotents $e$ and $f$, $Re\cong Rf$ if and only if $eR\cong fR$ (Proposition \ref{21.20}), we get \begin{equation}\label{rdec} R_R\cong \overset{m}{\underset{i=1}{\bigoplus}}k_i e_iR.\end{equation} The right counterparts of the modules $T(Re_i)$ are the modules $T(e_iR)=e_iR/e_i\mathrm{rad} R$. Further, $T(Re_i)\cong T(Re_j)$ if and only if $i=j$ by Proposition \ref{19.27}.
Now, with repeated use of Proposition \ref{radmod} (items (2),(3) and (6)), we get the following calculation.
\begin{equation}_R(R/\mathrm{rad} R)\cong\overset{m}{\underset{i=1}{\bigoplus}}k_i\;T(Re_i).\end{equation} \end{par}

However, there are more to know about these top quotients. By the uniqueness of a Krull-Schmidt decomposition, equations (\ref{wedec}) and (\ref{rdec}), together, imply that $n=m$, and, after some ordering, we may assume that $\mu_i=k_i$ and $T(Re_i)\cong T_i$ for $i=1,\ldots,n$. Similarly, $T(e_iR)\cong T'_i$. Thus,
\begin{equation}\label{now}_R(R/\mathrm{rad} R)\cong\overset{n}{\underset{i=1}{\bigoplus}}\mu_i\;T(Re_i)\;\;\text{and}\;\;
(R/\mathrm{rad} R)_R\cong\overset{n}{\underset{i=1}{\bigoplus}}\mu_i\;T(e_iR).\end{equation}

Following Nakayama's definitions \cite{eniusean}, the artinian ring $R$ is called \emph{quasi-Frobenius} (QF) if  there exists a permutation $\sigma$ of $\{1,\ldots,n\}$ such that, $$T(Re_i)\cong \mathrm{soc}(Re_{\sigma(i)})\quad\text{and}\quad \mathrm{soc}(e_iR)\cong T(e_{\sigma(i)}R),$$for every $i$. If, in addition, $\mu_{\sigma(i)}=\mu_i$, the ring is called \emph{Frobenius}. Note that in a QF ring, $Re_i$ contains a unique simple submodule (left ideal).\\

\begin{proposition}\label{frobrad}An artinian ring $R$ is Frobenius if and only if
\begin{equation}\label{F}_R(R/\mathrm{rad}R)\cong \mathrm{soc}(_RR)\;\;\text{and}\;\;(R/\mathrm{rad}R)_R\cong \mathrm{soc}(R_R)\end{equation}\end{proposition}

\begin{proof}We only have to prove the ``if'' part. For this, notice that $ \mathrm{soc}(_RR)\cong  \overset{n}{\underset{i=1}{\bigoplus}}\mu_i\;\mathrm{soc}(Re_i)$,\footnote{It is an easy exercise that the socle of a direct sum is the direct sum of socles of summands.} and that the composition length of $\mathrm{soc}(Re_i)$, for each $i$,  is greater than or equal to one. In view of (\ref{now}), when comparing the composition lengths, we get that the $\mathrm{soc}(Re_i)$'s are simple (as the composition length is 1 for each). Thus, we obtain a permutation $\sigma$ such that $\mathrm{soc}(Re_i)\cong T(Re_{\sigma(i)})$ and $\mu_i=\mu_{\sigma(i)}$. Similarly, for the right counterpart, we get that the $\mathrm{soc}(e_iR)$'s are simple. Yet,  we do not know whether the permutation $\sigma$ plays the desired role for the right setting.  We proceed as follows, denoting $J=\mathrm{rad}R$. \\

\textbf{Step 1.} We show that $\mathrm{soc}(_RR)=\mathrm{soc}(R_R)$. Fixing $i$, we know from the above that $\mathrm{soc}(_RR)$ contains a copy of the simple $R$-module $T(Re_i)\cong Re_i/Je_i\cong  \overline{R}\bar{e}_i$.  $e_i\cdot\overline{R}\bar{e}_i=\bar{e}_i\overline{R}\bar{e}_i\neq0$, since otherwise $e\in eRe\subset J$ contradicting $\overline{R}\bar{e}_i$ is nonzero. As isomorphic modules have the same annihilators, we find that
 $e_i\cdot \mathrm{soc}(_RR)\neq 0$.  Now, $R$ is artinian and, by Proposition \ref{socrad},
$\mathrm{soc}(_RR)=\mathrm{ann}_r(J)$,  thus $\mathrm{soc}(_RR)$ is an ideal  and $e_i\cdot\mathrm{soc}(_RR)$ is a right submodule of $e_iR$. Again, since $R$ is artinian, $e_i\cdot\mathrm{soc}(_RR)$ contains a simple right submodule (minimal right ideal), namely $\mathrm{soc}(e_iR)$, being the only simple submodule of $e_iR$. Thus, for every $i$, $\mathrm{soc}(e_iR)\subset e_i\cdot \mathrm{soc}(_RR)\subset \mathrm{soc}(_RR)$, giving  $\mathrm{soc}(R_R)\subset\mathrm{soc}(_RR) $. By symmetry, $\mathrm{soc}(_RR)\subset\mathrm{soc}(R_R)$.

\textbf{Step 2.} For a left (resp. right) $R$-module $A$, set $A^*_R=\mathrm{Hom}_R({_RA},{_RR})$ (resp. $_RA^*=\mathrm{Hom}_R({A_R},{R_R})$), called the (first) dual of $A$. There is a natural homomorphism $A\rightarrow A^{**}$ given by $a\mapsto (\varphi\mapsto \varphi(a))$. For a simple module, this map is a monomorphism\footnote{Let $_RA$ be simple. For any nonzero element $a\in A$, $A=Ra$, and the map $ra\mapsto r$ is a  homomorphism in $\mathrm{Hom}_R({_RA},{_RR})$ such that $\varphi(a)=1\neq0.$}. \\
Fixing $i$, we show here that $$(\mathrm{soc}(Re_i))^*\cong T(e_iR).$$
Let $\phi:e_iR\rightarrow (\mathrm{soc}(Re_i))^*$ be the right $R$-homomorphism sending $re_i$ to the left multiplication by $re_i$ on $\mathrm{soc}(Re_i)$. From Step 1, and Proposition \ref{socrad}, we have $$\phi(e_iJ)(\mathrm{soc}(Re_i))=(\mathrm{soc}(Re_i))J\subset \mathrm{soc}(_RR)\cdot J=\mathrm{soc}(R_R)\cdot J=0.$$
Consequently, $\phi$ induces a homomorphism $e_iR/e_iJ\rightarrow (\mathrm{soc}(Re_i))^*$. But $e_iR/e_iJ$ is simple, and hence we are done if we could show that $\phi$ is an epimorphism.\\

\textbf{Step 3.} Let $s\in \mathrm{soc}(Re_i)$ corresponds to $\bar{e}_{\sigma(i)}$ under the isomorphism $\mathrm{soc}(Re_i)\cong T(Re_{\sigma(i)})\cong \overline{R}\bar{e}_{\sigma(i)}$. Being  simple, $\mathrm{soc}(Re_i)=Rs$, besides, $s=e_{\sigma(i)}s$, for $e_{\sigma(i)}s$ will also correspond to $e_{\sigma(i)}\bar{e}_{\sigma(i)}=\bar{e}_{\sigma(i)}$. Now, consider any nonzero $\theta\in  (\mathrm{soc}(Re_i))^*$, then $$\theta(s)=\theta(e_{\sigma(i)}s)=e_{\sigma(i)}\theta(s)\in \theta(\mathrm{soc}(Re_i))\subset \mathrm{soc}(_RR).$$ By Step 1, we have
$$s\in e_{\sigma(i)}R\cap \mathrm{soc}(Re_i)\subset e_{\sigma(i)}R\cap \mathrm{soc}(_RR)=e_{\sigma(i)}R\cap \mathrm{soc}(R_R)\subset \mathrm{soc}(e_{\sigma(i)}R),$$
and similarly $\theta(s)\in \mathrm{soc}(e_{\sigma(i)}R)$. Since $\mathrm{soc}(e_{\sigma(i)}R)$ is simple, we have $sR=\mathrm{soc}(e_{\sigma(i)}R)=\theta(s)R$. Thus, $\theta(s)=sr$ for some $r\in R$. Finally, as $Rs=\mathrm{soc}(Re_i)\subset Re_i$, we have, for any $x\in R$,  $$\theta(xs)=x\theta(s)=xsr=(xs)(e_ir).$$ This shows that $\theta$ is just a right multiplication by $e_ir$. Thus, $\phi$ is an isomorphism.\\
\textbf{Step 4.} The previous steps resulted in that $(\mathrm{soc}(Re_i))^*\cong T(e_iR)$, similarly, of course, we have  the right analogues, $(\mathrm{soc}(e_iR))^*\cong T(Re_i)$, for each $i$. Taking the duals for the isomorphism $\mathrm{soc}(Re_i)\cong T(Re_{\sigma(i)})$, we obtain $T(e_iR)=(\mathrm{soc}(Re_i))^*\cong (T(Re_{\sigma(i)}))^*\cong (\mathrm{soc}(e_{\sigma(i)}R))^{**}$. Furthermore, this shows that $(\mathrm{soc}(e_{\sigma(i)}R))^{**}$ is simple, and therefore the natural map  $$\mathrm{soc}(e_{\sigma(i)}R)\rightarrow (\mathrm{soc}(e_{\sigma(i)}R))^{**}$$ is an isomorphism.  Thus $T(e_iR)\cong\mathrm{soc}(e_{\sigma(i)}R)$.

\end{proof}

To keep good balances inside the text, many interesting characterizations and examples of QF and Frobenius rings had to be excluded, or just stated without proof. In \cite{eniusean}, Nakayama also proved the famous double annihilator\footnote{It should be mentioned that this idea is due to M. Hall \cite{MHall}, when he discovered it for semisimple algebras. Later, it was very crucial in  the theory of Frobenius and quasi-Frobenius rings.} characterization of QF rings, which, in many texts, is used as a definition of QF rings.

\begin{theorem}\label{doubleann}$\mathrm{\cite[\S4, Theorem \;6]{eniusean}}$.  A ring $R$ is QF if and only if $R$ is artinian and satisfies $$\mathrm{ann}_l(\mathrm{ann}_r(\mathfrak{A}))=\mathfrak{A}, \;\;\text{and}\;\;\mathrm{ann}_r(\mathrm{ann}_l(A))=A,$$ for every left ideal $\mathfrak{A}\subset R$ and every right ideal $A\subset R$.\end{theorem}

(This appropriate notation for annihilators will be  used  more generally for subsets of left and right modules).
The proof of the next theorem is a very beautiful, thoughtfully constructed one; however, lest a degression would occur, it is just the statement that will be displayed.
\begin{theorem}$\mathrm{\cite[Theorem \;58.6]{curtis}}$. An artinian ring $R$ is QF if and only if $_RR$ is injective.\end{theorem}

Let $R$ be a finite ring. Let $S\subset R$ , $T\subset\widehat{R}$ be any subsets. We will use the following $*$-notation.  $S^*=\{\rho\in \widehat{R}\;|\; \rho(S)=1\}$, $T^*=\underset{\rho\in T}{\bigcap}\mathrm{Ker}\rho$.

Note that if  $T$ is a subgroup of $(\widehat{R},\cdot)$, then by Definition \ref{annchar} and the natural isomorphism in the proof of Proposition \ref{char},(2):
\begin{align*}
(\widehat{R}:T)&=\{\widehat{r}\in\widehat{\widehat{R}}\;|\;\widehat{r}(\rho)=1\;\;\text{for all}\;\; \rho\in T\}\\
                        &=\{\widehat{r}\in\widehat{\widehat{R}}\;|\; \rho(r)=1\;\;\text{for all}\;\; \rho\in T\}\\
                        &=\{\widehat{r}\in\widehat{\widehat{R}}\;|\; r\in T^*\}.
\end{align*}
It is then clear that, $|T^*|=|(\widehat{R}:T)|=|\widehat{R}|/|T|$.

Besides, One can see that, if $S\leq {_RR}$ is a left submodule, then $S^*$ is a right submodule of $\widehat{R}_R$; also, if $T\leq\widehat{R}_R$ is a right submodule, then $T^*$ is a left submodule of $_RR$. The maps $S\mapsto S^*, T\mapsto T^*$ form mutually inverse lattice anti-isomorphisms between the lattice of left ideals in $R$ and the lattice of right submodules of $\widehat{R}_R$. If $T\leq \widehat{R}_R$ is a right submodule, then
\begin{align*} \mathrm{ann}_l(T)&=\{r\in R\;|\;r\rho=1 \;\;\text{for all}\;\; \rho\in T\}\\
                                                    &=\{r\in R\;|\;\rho(Rr)=1\;\;\text{for all}\;\; \rho\in T\}\\
                                                    &=\{r\in R\;|\;\pi(r)=1 \;\;\text{for all}\;\; \pi\in TR=T\}=T^*.\end{align*}

\begin{proposition}\label{Lmprcht}$\mathrm{\cite[\S3.2, Lemma 1]{lmprcht2}}$ Every character $\rho\in\widehat{R}$ satisfies that $ | \mathrm{ann}_l(\rho)|=| \mathrm{ann}_r(\rho)|$ and $|R\rho|=|\rho R|$. In particular, $\rho$ is a left generating character if and only if $\rho$ is a right generating character.

\end{proposition}

\begin{proof}
By the remarks above, we know that $|\rho R|\cdot|(\rho R)^*|=|\widehat{R}|=|R|$, and $\mathrm{ann}_l(\rho)=\mathrm{ann}_l(\rho R)=(\rho R)^*$, whence $|\mathrm{ann}_l(\rho)|=|R|/|\rho R|$. But, from another way, $R/\mathrm{ann}_l(\rho)\cong R\rho$ as left $R$-modules, and thus $|\mathrm{ann}_l(\rho)|=|R|/|R\rho|$. Hence, $|R\rho|=|\rho R|$, and $|\mathrm{ann}_l(\rho)|=|\mathrm{ann}_r(\rho)|$, that's because $R/\mathrm{ann}_r(\rho)\cong \rho R$ as right $R$-modules.

\end{proof}
\textbf{Remark:} As $|R|=|\widehat{R}|$, it follows that $\widehat{R}=R\rho$ if and only if $\mathrm{ann}_l(\rho)=0$. But,
$\mathrm{ann}_l(\rho)=\mathrm{ann}_l(\rho R)=(\rho R)^*=\{r\in R\;|\; \rho(Rr)=1\}$, hence $\widehat{R}=R\rho$ if and only if $\mathrm{Ker}\rho$ contains no nonzero left ideals.\\

In the next part we prove that, for a finite ring $R$, $_R(R/\mathrm{rad}R)\cong \mathrm{soc}(_RR)$ if and only if $(R/\mathrm{rad}R)_R\cong \mathrm{soc}(R_R)$. In \cite{supp}, Nakayama showed this is true for finite dimensional algebras over fields, he also provided a counter-example to show that it does not hold for all artinian rings. W. Xue \cite{xue} proved the assertion for finite local rings. The question whether for finite rings $_R(R/\mathrm{rad}R)\cong \mathrm{soc}(_RR)$ implies its right counterpart was left open in \cite{Honnech}, and was asked again in \cite{gref}. The proof was finally given in \cite{Honold}, which we mainly follow in Theorem \ref{hon}, except when proving  (2)$\Longrightarrow$(1).

\begin{theorem}\label{hon}$\mathrm{\cite[\S4, Theorem 1]{Honold}}$. The following statements are equivalent for a finite ring $R$.
\begin{enumerate}
\item[$(1)$]$_R\widehat{R}$ is cyclic.
\item[$(2)$]$\mathrm{soc}(_RR)$  is cyclic.
\item[$(3)$]$_R(R/\mathrm{rad}R)\cong \mathrm{soc}(_RR)$.
\item[$(4)$]$|I|\cdot|\mathrm{ann}_r(I)|=|R|$ for every left ideal $I$ in $R$.
\end{enumerate}
\end{theorem}

\begin{proof}$(2)\Longleftrightarrow(3)$. As constructed in the beginning of this section, $_RR$ has a decomposition $_RR=Re_{11}\oplus\cdots\oplus Re_{1\mu_1}\oplus\cdots\oplus Re_{n1}\oplus\cdots\oplus Re_{n\mu_n}$. Each of these indecomposable summands contains at least one simple submodule (since all are finite). Thus, $\mathrm{soc}(_RR)$ has composition length greater than or equal to that of $_R(R/\mathrm{rad}R)$ (see equation \ref{now}). Now, $\mathrm{soc}(_RR)$ is cyclic if and only if  it is an epimorphic image of $_RR$, if and only if it is an epimorphic image of $_R(R/\mathrm{rad}R)$, which, by the previous, is equivalent to $\mathrm{soc}(_RR)\cong {_R(R/\mathrm{rad}R)}$.\\

$(1)\Longrightarrow(4)$. Assume that $\widehat{R}=R\rho$. By the remark after Proposition \ref{Lmprcht}, this is equivalent to that $\mathrm{Ker}\rho$ contains no nonzero left ideals. Consequently, for a left ideal $I$ and any $r\in R$, $Ir=0$ if and only if $(r\rho)(I)=\rho(Ia)=1$, that is, $I^*=\mathrm{ann}_r(I) \rho$. Hence, $|\mathrm{ann}_r(I)|=|\mathrm{ann}_r(I)\rho|=|I^*|=|R|/|I|$.\\

$(4)\Longrightarrow(2)$. The ring $R$ is artinian (being finite), and hence, by Proposition \ref{socrad}, $\mathrm{soc}(_RR)=\mathrm{ann}_r(\mathrm{rad}R)$. For the multiplicities, suppose that $\mathrm{soc}(_RR)\cong \overset{n}{\underset{i=1}{\bigoplus}}s_i T_i$. By the hypothesis, $|\mathrm{soc}(_RR)|=|R|/|\mathrm{rad}R|=|R/\mathrm{rad}R|$. If $I$ is a left ideal of $R$, consider the map $\theta:R_R\rightarrow \mathrm{Hom}_R(_RI,_RR)$ defined by $a\mapsto \theta_a$, where $\theta_a:x\mapsto xa$ for $x\in I$. Clearly $\mathrm{Ker}\theta= \mathrm{ann}_r(I)$, hence $|\mathrm{Im}\theta|=|R|/|\mathrm{ann}_r(I)|=|I|$. Consequently, $|I|\leq |\mathrm{Hom}_R(_RI,_RR)|$. In particular, for the homogeneous component $S_i=n_iT_i$ of the socle, we get $|\mathbb{M}_{\mu_i\times s_i}(\mathbb{F}_{q_i})|=|S_i|\leq|\mathrm{Hom}_R(_RS_i,_RR)|=|\mathrm{End}_R(_RS_i)|=|\mathbb{M}_{s_i}(\mathbb{F}_{q_i})|$. Hence, either $s_i=0$ or $\mu_i\leq s_i $. Since $T_i$ is a simple left $R$-module, there exists a (maximal) left ideal $I$, such that $T_i\cong R/I$. Now, notice that each element $a\in \mathrm{ann}_r(I)$   determines exactly one homomorphism  in $\mathrm{Hom}_R(R/I,_RR)$, defined by $x+I\mapsto xa$.

 Conversely, for any $\lambda\in\mathrm{Hom}_R(R/I,_RR)$, $\lambda(1)\in \mathrm{ann}_r(I)$.
This makes $|\mathrm{Hom}_R(_RT_i,_RR)|=|\mathrm{Hom}_R(R/I,_RR)|=|\mathrm{ann}_r(I)|=|R|/|I|=|T_i|>1$. This discovered nonzero homomorphism $_RT_i\rightarrow _RR$ must be a monomorphism since $T_i$ is simple, and hence $R$ contains a minimal left ideal $J\cong T_i$. This implies that $s_i\neq 0$\footnote{Let $\underset{i\in I}{\bigoplus}N_i$ be a direct sum of a family of simple left $R$-modules. It is easy to prove that if $M$ is a simple left $R$-module such that $M$ can be embedded into $\underset{i\in I}{\bigoplus}N_i$, then $M\cong N_j$ for some $j$.}, and $\mu_i\leq s_i$. This result, together with that $|\mathrm{soc}(_RR)|=|R/\mathrm{rad}R|$, give that $\mu_i=s_i$ for every $i$. Whence, we have $\mathrm{soc}(_RR)\cong {_R(R/\mathrm{rad}R)}$.\\

$(2)\Longrightarrow(1)$. By Proposition \ref{Rhat}, $\mathrm{soc}(_RR)$ is cyclic if and only if $_RR$ can be embedded into $_R\widehat{R}$, which is equivalent to $_R\widehat{R}\cong{_RR}$, and hence implies $_R\widehat{R}$ is cyclic.

\end{proof}
Notice that the first condition in the theorem is left-right symmetric by Proposition \ref{Lmprcht}, hence, in view of Proposition \ref{frobrad} and Theorem \ref{hon}, the following corollary is a direct consequence.
\begin{corollary}\label{honcor} A finite ring $R$ is Frobenius if and only if, as left modules, $_R(R/\mathrm{rad}R)\cong \mathrm{soc}(_RR)$, if and only if $_R\widehat{R}$ is cyclic.\end{corollary}

In \cite{duality}, this characterization of finite Frobenius rings was used to prove the following extension theorem for linear codes over finite rings. It stands as a good example of how the proof of Theorem \ref{mac} produces more generalities. Anyway, the theorem will be stated without proof, as we shall display the proof only once in its more general setting of Frobenius bimodules.

\begin{theorem}\label{du1}$\mathrm{\cite[Theorem\; 6.3]{duality}}$. Let $R$ be a finite Frobenius ring. Then $R$  has the extension property with respect to Hamming weight.\end{theorem}

In \cite{gref}, Greferath and Schmidt  gave another proof of this result. It is noteworthy  that, also in \cite{duality}, J. Wood proved the following  partial converse of the above theorem.

\begin{theorem}\label{du2}$\mathrm{\cite[Theorem\; 6.4]{duality}}$. Suppose $R$ is a finite commutative ring, and that $R$ has the extension property with respect to Hamming weight. Then $R$ is Frobenius.\end{theorem}

However, it remained a question whether the full converse is true, that is, whether a finite ring satisfying the property with respect to Hamming weight is necessarily Frobenius. In 2000, in \cite{gref}, Greferath and Schmidt  gave an example of a QF ring that did not satisfy the property. Thus, the full converse was likely true. In \cite{r1}, Dinh and L\'{o}pez-Permouth showed that the converse holds within the class of finite local rings and homogeneous  semilocal rings. Also, in their paper \cite{r2}, they attempted to characterize generally the modules satisfying the extension property. The general context of arbitrary modules will be exposed in the next parts.

\subsection{Frobenius Bi-modules}

In the previous work we characterized Finite Frobenius rings to be those rings $R$ satisfying $_RR\cong{_R\widehat{R}}$ and the proved-equivalent condition $R_R\cong \widehat{R}_R$. This suggests some generalization, namely, Frobenius bimodules. Though it turns out to be true, we do not pre-assume in the definition that the left and right conditions imply each other. Here we're mainly following \cite{r7} and \cite{2004}.

\begin{definition}(Frobenius Bi-module). Let $_RA_R$ be a finite (R,R)-bimodule\footnote{The hyphen in the word `bi-module' will often be dropped.} over the finite ring $R$. We say that $A$ is a \emph{Frobenius bi-module} if $_RA\cong{_R\widehat{R}}$ and $A_R\cong \widehat{R}_R$. \end{definition}
Clearly, the character bimodule $_R\widehat{R}_R$ is Frobenius.  Besides, it is important to see that a Frobenius bimodule need not be isomorphic, as a bimodule,  to $_R\widehat{R}_R$.

\begin{lemma}\label{fb1.2.4} Let $_RA_R$ be  a Frobenius bimodule. Then
\begin{center}$_R\widehat{A}\cong {_RR}\;\;$   and  $\;\; \widehat{A}_R\cong R_R.$\end{center}\end{lemma}
\begin{proof}Immediate from the definition of a Frobenius bimodule,  upon applying the character functor.\end{proof}

In view of the previous lemma, we define generating characters of  Frobenius bimodules.
\begin{definition}(Generating Character). We say that a character $\varrho\in\widehat{A}$ is a \emph{left generating character} of $A$ if
 \begin{center}$\bullet\varrho:_RR\rightarrow{_R\widehat{A}},\;\;$ defined by $\;r\mapsto r\varrho$, \end{center} is an isomorphism. \\
 Similarly, $\varrho$ is a \emph{ right generating character}  of $A$ if
  \begin{center}$\varrho\bullet:R_R\rightarrow\widehat{A}_R,\;\;$ defined by $\;r\mapsto \varrho r$, \end{center} is an isomorphism. \end{definition}

Notice  that, by definition, $_RA_R$ is Frobenius if and only if it admits left and right generating characters.

\begin{lemma}\label{fb1.2.6} Let $_RA_R$ be a Frobenius bimodule. Let $\varrho\in\widehat{A}$.  If $\mathrm{Ker}\varrho$ contains no nonzero left (respectively, right) $R$ submodules of $A$, then $\varrho$ is a left (respectively, right) generating character of $A$.  \end{lemma}

\begin{proof}Suppose that $\varrho$ is not a left generating character, then $\bullet\varrho$ is not an isomorphism. Since $|R|=|\widehat{A}|$, this is equivalent to that $\bullet\varrho$ is not injective. Let $r\in \mathrm{Ker}(\bullet\varrho), r\neq0$. Then $0=(r\varrho)(a)=\varrho(ar)$ for every $a\in A_R$. Therefore, $Ar\subset \mathrm{Ker}\varrho$. Since $A_R\cong \widehat{R}_R$, and $r\neq0$, it follows from the right analogue of Lemma \ref{111111} that $Ar\neq0$. Thus, $Ar$ is a nonzero left $R$-module lying inside $\mathrm{Ker}\varrho$.
\end{proof}

\begin{lemma}\label{fb1.2.7} If $\varrho$ is a left (respectively, right) generating character of $A$, then $\mathrm{Ker}\varrho$ contains no nonzero right (respectively, left) $R$-submodules of $A$.
\end{lemma}
\begin{proof}Assume that $B_R\leq A_R$ is a right submodule such that $B\subset \mathrm{Ker}\varrho$. Let $\varpi=r\varrho\in \widehat{A}$ be any character, where $r\in R$. For any $b\in B$, $\varpi(b)=(r\varrho)(b)=\varrho(br)=0$. Thus, $B\subset\mathrm{Ker}\varpi$ for every $\varpi\in \widehat{A}$, and hence $B=0$ by Proposition \ref{(g:h)}. Similarly, the kernel of a right generating character contains no nonzero left submodule of $A$.\end{proof}

The last two lemmas give rise to the following corollary.

\begin{corollary}\label{genkerempty}If $_RA_R$ is a Frobenius bimodule, then a character $\varrho\in \widehat{A}$ is a left generating character if and only if $\varrho$ is a right generating character, if and only if $\mathrm{Ker}\varrho$ contains no nonzero left submodules of $A$.\end{corollary}

The previous corollary prompts us to extend the definition of a generating character to comply with one-sided modules \cite{woodappl}.

\begin{definition}\label{exactly}Let $A$ be a finite left (respectively, right) $R$-module. A character $\varpi$ of $A$ is a generating character of $A$ if $\mathrm{Ker}\varrho$ contains no nonzero left (respectively, right) submodule of $A$.\end{definition}

The next three results first appeared in \cite{woodappl}, p.228.
\begin{lemma}\label{lemma1.2.9}Let $A$ be a finite left $R$-module, and let $B\subset A$ be a submodule. If $A$ admits a generating character, then also $B$ admits a generating character.\end{lemma}

\begin{proof}When restricting the generating character of $A$ to $B$, we obtain a character od $B$, besides, this is a generating character of $B$ since its kernel cannot contain a submodule  $C\leq B\leq A$.\end{proof}

\begin{lemma}\label{lemma1.2.10}Let $R$ be a finite ring with unity. Define a group homomorphism $\varrho:\widehat{R}\rightarrow \mathbb{Q/Z}$ by $\varrho(\varpi)=\varpi(1),\;\;\varpi\in \widehat{R}$. Then $\varrho$ is a generating character of $\widehat{R}$. \end{lemma}

\begin{proof} Suppose that for some  character $\varpi_0\in \widehat{R}$, $R\varpi_0\subset\mathrm{Ker}\varrho$. Then, $0=\varrho(r\varpi_0)=(r\varpi_0)(1)=\varpi_0(r)$ for all $r\in R$, hence $\varpi_0=0$. Thus $\mathrm{Ker}\varrho$ contains no nonzero left submodule of $\widehat{R}$.
\end{proof}

The following characterizations of modules with generating character will become very crucial when proving the extension theorems.
\begin{proposition}\label{prop1.2.11}Let $A$ be a finite left $R$-module. Then $A$ admits a left generating character if and only if $A$ can be embedded into $_R\widehat{R}$.\end{proposition}

\begin{proof}If $A$ can be embedded into $_R\widehat{R}$, then $A$ admits a generating character by Lemmas \ref{lemma1.2.9} and \ref{lemma1.2.10}.\\
Conversely, let $\varrho$ be a generating character of $A$. Define $f:A\rightarrow\widehat{R}$ such that for $a\in A$, $f(a)(r)= \varrho(ra)$, $r\in R$. Then $f$ is an $R$-monomorphism. Indeed,$f(a)\in \widehat{R}$, and upon recalling the left$R$-module structure of $\widehat{R}$, it is easily seen that $f$ is an $R$-module homomorphism. Moreover, if $a\in \mathrm{Ker}f$, then $0=f(a)(r)=\varrho(ra)$ for all $r\in R$. Thus $Ra\subset \mathrm{Ker}\varrho$, and since $\varrho$ is a generating character, it must be that $Ra=0$, whence $a=0$.  \end{proof}

By Proposition \ref{Rhat}, the next corollary becomes direct.
\begin{corollary}$A$ admits a left generating character if and only if $\mathrm{soc}(A)$ is cyclic.
\end{corollary}
Though it may seem redundant, it may be good to re-display the terms and notions pertaining to the extension problem, but this time in the more general context of arbitrary finite modules over a finite ring with unity.
Let $R$ be a finite ring with unity, and $_RA$ be a finite left $R$-module. As we used to do in chapter 1, denote by $\mathcal{U}$ the group of units of $R$. From now on, the convention for left (respectively, right) $R$-homomorphisms will be that inputs are to the left (respectively, right).

\begin{definition}(Monomial Transformation). A monomial transformation of $A^n$ is an $R$-automorphism $T$ of $A^n$ such that, for some permutation $\sigma$ of $\{1,\ldots,n\}$ and automorphisms $\tau_1,\ldots,\tau_n\in \mathrm{Aut}_R(A)$, $T$ has the special component-wise action $$(a_1,\ldots,a_n)T=(a_{\sigma(1)}\tau_1,\ldots,a_{\sigma(n)}\tau_n),$$ where $(a_1,\ldots,a_n)\in A^n$.\\
If $\tau_1,\ldots,\tau_n\in G\leq \mathrm{Aut}_R(A)$ (i.e. $G$ is a subgroup of $\mathrm{Aut}_R(A)$), we shall say then that $T$ is a $G$-monomial transformation.  \end{definition}

We define a weight on $A$ just as we defined it before for a finite field.
\begin{definition}\label{weight}(Weight). A weight $w$ on $A$ is a function $w:A\rightarrow \mathbb{Q}$ with $w(0)=0$. \end{definition}
This is naturally extended to $A^n$ by defining$$w(a_1,\ldots,a_n)=\underset{i=1}{\overset{n}{\sum}}w(a_i).$$
The Hamming weight is denoted by $\mathrm{wt}$.
\begin{definition}(Symmetry Groups). Given a weight $w$ on $A$, define the left and right symmetry groups of $w$ as follows.\begin{align*}G_{\text{lt}}&:=\{u\in \mathcal{U}: w(ua)=w(a),\;\text{for all}\;\;a\in A\},\\
G_{\text{rt}}&:=\{\tau\in \mathrm{Aut}_R(A): w(a\tau)=w(a),\;\text{for all}\;\;a\in A\}.\end{align*}If $w$ has $G_{\text{lt}}=\mathcal{U}$ and $G_{\text{rt}}=\mathrm{Aut}_R(A)$, we say that $w$ is maximally symmetric.\end{definition}

Note that, for instance, the Hamming weight is maximally symmetric.

\begin{definition}(Monomial Equivalence). Let $_RA$ be a left $R$-module serving as an alphabet. Consider a weight $w$ on $A$ with symmetry groups $G_{\text{rt}}$ and $G_{\text{lt}}$ . Two codes $C_1,C_2\subset A^n$ are said to be \emph{monomially equivalent} if there exists a $G_{\text{rt}}$-monomial transformation $T$ of $A^n$ with $C_1T=C_2$.\end{definition}

Notice that whenever $C_1,C_2$ and $T$ are as in the definition above, then $T|_{C_1}$ is an isomorphism $C_1\rightarrow C_2$ preserving $w$.

\begin{definition}(Extension Property). The alphabet $A$ has the \emph{extension property} (EP) with respect to the weight $w$ if, for any $n$, for any two codes $C_1,C_2\subset A^n$, any weight preserving $R$-isomorphism $f:C_1\rightarrow C_2$ extends to a $G_{\text{rt}}$-monomial transformation $T$ of $A^n$.   \end{definition}

Now, our objective will be studying necessary and sufficient conditions for an alphabet $_RA$ to satisfy the EP with respect to Hamming weight, and later with respect to symmetrized weight compositions.

\begin{lemma}\label{bass}$\mathrm{\cite[H.\; Bass,\;Lemma\;6.4]{bass}}$. Let $R$ be a semilocal ring. If $b\in R$ and $\mathfrak{A}$ is a left ideal such that $Rb+\mathfrak{A}=R$, then $b+\mathfrak{A}$ contains a unit.\end{lemma}
\begin{proof}By Proposition \ref{rad andsim}, $x\in R$ is a unit $\Longleftrightarrow$ $\bar{x}\in R/\mathrm{rad}R$ is a unit, so, we may assume $R$ is semisimple. Then, by the Wedderburn theorem, we can pass to the case $R=\mathbb{M}_n(D)$ where $D$ is a division ring. In that case, $R$ may be viewed as the endomorphism ring $\mathrm{End}_D(V)$ of the $D$-vector space $V=D^n$. Now, it is known that $\mathfrak{A}=Re$, where $e$ is an idempotent \cite[Theorem IX.3.7]{hungr}. Consider the subspace $W=\mathrm{Ker}e$. Using the decompositions $R=Re\oplus R(1-e)\;\text{and}\; V=\mathrm{Im}e\oplus \mathrm{Ker}e$, one can check that $\mathfrak{A}$ is exactly the set of endomorphisms in $R=\mathrm{End}_D(V)$ annihilating $W$. The assumption that $Rb+\mathfrak{A}=R$ implies that $W\cap \mathrm{Ker}b=0$. Now, writing $V=W\oplus W'=bW\oplus U$ and noting that $W\cong bW$, we can construct an automorphism $u$ (a unit of $R$) such that $u|_W=b|_W$ and $u(W')=U$. Finally, since $a=u-b$ annihilates $W$, we get $a\in \mathfrak{A}$.\end{proof}
A main step in the generalized character theoretic proofs of the extension theorem is defining a partial ordering on an $R$-module. The previous lemma of Bass is used in proving the following \cite[\S 5]{duality}.

\begin{lemma}\label{beneath}Let $_RM$ be a left $R$-module. If $x,y\in M$ satisfy $y=ax$ and $x=by$ for some $a,b\in R$, then $x=uy$ for some unit $u\in R$.\end{lemma}

\begin{proof}First, $R=Rba+R(1-ba)$ gives $R=Ra+R(1-ba)$. By Lemma \ref{bass}, $a+R(1-ba)$ contains a unit. Thus there is $r\in R$ such that $u=a+r(1-ba)$ is a unit. Now, $ux=ax+r(1-ba)x=ax=y$ (use that $(1-ba)x=0$).
\end{proof}

We can define an equivalence relation $\approx$ on $M$ by $x\approx y$ if $x=uy$ for some unit $u\in R$. Now, on the $\approx$-equivalence classes of $M$, we can define $y\preceq x$ if $y=ax$ for some $a\in R$. This relation is well-defined, reflexive, and transitive. Besides, it is anti-symmetric by Lemma \ref{beneath}. Thus we have the following.

\begin{corollary}The relation $\preceq$  is a partial ordering on the $\approx$-equivalence classes of a left $R$-module $M$.\end{corollary}

The next theorem is a generalization of Theorem \ref{du1}. It was first proved by Greferath, Nechaev and Wisbauer in \cite{2004} (2004). However, we will display Wood's proof appearing in \cite{r7}, which is similar to that of the main theorem in \cite{duality}.

\begin{theorem}\label{FBI} Let $R$ be a finite ring, and $_RA_R$ be a Frobenius bimodule. Then $A$ has the extension property with respect to Hamming weight.
\end{theorem}

\begin{proof}Let $C_1,C_2\subset A^n$ be two linear codes of length $n$, and let $f:C_1\rightarrow C_2$ be an $R$-isomorphism preserving Hamming weight. Let $M$ be a common underlying $R$-module of $C_1$ and $C_2$. Consider the embeddings $\Lambda=(\lambda_1,\ldots,\lambda_n):M\rightarrow A^n$ and $\Psi=(\psi_1,\ldots,\psi_n):M\rightarrow A^n$ whose images are $C_1$ and $C_2$ respectively, where $\lambda_i,\psi_i\in \mathrm{Hom}_R(M,A)$. By the assumption, it follows that for every $m\in M$ $$|\{i\;|\;m\lambda_i\neq0\}|=|\{j\;|\;m\psi_j\neq0\}|.$$
By Proposition \ref{char}, this can be rewritten as,\begin{equation}\label{iknow}\overset{n}{\underset{i=1}{\sum}}\underset{\pi\in \widehat{A}}{\sum}\;\pi(m\lambda_i)=\overset{n}{\underset{j=1}{\sum}}\underset{\theta\in \widehat{A}}{\sum}\;\theta(m\psi_j),\quad m\in M.\end{equation}
Let $\rho$ be a left generating character of $A$ (in multiplicative form). Thus any character of $A$ is given by $^r\rho$ for some $r\in R$. We can now rewrite equation (\ref{iknow}) as follows, $$\overset{n}{\underset{i=1}{\sum}}\underset{r\in R}{\sum}{\;^r \rho}(m\lambda_i)=\overset{n}{\underset{j=1}{\sum}}\underset{s\in R}{\sum}{\;^s\rho}(m\psi_j),\quad m\in M,$$
which is, \begin{equation}\label{iknow2}
          \overset{n}{\underset{i=1}{\sum}}\underset{r\in R}{\sum}\; \rho(m\lambda_ir)=\overset{n}{\underset{j=1}{\sum}}\underset{s\in R}{\sum}\;\rho(m\psi_j s),\quad m\in M,\end{equation}
where it is understood that $\mathrm{Hom}_R(M,A)$ is given the standard right $R$-module structure. Now it is time to use the pre-described partial ordering $\preceq$, defined on the right $\mathcal{U}$-orbits in $\mathrm{Hom}_R(M,A)$ by $\lambda\preceq\psi$ if $\lambda=\psi r$ for some $r\in R$.\\
Among the list of  elements $\lambda_1,\ldots,\lambda_n;\psi_1,\ldots,\psi_n$, choose one whose orbit is maximal with respect to $\preceq$. Without loss of generality, let that element be $\lambda_1$. Consider that term, on the left side of equation (\ref{iknow2}), with $r=1, i=1$. By the linear independence of characters of $M$, since the equation holds for all $m\in M$, there exists $j=\sigma(1)$ and an element $s\in R$ such that $\rho(m\lambda_1)=\rho(m\psi_js)$ for all $m\in M$. Thus $\rho(m\lambda_1-m\psi_js)=1$  for all $m\in M$ and $\mathrm{Im}(\lambda_1-\psi_js)\subset\mathrm{Ker}\rho$. Fortunately, $\mathrm{Im}(\lambda_1-\psi_js)$ is a left $R$-submodule of $A$, and hence, by Corollary \ref{genkerempty}, $\mathrm{Im}(\lambda_1-\psi_js)=0$. This shows that $\lambda_1=\psi_js$, and thus $\lambda_1\preceq \psi_j$. The maximality assumption on $\lambda_1$ implies that $\psi_j=\lambda_1u_1$ for some unit $u_1\in\mathcal{U} \subset R$. Thus, $$\underset{s\in R}{\sum}\;\rho(m\psi_j s)=\underset{s\in R}{\sum}\;\rho(m\lambda_1u_1 s)=\underset{r\in R}{\sum}\; \rho(m\lambda_1r),\quad m\in M.$$
This allows to reduce the outer summation in equation (\ref{iknow2}) by one. Proceeding this way, we obtain a permutation $\sigma$ of $\{1,\ldots,n\}$, and units $u_1,\ldots,u_n\in \mathcal{U}$ with $\lambda_iu_i=\psi_{\sigma(i)}$. This provides the desired monomial transformation.
 \end{proof}

\section{One-Sided Modules and MacWilliams' Theorem}

In this part we begin by introducing a result of Dinh and L\'{o}pez-Permouth, proving that a left module satisfies the extension property with respect to Hamming weight for codes of length 1 if and only if the module is pseudo-injective \cite[Proposition 3.2]{r1}.

\begin{lemma}\label{psss}Let $_RA$ be a finite left $R$-module, and let $B\subset A$ be a submodule. If $f: B\rightarrow A$ is a monomorphism that extends to a homomorphism $\bar{f}:A\rightarrow A$, then $f$ can be extended to an automorphism of $A$.\end{lemma}

\begin{proof}\begin{par}Case1: $\mathrm{soc}B=\mathrm{soc}A$. In this case $\mathrm{Ker} \bar{f}\cap\mathrm{soc}A=\mathrm{Ker} \bar{f}\cap\mathrm{soc}B= \mathrm{Ker} f\cap\mathrm{soc}B=0$, which implies, since $A$ is finite, that $\mathrm{Ker} \bar{f}=0$. Thus $\bar{f}$ is the desired automorphism.\end{par}
\begin{par}Case2: $\mathrm{soc}B\neq\mathrm{soc}A$. Then there are submodules $M,N\subset \mathrm{soc}A$ such that $\mathrm{soc}A=\mathrm{soc}B\oplus M=f(\mathrm{soc}B)\oplus N$, whence $M\cong N$. Again, since all are finite modules, the previous gives $B\cap M=0$. Extend $f:B\rightarrow A$ to the injective map $h:B\oplus M\rightarrow A$ defined by $(b+m)h=bf+mg$, for all $b\in B$ and $m\in M$. But $\mathrm{soc}(B\oplus M)=\mathrm{soc}A$, hence, by case1, $h$ can be extended to an automorphism of $A$.\end{par}\end{proof}

\begin{definition}\label{1234}(Pseudo-injective modules) Let $R$ be any ring (possibly infinite). A left $R$-module $M$ is said to be \emph{pseudo-injective} if, for any submodule $B\subset M$, every monomorphism $f:B\rightarrow M$ extends to an endomorphism of $M$.\end{definition}
Thus, Lemma \ref{psss} can be restated as follows (this is exactly Proposition 3.2, \cite{r1}).
\begin{corollary}\label{pseudoiff}The alphabet $A$ has the extension property for length 1 codes with respect to Hamming weight if and only if $A$ is a pseudo-injective module.\end{corollary}
In \cite{r7}, J. Wood was able to provide sufficient conditions (that were shown necessary, in the same paper) for a left-module alphabet $_RA$ to satisfy the extension property with respect to Hamming weight.

\begin{theorem}\label{sufficient}An alphabet $_RA$ has the extension property with respect to Hamming weight if \begin{enumerate}
\item[$(1)$] $_RA$ is pseudo-injective, and
\item[$(2)$] $\mathrm{soc}(A)$ is cyclic.
\end{enumerate}\end{theorem}

\begin{proof}Assume that conditions (1) and (2) hold for the alphabet $_RA$. Let $c_1,C_2\subset A^n$ be two $R$-linear codes, and  $f:C_1\rightarrow C_2$ be an isomorphism preserving Hamming weight. We show that $f$ can be extended into a monomial transformation.\\
By Proposition \ref{Rhat}, $_RA$ can be considered as a submodule of $_R\widehat{R}$, and thereby, $C_1,C_2\subset \widehat{R}^n$ can be considered  as two $R$-linear codes over $_R\widehat{R}$. Of course, the Hamming weight of a codeword stays the same under this assumption. Now, since $\widehat{R}$ is a Frobenius $(R,R)$-bimodule, by Theorem \ref{FBI}, $f$ extends to a monomial transformation $F:\widehat{R}^n\rightarrow\widehat{R}^n$.
Thus there exist $\tau_1,\ldots,\tau_n\in \mathrm{Aut}_R(_R\widehat{R})$ and a permutation $\sigma$ such that
$$(x_1,\ldots,x_n)F=(x_{\sigma(1)}\tau_1,\ldots,x_{\sigma(n)}\tau_n),\quad (x_1,\ldots,x_n)\in \widehat{R}^n.$$

Now, notice that a monomial transformation can split into two components, a diagonal component $D$, and a permutation component $P$, acting as follows:
\begin{align*}(x_1,\ldots,x_n)D&=(x_1\tau_1,\ldots,x_n\tau_n),\; \text{and}\\
(x_1,\ldots,x_n)P&=(x_{\sigma(1)},\ldots,x_{\sigma(n)}),\quad (x_1,\ldots,x_n)\in \widehat{R}^n.\end{align*}
Then $xF=xPD, x\in \widehat{R}^n$.\\
Set $C_3=C_1P\subset A^n\subset \widehat{R}^n$, then $D$ is an $R$-linear isomorphism from $C_3$ to $C_2$ preserving Hamming weight. We consider the components of $D$. For each $i$, project $C_3,C_2$ to codes $c_3^{(i)},C_2^{(i)}\subset A\subset\widehat{R}$. Now, the transformation $D^{(i)}:\widehat{R}\rightarrow \widehat{R}$ defined by $xD^{(i)}=x\tau_i,$ for $x\in \widehat{R}$ is an $R$-linear isomorphism taking $C_3^{(i)}$ to $C_2^{(i)}$, besides, it preserves Hamming weight. Since $A$ is pseudo-injective, it follows from Corollary \ref{pseudoiff} that $D^{(i)}:C_3^{(i)}\rightarrow C_2^{(i)}$ extends to an automorphism $\tau'_i\in \mathrm{Aut_R(A)}$. These automorphisms of $A$ can be used to produce a monomial transformation $F'$ of $A^n$ extending $f$ in a clear manner.

 \end{proof}

As we saw in Theorem \ref{du1}, finite Frobenius rings satisfy the EP with respect to Hamming weight. In the same paper, \cite{duality}, a partial converse was proved, stated in Theorem \ref{du2}. In \cite{r6}, in 2008, the full converse was provided, following a strategy of Dinh and L\'{o}pez-Permouth \cite{r1} (it should be mentioned that in \cite{grefff} T. Honold and M. Greferath made a  similar construction for codes over $\mathbb{Z}_m$ that predated that in \cite{duality}).
The strategy has three parts. (1) If a finite ring is not Frobenius, its socle contains a matrix module of a particular type. (2) Provide a counter-example to the EP in the context of linear codes over this special module. (3) Show that this counter-example over the matrix module pulls back to give a counter-example over the original ring. The proof was easily adapted in \cite{r7} (2009) to prove that a module alphabet $_RA$ has the EP for Hamming weight if and only if $A$ is pseudo-injective with cyclic socle. This is what we prove next. For the technical results in combinatorics used here, see \cite{Lintcomb} and \cite{stanley}.

\begin{definition}($q$-Binomial Coefficients). The $q$-binomial coefficients, also known as Gaussian coefficients, and Gaussian numbers, are defined and denoted as follows$$\begin{bmatrix}k\\l\end{bmatrix}_q=\frac{(1-q^k)(1-q^{k-1})\cdots (1-q^{k-l+1})}{(1-q^l)(1-q^{l-1})\cdots(1-q)}.$$\end{definition}

\begin{lemma}\label{gauss}The $q$-binomial coefficient $\begin{bmatrix}k\\l\end{bmatrix}_q$ is exactly the number of row reduced echelon matrices (RRE's) of size $l\times k$ and rank $l$ over $\mathbb{F}_q$, which also equals the number of column reduced echelon matrices (CRE's) of size $k\times l$ and rank $l$ over $\mathbb{F}_q$. It is also the number of $\mathbb{F}_q$-linear subspaces of dimension $l$ in a $\mathbb{F}_q$-vector space of dimension $k$.  \end{lemma}

\begin{lemma}\label{cauchy}$\mathrm{(Cauchy\;Binomial\;Theorem)}.$

$$\overset{k-1}{\underset{i=0}{\prod}}(1+xq^i)=\overset{k}{\underset{j=0}{\sum}}\begin{bmatrix}k\\j\end{bmatrix}_q
q^{{j\choose 2}}x^j.$$\end{lemma}
The next theorem proves the necessary conditions for a module alphabet to satisfy the extension property with respect to Hamming weight.

\begin{theorem}\label{wood}Let $R=\mathbb{M}_m(\mathbb{F}_q)$ and $A=\mathbb{M}_{m\times k}(\mathbb{F}_q)$. If $k>m$, there exist linear codes $C_+,C_-\subset A^N, N=\overset{k-1}{\underset{i=1}{\prod}}(1+q^i)$, and an $R$-linear isomorphism $f:C_+\rightarrow C_-$ that preserves Hamming weight, yet there is no monomial transformation extending $f$.\end{theorem}

\begin{proof}Within the $N$-tuples of $\mathbb{M}_k(\mathbb{F}_q)^N$, whose entries are $k\times k$ matrices, we consider two special vectors $v_+,v_-$: The entries of the vector $v_+$ are all $k\times k$ CRE's of even rank,  each matrix is repeated $q^{{r\choose 2}}$ times within $v_+$, where $r$ is the rank of the indicated matrix; the order of the entries is not significant. In particular, the zero matrix occurs only once in $v_+$, as ${0\choose 2}=0$. Thus, the length $L_+$ of $v_+$ is given by $$L_+=\overset{k}{\underset{\begin{smallmatrix}r=0\\r\;\text{even}\end{smallmatrix}}{\sum}}q^{{r\choose 2}}
\begin{bmatrix}k\\r\end{bmatrix}_q.$$
Similarly, $v_-$ consists of all CRE's of odd rank, following the same repetition rule. The length $L_-$ of $v_-$ is given by
$$L_-=\overset{k}{\underset{\begin{smallmatrix}r=1\\r\;\text{odd}\end{smallmatrix}}{\sum}}q^{{r\choose 2}}
\begin{bmatrix}k\\r\end{bmatrix}_q.$$
Applying Cauchy theorem for $x=-1$ and $x=1$ yield, respectively, that $$L_+-L_-=0\quad\text{and}\quad L_++L_-=\overset{k-1}{\underset{i=0}{\prod}}(1+q^i).$$
Since the $i=0$ term in the product equals $2$, we get $$L_+=L_-=\overset{k-1}{\underset{i=1}{\prod}}(1+q^i)=N,$$as claimed.\\
Now, define the $R$-linear homomorphisms $g_\pm:A\rightarrow A^N$ by $Xg_\pm=Xv_\pm, X\in A$, where $Xv_\pm$ denotes entry-wise matrix multiplication. Consider the two codes $C_\pm=Ag_\pm\subset A^N$, for which we prove the following two claims:
\begin{enumerate}
\item[$1-$] For $X\in A$, $\mathrm{wt}(Xg_+)=\mathrm{wt}(Xg_-)$ (Hamming weight preserved).
\item[$2-$] The map $f:C_+\rightarrow C_-$ defined by $Xg_+\mapsto Xg_-$ (i.e. $g_-=g_+\circ f$) is an $R$-linear isomorphism that preserves Hamming weight, yet, there is no monomial transformation extending $f$.
\end{enumerate}
Setting $\Delta(X)=\mathrm{wt}(Xg_+)-\mathrm{wt}(Xg_-)$, we prove the first claim by showing that $\Delta(X)=0$ for any $X\in A$. We have
$$\Delta(X)=\overset{k}{\underset{\begin{smallmatrix}r=0\\r\;\text{even}\end{smallmatrix}}{\sum}}q^{{r\choose 2}}
\underset{\begin{smallmatrix}\lambda \;\text{CRE}\\\text{rank}\; r\end{smallmatrix}}{\sum}\delta(X\lambda)-\overset{k}{\underset{\begin{smallmatrix}r=0\\r\;\text{odd}\end{smallmatrix}}
{\sum}}q^{{r\choose 2}}
\underset{\begin{smallmatrix}\lambda \;\text{CRE}\\\text{rank}\; r\end{smallmatrix}}{\sum}\delta(X\lambda),$$
where $\delta(Y)=\left\{\begin{smallmatrix}1,&\; Y\neq0\\0,&\; Y=0\end{smallmatrix}\right.$. The terms in the equation can be re-collected as

$$\Delta(X)=\overset{k}{\underset{r=0}{\sum}}(-1)^r q^{{r\choose 2}}
\underset{\begin{smallmatrix}\lambda \;\text{CRE}\\\text{rank}\; r\end{smallmatrix}}{\sum}\delta(X\lambda).$$
We now show that $\Delta(X)$ depends only on the rank of $X$. Given that $X\in A$ is of rank $s, 1\leq s\leq m $, then $$X=P\begin{pmatrix}I_s&0\\0&0\end{pmatrix}Q,$$ for some $P\in GL(m, \mathbb{F}_q)$ and $Q\in GL(k, \mathbb{F}_q)$, i.e. $P$ and $Q$ are invertible matrices of sizes $m\times m$ and $k\times k$ respectively. For simplicity we shall denote $\left(\begin{smallmatrix}I_s&0\\0&0\end{smallmatrix}\right)$ by $I_s'$, thus $X=PI_s'Q$.\\
It is clear that $\delta(PY)=\delta(Y)=\delta(YQ)$ for any  $Y\in A, P\in GL(m, \mathbb{F}_q)$, and $Q\in GL(k, \mathbb{F}_q)$. It follows that $\Delta(X)=\Delta(I_s'Q)$. In computing $\Delta(I_s'Q)$ we encounter the inner sum $$\underset{\begin{smallmatrix}\lambda \;\text{CRE}\\\text{rank}\; r\end{smallmatrix}}{\sum}\delta(I_s'Q\lambda).$$
To dissolve this one, we appeal to the following remark. Every matrix is column (row) equivalent to a unique CRE (RRE), whence $Q\lambda=\lambda'Q'$ for some CRE matrix $\lambda'$, and some $Q'\in GL(k, \mathbb{F}_q)$ (depending on $\lambda$). Moreover, if $\lambda_1,\lambda_2$ are two CRE's sharing the same $\lambda'$ in this manner, that is, $Q\lambda_1=\lambda'Q_1'$ and $Q\lambda_2=\lambda'Q_2'$, then $\lambda_1=\lambda_2 Q_2'^{-1}Q_1'$ and $\lambda_1$ is column equivalent to $\lambda_2$,
implying $\lambda_1=\lambda_2$ as they are CRE's. Thus, $$\underset{\begin{smallmatrix}\lambda \;\text{CRE}\\\text{rank}\; r\end{smallmatrix}}{\sum}\delta(I_s'Q\lambda)=
\underset{\begin{smallmatrix}\lambda' \;\text{CRE}\\\text{rank}\; r\end{smallmatrix}}{\sum}\delta(I_s'\lambda'Q')=\underset{\begin{smallmatrix}\lambda' \;\text{CRE}\\\text{rank}\; r\end{smallmatrix}}{\sum}\delta(I_s'\lambda').$$
Notice that although each $Q'$ depends on its corresponding $\lambda$, they disappeared from the last summation as they don't affect $\delta$. This shows that $\Delta(X)=\Delta(I_s')$, and that $\Delta(X)$ depends only on the rank of $X$. We still need to prove that $\Delta(I_s')=0$, for all $ 1\leq s\leq m$.
If $\lambda$ is a CRE of rank $r$, the rows of $I_s'\lambda$ consist of the first $s$ rows of $\lambda$ followed by $k-s$ zero rows. Thus, $I_s'\lambda=0$ if and only if the first $s$ rows of $\lambda$ are zero. There are exactly $\begin{bmatrix}k-s\\r\end{bmatrix}_q$ column reduced echelon matrices of rank $r$ with first $s$ rows only zeros\footnote{If $r\leq k$, the number $\begin{bmatrix}k\\r\end{bmatrix}_q$ counts the number of CRE's of size $k\times k$ of rank $r$ over $\mathbb{F}_q$. However,  the number of nonzero columns in such a matrix is equal to $r$, so, whenever $r<k$ the matrix will have the last $k-r$ columns equal to zero. This shows that $\begin{bmatrix}k\\r\end{bmatrix}_q$ is also the number of $k\times r$ CRE's of rank $r$ over $\mathbb{F}_q$. If $r>k$, the number $\begin{bmatrix}k\\r\end{bmatrix}_q$ is equal to zero.}. Thus, in the summation $$\underset{\begin{smallmatrix}\lambda \;\text{CRE}\\\text{rank}\; r\end{smallmatrix}}{\sum}\delta(I_s'\lambda)$$ there are $\begin{bmatrix}k\\r\end{bmatrix}_q$ terms, out of which there are $\begin{bmatrix}k-s\\r\end{bmatrix}_q$ zero terms and the remaining terms equal to 1. This gives that
\begin{align*}\Delta(I_s')&=\overset{k}{\underset{r=1}{\sum}}(-1)^r q^{{r\choose 2}}
\left\{\begin{bmatrix}k\\r\end{bmatrix}_q-\begin{bmatrix}k-s\\r\end{bmatrix}_q\right\}\\
&=
\overset{k}{\underset{r=1}{\sum}}(-1)^r q^{\left(\begin{smallmatrix}
r\\2\end{smallmatrix}\right)}\begin{bmatrix}k\\r\end{bmatrix}_q-\overset{k-s}{\underset{r=1}{\sum}}(-1)^r q^{{r\choose 2}}\begin{bmatrix}k-s\\r\end{bmatrix}_q=\end{align*}
\begin{align*}&=\left(\overset{k-1}{\underset{i=0}{\prod}}(1-q^i)-1\right)-\left(\overset{k-s-1}{\underset{i=0}{\prod}}(1-q^i)-1\right)\\
&=(0-1)-(0-1)=0,\end{align*}
use has been made of Cauchy theorem (Theorem \ref{cauchy}), and that $\begin{bmatrix}k\\r\end{bmatrix}_q$ is equal to zero when  $r>k$. Note that the assumption $k>m$ guaranteed that the sum  containing $\begin{bmatrix}k-s\\r\end{bmatrix}_q$ is nontrivial. Thus, $\Delta(X)=0$ for all $X\in A$, this  proves the first claim.\\
For the second claim, it is easily verified that the map $f$, defined as  in the statement of the claim, is an $R$-homomorphism. If $X\neq0$, then $Xv_\pm\neq0$ since  the identity matrix $I_k$ is, after all, a column reduced echelon matrix that will appear in either $v_+$ or $v_-$ and hence, using the above preservation of Hamming weight, $\mathrm{wt}(Xv_+)=\mathrm{wt}(Xv_-)\neq0$. This shows that $f$ is a monomorphism. Also, this previous argument shows that $|C_+|=|C_-|$. Thus $f:C_+\rightarrow C_-$ is an $R$-isomorphism preserving Hamming weight. Finally, $v_+$ contains the zero matrix in one component, and hence there is a fixed component that is zero in all the codewords of $C_+$, which is not the case for $C_-$ (there is no component that is zero in $Xv_-$ for all $X\in A$). This prevents the two codes $C_+,C_-$ of being monomially equivalent.
\end{proof}

The conditions in Theorem \ref{sufficient} are necessary by the following theorem \cite{r7}.
\begin{theorem}
If the  alphabet $_RA$ has the extension property with respect to Hamming weight then \begin{enumerate}
\item[$(1)$] $_RA$ is pseudo-injective, and
\item[$(2)$] $\mathrm{soc}(A)$ is cyclic.
\end{enumerate}
\end{theorem}
\begin{proof}
In view of Corollary \ref{pseudoiff}, we need only to check the second condition for necessity. Suppose that $\mathrm{soc}(A)\cong s_1T_1\oplus\cdots\oplus s_nT_n$ is not cyclic, then by Lemma \ref{simui} there is an index $i$ such that $s_i>\mu_i$ (we use the same notation as in the Lemma). So, we may assume  $s_iT_i\subset \mathrm{soc}(A)\subset A$. Recall that $T_i$  is the pullback to $R$ of the matrix module $_{\mathbb{M}_{\mu_i}(\mathbb{F}_{q_i})}\mathbb{M}_{\mu_i\times 1}(\mathbb{F}_{q_i})$, so that $s_iT_i$ is the pullback to $R$ of the $\mathbb{M}_{\mu_i}(\mathbb{F}_{q_i})$-module $B=\mathbb{M}_{\mu_i\times s_i}(\mathbb{F}_{q_i})$. Theorem \ref{wood} implies the existence of $\mathbb{M}_{\mu_i}(\mathbb{F}_{q_i})$-linear codes $C_+,C_-\subset B^N$, and an isomorphism $f:C_+\rightarrow C_-$ that preserves Hamming weight, yet $f$ can not extend to any monomial transformation. The projections $$R\rightarrow R/\mathrm{rad}R\rightarrow \mathbb{M}_{\mu_i}(\mathbb{F}_{q_i})$$ allow us to consider  $C_\pm$ as $R$-modules. Also, $f$ becomes an $R$-isomorphism preserving Hamming weight. Thus, we have produced two $R$-linear codes $C_\pm\subset (s_iT_i)^N\subset \mathrm{soc}(A)^N\subset A^N$ that are isomorphic through an isomorphism which preserves Hamming weight, yet the two codes are not monomially equivalent.
\end{proof}

In particular, by Theorem \ref{hon}, a finite ring $R$ has the extension property with respect to Hamming weight if and only if $R$ is Frobenius, this result first appeared in \cite{r6}.

\section{Parameterized Codes and Multiplicity Functions}
\subsection{Parameterized Codes}
In this section, the theory of the extension theorem for linear codes is displayed in an alternative way, a more abstract one, and the necessary and sufficient conditions are derived in terms of a certain mapping $W$ pertaining to the weight function in consideration.

\begin{definition}(Parameterized Code). Let $_RA,_RM$ be  finite left $R$-modules. A pair $(M,\Lambda)$, where $\Lambda:M\rightarrow A^n$ is an $R$-homomorphism, is called a \emph{parameterized code} of length $n$ over the alphabet $A$. Clearly, $M$ serves as an information space and $\Lambda$ serves as an encoder. The actual code is $C=M\Lambda\subset A^n$. \end{definition}
This way of defining a code is justified by the context of the extension problem. As we will see, the information space $M$ and the map $\Lambda$ are the tools needed in this alternative tackling.\\
Note that a code $C\subset A^n$ may be referred through more than one parameterized code. $\Lambda$ can be uniquely identified by its list of  coordinate functionals, thus $\Lambda=(\lambda_1,\ldots,\lambda_n), \lambda_i\in \mathrm{Hom}_R(M,A)$.\\
Fixing $M$ as an information space, $\mathcal{C}_n(M)$ will denote the set of all parameterized codes $(M,\Lambda)$ of length $n$ over the (fixed) alphabet $A$. For completeness, define $\mathcal{C}_0(M)$ to be the singleton containing only the empty code of length 0. On the disjoint union $\mathcal{C}(M)=\bigcup_{n\geq0}\mathcal{C}_n(M)$, we can define a binary operation, giving rise to a monoid structure, the operation is just the following \emph{concatenation}.
\begin{center}$\mathcal{C}_{n_1}(M)\times\mathcal{C}_{n_2}(M)\rightarrow \mathcal{C}_{n_1+n_2}(M),$\\
$((M,\Lambda_1),(M,\Lambda_2))\mapsto (M,\Lambda_1|\Lambda_2),$\end{center}
where $\Lambda_1|\Lambda_2=(\lambda_{11},\ldots,\lambda_{1n_1};\lambda_{21},\ldots,\lambda_{2n_2})$.

\begin{proposition}$\mathcal{C}(M)$ is a monoid (associative semigroup with identity) under concatenation. In particular, the identity is the empty code in $\mathcal{C}_0(M)$.\end{proposition}

\begin{proof}Clear.\end{proof}

Let $w:A\rightarrow \mathbb{Q}$ be a weight function with symmetry groups $G_{\text{rt}}, G_{\text{lt}}$. Sometimes we need to focus on $G_{\text{rt}}$-monomial transformations of a specific length $n$, so we set $$\mathcal{G}_n=\{T\;|\;T\;\text{is a $G_{\text{rt}}$-monomial transformation of $A^n$}\}.$$ Clearly, for any $n$, $\mathcal{G}_n$ is a group. Actually, $\mathcal{G}_n$ has a right action on $\mathcal{C}_n(M)$, $\Lambda T:=\Lambda\circ T$, where it is understood that $\Lambda\in \mathrm{Hom}_R(M,A^n)$ stands for the parameterized code $(M,\Lambda)$ and $T\in\mathcal{G}_n$. The orbit of $(M,\Lambda)$ is thus the set $\{(M,\Lambda\circ T)\;|\;T\in\mathcal{G}_n\}$. Let $\overline{\mathcal{C}}_n(M)$ denote the orbit space of $\mathcal{C}_n(M)$ under this action, and let $\overline{\mathcal{C}}(M)$ be the disjoint union $\bigcup_{n\geq0}\overline{\mathcal{C}}_n(M)$.

\begin{proposition}\label{next} On $\overline{C}(M)$, concatenation is a well-defined operation, making it a commutative monoid.\end{proposition}
\begin{proof}Let $(M,\Lambda_1),(M,\Lambda_2)$ be two representatives of the same element $\overline{\Lambda}\in \overline{\mathcal{C}}_{n_1}(M)\subset\overline{\mathcal{C}}(M)$; similarly, let $\Theta_1,\Theta_2$ be two representatives of the same element $\overline{\Theta}\in \overline{\mathcal{C}}_{n_2}(M)\subset\overline{\mathcal{C}}(M)$. Thus, for some $T_\Lambda\in \mathcal{G}_{n_1}$ and $T_\Theta\in\mathcal{G}_{n_2}$, $\Lambda_2=\Lambda_1\circ T_\Lambda$ and $\Theta_2=\Theta_1\circ T_\Theta$. We shall show that $(M, \Lambda_1|\Theta_1)$ and $(M,\Lambda_2|\Theta_2)$ lie in the same orbit in $\overline{\mathcal{C}}_{n_1+n_2}(M)$. Assume that $T_\Lambda$ is built up by $\tau_1,\ldots,\tau_{n_1}\in \mathrm{Aut}_R(A)$, $\sigma\in S_{n_1}$, and that $T_\Theta$ is built up by $\zeta_1,\ldots,\zeta_{n_2}\in \mathrm{Aut}_R(A)$, $\pi\in S_{n_2}$. Consider the
map $T:A^{n_1+n_2}\rightarrow A^{n_1+n_2}$ defined by \begin{align*}(x_1,\ldots,x_{n_1+n_2})T&=((x_1,\ldots,x_{n_1})T_\Lambda;(x_{n_1+1},\ldots,x_{n_1+n_2})T_\Theta)\\
&=(x_{\sigma(1)}\tau_1,\ldots,x_{\sigma(n_1)}\tau_{n_1};x_{\pi(1)+n_1}\zeta_1,\ldots,x_{\pi(n_2)+n_1}\zeta_{n_2}).\end{align*}
It is easily seen that $T\in \mathcal{G}_{n_1+n_2}$, and that $\Lambda_2|\Theta_2=(\Lambda_1|\Theta_1)\circ T$. This establishes that concatenation is well defined on $\overline{\mathcal{C}}(M)$. Commutativity of concatenation on $\overline{C}(M)$ follows since any permutation of the coordinate functionals represents the action of a $G_{\text{rt}}$-monomial transformation.
\end{proof}
Notice that the two codes $(M,\Lambda)$ and $(M,\Lambda|0)$ are essentially the same though they do not have the same length. Thus, concatenating with the zero code $(M,0)\in \mathcal{C}_1(M)$ provides an injection from $\mathcal{C}_n(M)$ into $\mathcal{C}_{n+1}(M)$, defined by $\Lambda\mapsto \Lambda|0$. Besides, these injections are still well-defined from $\overline{\mathcal{C}}_n(M)$ into $\overline{\mathcal{C}}_{n+1}(M)$ for all $n$. Identifying any two elements of $\overline{\mathcal{C}}(M)$ whenever they differ only by concatenating zero codes, we obtain the identification space $\mathcal{E}(M)$, whose elements are parameterized codes with no zero components and identified up to $G_{\text{rt}}$-monomial transformations. The proof of the following proposition straightforward in view of Proposition \ref{next}.
\begin{proposition}On $\mathcal{E}(M)$, concatenation is a well-defined operation, making it a commutative monoid.\end{proposition}

\subsection{Multiplicity Functions}
The group $G_{\text{rt}}$ acts from the right on $\mathrm{Hom}_R(M,A)$, the orbit $[\lambda]$ of an element $\lambda\in\mathrm{Hom}_R(M,A)$ is given by $[\lambda]=\{\lambda\circ \tau\;|\;\tau\in G_{\text{rt}}\}$. Let $O^\sharp$ denote the orbit space of this action. The idea for multiplicity functions is to count, for a code parameterized by $\Lambda$, the number of coordinate functionals in each orbit. Set $F(O^\sharp,\mathbb{N})$ to be the set of all functions from $O^\sharp$ to $\mathbb{N}$. $F(O^\sharp,\mathbb{N})$ is a commutative monoid under point-wise addition.  Now, consider the sub-monoid
$$F_0(O^\sharp,\mathbb{N})=\{\eta:O^\sharp\rightarrow\mathbb{N}\;|\; \eta([0])\}$$ of $F(O^\sharp,\mathbb{N})$. The following theorem verifies that we can effectively identify codes with their corresponding multiplicity functions.

\begin{theorem}\label{muf}The following are true for a finite left $R$-module $M$:
\begin{enumerate}
\item[$(1)$]$\overline{\mathcal{C}}(M)$ and $F(O^\sharp,\mathbb{N})$ are isomorphic as monoids.
\item[$(2)$]$\mathcal{E}(M)$ and $F_0(O^\sharp,\mathbb{N})$ are isomorphic as monoids.
\end{enumerate}\end{theorem}

\begin{proof}The orbit of $\Lambda\in\mathrm{Hom}_R(M,A^n)$, under the right action of $\mathcal{G}_n$, is given by $[\Lambda]=\{\Lambda\circ T\;|\; T\in \mathcal{G}_n\}$. Assume that $\Lambda=(\lambda_1,\ldots,\lambda_n)$, $\lambda_i\in\mathrm{Hom}_R(M,A), i=1,\ldots,n $. Define the function $\eta_\Lambda:O^\sharp\rightarrow\mathbb{N}$ by $$\eta_\Lambda:[\mu]\mapsto |\{i\;|\;\lambda_i\in [\mu]\}|.$$
Now consider the map $\Phi:\overline{\mathcal{C}}(M)\rightarrow F(O^\sharp,\mathbb{N})$ defined by $[\Lambda]\mapsto \eta_\Lambda$. This map, as easily verified, is an isomorphism of monoids. Besides, the second isomorphism is clear in view of such a map.
\end{proof}
\subsection{The Weight Mapping}

Consider an alphabet $A$, on which a weight $w$ is defined as in Definition \ref{weight}. The weight $w$ gives rise to a well-defined map on $\overline{\mathcal{C}}(M)$ and $\mathcal{E}(M)$, defined in terms of multiplicity functions as follows:
\begin{center}$W: F(O^\sharp,\mathbb{N})\rightarrow F(M,\mathbb{Q})$
\begin{equation}\label{W}\eta\mapsto(x\mapsto\underset{[\lambda]\in O^\sharp}{\sum}w(x\lambda)\eta(\lambda))\end{equation}\end{center}
Notice that, with no ambiguity, we can sometimes remove  orbit brackets and use the same element symbol, just as done for the input of $\eta$ in the last definition. Also, if $\eta_\Lambda$ is the multiplicity function corresponding to a parameterized code $(M,\Lambda)$, as in Theorem \ref{muf}, then $\underset{[\lambda]\in O^\sharp}{\sum}w(x\lambda)\eta_\Lambda(\lambda)=w(x\Lambda)$. In the next proposition it is implicitly understood that $F(M,\mathbb{Q})$ is a monoid under point-wise addition, just as we mentioned earlier for $F(O^\sharp,\mathbb{N})$.
\begin{proposition}\label{Wp}The map $W$ is  well-defined, additive, and satisfies that $W(\eta)(0)=0$ for any $\eta\in F(O^\sharp,\mathbb{N})$. Moreover, the image of $W$ is contained within the $G_{\mathrm{lt}}$-invariant functions from $M$ to $\mathbb{Q}$.\end{proposition}
\begin{proof}$W$ is  well-defined, simply by observing that whenever $\lambda,\lambda'$ are in the same orbit, say $\lambda'=\lambda\circ\tau$ with $\tau\in G_{\text{rt}}$,  then $w(x\lambda)=w(x(\lambda\circ\tau))=w(x\lambda')$; and that $\eta(\lambda)=\eta(\lambda')$. Let $\eta_1,\eta_2\in F(O^\sharp,\mathbb{N})$, then
\begin{align*}W(\eta_1+\eta_2)(x)&=\underset{[\lambda]\in O^\sharp}{\sum}w(x\lambda)(\eta_1+\eta_2)(\lambda)\\
&=\underset{[\lambda]\in O^\sharp}{\sum}w(x\lambda)(\eta_1(\lambda)+\eta_2(\lambda))\\
&=\underset{[\lambda]\in O^\sharp}{\sum}w(x\lambda)\eta_1(\lambda)+\underset{[\lambda]\in O^\sharp}{\sum}w(x\lambda)\eta_2(\lambda)\\
&=W(\eta_1)(x)+W(\eta_2)(x).\end{align*}
That $W(\eta)(0)=0$ for any $\eta\in F(O^\sharp,\mathbb{N})$ is direct from the definition of $W$. We finally prove the assertion concerning the image. Recall that $G_{\text{lt}}=\{u\in \mathcal{U}: w(ua)=w(a),\;\text{for all}\;\;a\in A\}$, and let $u\in G_{\text{lt}}$ and $x\in M$. Then
\begin{align*}W(\eta)(ux)&=\underset{[\lambda]\in O^\sharp}{\sum}w((ux)\lambda)\eta(\lambda)\\
&=\underset{[\lambda]\in O^\sharp}{\sum}w(u(x\lambda))\eta(\lambda)\\
&=\underset{[\lambda]\in O^\sharp}{\sum}w(x\lambda)\eta(\lambda)\\
&=W(\eta)(x).\end{align*}\end{proof}

It is important now to consider the left action of $G_{\text{lt}}$ on $M$ given by the left $R$-module structure of $M$. Thus, the orbit of $m\in M$ under this action is $\{um\;|\;u\in G_{\text{lt}}\}$. Denoting the orbit space of this action by $O$, we get that the set of  $G_{\mathrm{lt}}$-invariant functions from $M$ to $\mathbb{Q}$ is the same as the set $F(O,\mathbb{Q})$. Under the  point-wise addition, the space $F(O,\mathbb{Q})$ is a $\mathbb{Q}$-vector space of dimension $|O|$, a conceivable basis is  the family of functions assigning $1$ to only one orbit and 0 for all other orbits. Let $F_0(O,\mathbb{Q})$ denote the set of functions for which the image of the zero orbit is zero, then $F_0(O,\mathbb{Q})$ is a $\mathbb{Q}$-subspace of $F(O,\mathbb{Q})$, of dimension $|O|-1$. By Proposition \ref{Wp}, $W$ maps $F(O^\sharp,\mathbb{N})$ to $F_0(O,\mathbb{Q})$. The next theorem establishes the relation between $W$ and the extension property with respect to $w$.

\begin{theorem}\label{imppar}For an alphabet $A$, $A$ has the extension property with respect to the weight $w$ if the restriction $W:F_0(O^\sharp,\mathbb{N})\rightarrow F_0(O,\mathbb{Q})$ is injective for every finite left $R$-module $_RM$.\\
Besides, the converse holds if, for every $n$, the weight of any nonzero codeword in $A^n$ is nonzero. \end{theorem}

\begin{proof}Assume that $W$ is injective for every finite left $R$-module $M$, and let $f:C_1\rightarrow C_2$ be a weight preserving $R$-isomorphism between the two codes $C_1,C_2\subset A^n$. Let $M$ be the underlying module of the two codes. Suppose $C_1$ and $C_2$ are parameterized in $\mathcal{E}(M)$, respectively, as $(M,\Lambda_1)$ and $(M,\Lambda_2)$. Thus $C_1$ and $C_2$ are, respectively, the images, in $A^n$, of $\Lambda_1$ and $\Lambda_2$; in particular, $\Lambda_2=\Lambda_1\circ f$. Let $\eta_1$ and $\eta_2$ be the multiplicity functions associated, respectively, with $(M,\Lambda_1)$ and $(M,\Lambda_2)$ (use Theorem \ref{muf}). As $f$ preserves $w$, we get $W(\eta_1)=W(\eta_2)$. Then, $\eta_1=\eta_2\in F_0(O^\sharp,\mathbb{N})$ since $W$ is injective. Thus, $\Lambda_1$ and $\Lambda_2$ must be in the same orbit, that is, there exists a $G_{\text{rt}}$-monomial transformation $T$ of $A^n$ (i.e. $T\in\mathcal{G}_n$) such that $\Lambda_2=\Lambda_1\circ T$. $A$ is thus proved to have the extension property with respect to $w$.\\

For the second converse assertion, assume that $A$ has the extension property with respect to $w$ and that for every $n$, the weight of any nonzero codeword in $A^n$ is nonzero. Let $M$ be a finite left $R$-module and that for some multiplicity functions $\eta_1,\eta_2\in F_0(O^\sharp,\mathbb{N})$ we have $W(\eta_1)=W(\eta_2)$. Each of $\eta_1$ and $\eta_2$ correspond to a parameterized code, let these be $(M,\Lambda_1)$ and $(M,\Lambda_2)$, with index correspondence. Since $W(\eta_1)=W(\eta_2)$, we have $w(x\Lambda_1)=w(x\Lambda_2)$ for all $x\in M$. Now the assumption about the weight of nonzero codewords implies that $\mathrm{Ker}\Lambda_1=\mathrm{Ker}\Lambda_2$. However, we may assume the kernels are zero. Were these kernels nonzero, we could pass to the quotients $M/\mathrm{Ker}\Lambda_1=M/\mathrm{Ker}\Lambda_2$, and consider it as the underlying module, nothing affecting the multiplicity functions. Thus we're assuming $\Lambda_1,\Lambda_2$ are injective. Let $C_1=M\Lambda_1$ and $C_2=M\Lambda_2$, and let $f:C_1\rightarrow C_2$ be $f=\Lambda_1^{-1}\circ\Lambda_2$. Thus, $f$ is an $R$-isomorphism, moreover, $f$ preserves $w$ since $w(x\Lambda_1)=w(x\Lambda_2), x\in M$. By the extension property assumed for $A$, $f$ extends to a $G_{\text{rt}}$-monomial transformation $T$ of $A^n$ taking $C_1$ to $C_2$, hence $\Lambda_1$ and $\Lambda_2$ are in the same orbit, and by Theorem \ref{muf}, $\eta_1=\eta_2\in F_0(O^\sharp,\mathbb{N})$. This shows $W$ is injective.
\end{proof}

\subsection{Completion Over $\mathbb{Q}$: Virtual Codes}
When completing the space $F_0(O^\sharp,\mathbb{N})$ to $F_0(O^\sharp,\mathbb{Q})$, we obtain a $\mathbb{Q}$-vector space, this enables us, after proving some claims, to use theorems from linear algebra to decide the satisfaction of the extension property. As the elements of the new space $F_0(O^\sharp,\mathbb{Q})$ do not have to correspond to ``actual'' multiplicity functions standing for ``actual'' codes, we shall call these elements ``virtual codes''.

\begin{proposition}For any alphabet $A$ and finite $R$-module $M$,
\begin{enumerate}
\item[$(1)$] The mapping $W:F_0(O^\sharp,\mathbb{N})\rightarrow F_0(O,\mathbb{Q})$ extends to a linear transformation  $W:F_0(O^\sharp,\mathbb{Q})\rightarrow F_0(O,\mathbb{Q})$ of finite dimensional $\mathbb{Q}$-vector spaces.
\item[$(2)$] This extension is injective if and only if $W:F_0(O^\sharp,\mathbb{N})\rightarrow F_0(O,\mathbb{Q})$ is injective.
\item[$(3)$] Theorem \ref{imppar} still holds with the new map $W:F_0(O^\sharp,\mathbb{Q})\rightarrow F_0(O,\mathbb{Q})$ instead of $W:F_0(O^\sharp,\mathbb{N})\rightarrow F_0(O,\mathbb{Q})$.
\end{enumerate}
\end{proposition}

\begin{proof}(1). To remove any ambiguity while proving, denote the old map $F_0(O^\sharp,\mathbb{N})\rightarrow F_0(O,\mathbb{Q})$ by $W'$, and keep $W$ to denote the claimed extension $F_0(O^\sharp,\mathbb{Q})\rightarrow F_0(O,\mathbb{Q})$. Consider any $\eta\in F_0(O^\sharp,\mathbb{Q})$. Since the image of $\eta$ is finite ($O^\sharp$ is finite), there exist $a,b\in\mathbb{Z}-\{0\}$ and $\eta'\in F_0(O^\sharp,\mathbb{N})$ such that $\eta=\frac{a}{b}\eta'$ (take $a$ to be the g.c.d. of numerators of values of $\eta$, and $b$ the l.c.m. of denominators). Define $W(\eta):=\frac{a}{b} W'(\eta')$, $W$ is then a linear transformation of $\mathbb{Q}$-vector spaces.\\
(2). Using the same notation followed in (1), it is clear that $W'$ is injective whenever $W$ is injective, for we have $F_0(O^\sharp,\mathbb{N})\subset F_0(O^\sharp,\mathbb{Q})$. Conversely, suppose that $W'$ is injective, and suppose that $\eta\in \mathrm{Ker}W\subset F_0(O^\sharp,\mathbb{Q})$. Consider the function $k\eta$, where $k$ is any positive integer clearing the denominators of values of $\eta$, thus $k\eta\in F_0(O^\sharp,\mathbb{Z})$. We can write $k\eta$ as $k\eta=\eta_+-\eta_-$, where $\eta_+,\eta_-\in F_0(O^\sharp,\mathbb{N})$ (simply consider the positive and negative components of $k\eta$). By Proposition \ref{Wp}, $W$ is additive, and since $W(\eta)=0$ we get $0=W(k\eta)=W(\eta_+)-W(\eta_-)$, and $W(\eta_+)=W(\eta_-)$. But, then $W'(\eta_+)=W(\eta_+)=W(\eta_-)=W'(\eta_-)$, and the injectivity of $W'$ implies that $\eta_+=\eta_-$. Hence, $k\eta=0$ and $\eta=0$. This shows that $W$ injective. \\

(3). Follows directly from (2).\end{proof}

We can use these results to represent $W:F_0(O^\sharp,\mathbb{Q})\rightarrow F_0(O,\mathbb{Q})$ as a matrix, and to do this, we need to work on a bases of the $\mathbb{Q}$-vector spaces $F_0(O^\sharp,\mathbb{Q})$ and $F_0(O,\mathbb{Q})$, whose dimensions are $|O^\sharp|-1$ and $|O|-1$, respectively. Consider the elements $_\lambda\delta\in F_0(O^\sharp,\mathbb{Q})$, where $\lambda\in O^\sharp$, defined by \begin{center}$_\lambda\delta(\nu)=\left\{\begin{matrix}1&\text{if}\; \nu=\lambda,\\0&\text{if}\; \nu\neq\lambda.\end{matrix}\right.$\end{center}
Similarly, consider the elements $\delta_x\in F_0(O,\mathbb{Q})$, where $x\in O$, defined by \begin{center}$\delta_x(y)=\left\{\begin{matrix}1&\text{if}\; y=x,\\0&\text{if}\; y\neq x.\end{matrix}\right.$\end{center}
It is easily verified that these two families are natural bases for their corresponding spaces. Any $\eta\in F_0(O^\sharp,\mathbb{Q})$ is expressed as $\eta=\underset{\lambda\in O^\sharp}{\sum}\eta(\lambda){_\lambda\delta}$, and hence $\eta$ may be viewed as a column vector indexed by the nonzero elements of $O^\sharp$, with the $\lambda$-entry being $\eta(\lambda)$. Also, any $h\in F_0(O,\mathbb{Q})$ is expressed as $h=\underset{x\in O}{\sum}h(x)\delta_x$, and may be viewed as a row vector indexed by the nonzero elements of $O$, with the $x$-entry being $h(x)$.\\
Now we can represent $W$ by a matrix (also denoted $W$) of size $(|O|-1)\times(|O^\sharp|-1)$ with rows indexed by the nonzero elements of $O$ and columns indexed by the nonzero elements of $O^\sharp$, the entry of $W$ in row $x\in O$ and column $\lambda\in O^\sharp$ is $W_{x,\lambda}=w(x\lambda)$. The entries are well-defined by the definitions of symmetry groups. By equation \ref{W}, this matrix indeed represents $W$.

\subsection{Matrix Module Case}
Here we project the previous manipulations to the matrix module case. Let $R=\mathbb{M}_m(\mathbb{F}_q)$, the ring of all $m\times m$ matrices over $\mathbb{F}_q$. Let $A=\mathbb{M}_{m\times k}(\mathbb{F}_q)$ be our alphabet ($A$ is a left $R$-module). We are concerned here with the case of $\mathrm{wt}$, the Hamming weight. Then $G_{\text{lt}}=\mathcal{U}(R)=GL_m(\mathbb{F}_q)$, the group of all nonsingular $m\times m$ matrices over $\mathbb{F}_q$; and $G_{\text{rt}}=\mathrm{Aut}_R(A)=GL_k(\mathbb{F}_q)$. Let $M$ be a finite left $R$-module. Since $R$ is a simple ring, it follows that for some $l$, $M\cong \mathbb{M}_{m\times l}(\mathbb{F}_q)$, so we consider $M=\mathbb{M}_{m\times l}(\mathbb{F}_q)$. $\mathrm{Hom}_R(M,A)\cong \mathbb{M}_{l\times k}(\mathbb{F}_q)$, which acts on $M$ from the right by matrix multiplication. $O^\sharp$, the orbit space of $\mathrm{Hom}_R(M,A)$ under the right action of $G_{\text{rt}}=GL_k(\mathbb{F}_q)$, is represented by the column reduced echelon matrices of size $l\times k$. The orbit space $O$, of $M$ under the left action of $G_{\text{lt}}=GL_m(\mathbb{F}_q)$, is represented by the row reduced echelon matrices of size $m\times l$. Thus:
\begin{enumerate}\item[(i)] $|O|=$ the number of row reduced echelon matrices of size $m\times l$.
\item[(ii)] $|O^\sharp|=$ the number of column reduced echelon matrices of size $l\times k$.
\end{enumerate}

If $k>m$, then $|O^\sharp|>|O|$ and $\mathrm{dim}F_0(O^\sharp,\mathbb{Q})>\mathrm{dim}F_0(O,\mathbb{Q})$. Thus, $W$ cannot be injective and $\mathrm{Ker}W\neq0$.  Note that if $\eta\in\mathrm{Ker}W $ is nonzero, we can split $\eta$ into its positive and negative components, $\eta_+$ and $\eta_-$, with $\eta=\eta_+-\eta_-$. Then $W(\eta_+)=W(\eta_-)$, hence $\eta_+,\eta_-$ define isometric codes that are not monomially equivalent (as $\eta\neq0$).

\newpage
.
\newpage

\chapter{Extension Theorem: \\Symmetrized Weight Compositions' Setting}
\section{Again, Frobenius Rings Marked out}
\begin{par}In this section, we're mainly studying another extension property, this time with respect to \emph{symmetrized weight compositions}. It turns out, again, that Frobenius rings (bimodules) are significant in this situation.\end{par}
\subsection{Averaging Characters}

\begin{par}We begin by reviewing some facts about characters on finite abelian groups, then an averaging process of characters is discussed. \end{par}
\begin{par}Let $G$ be a finite abelian group, written with additive notation. As defined in section \ref{characters}, the group of characters of $G$ is the group $\widehat{G}=\mathrm{Hom}_\mathbb{Z}(G,\mathbb{C}^\times)$. $\widehat{G}$ is a finite abelian group  under pointwise multiplication (the inverses are obtained through complex conjugation: $\pi^{-1}(x)={\pi(x)}^{-1}=\overline{\pi(x)}/|\pi(x)|^2$ ). \end{par}
Let $\mathcal{F}=\{f:G\rightarrow\mathbb{C}\}$ be the vector space of all complex-valued functions on $G$; $\mathrm{dim}_\mathbb{C}\mathcal{F}=|G|$. We can equip $\mathcal{F}$ with an inner product \begin{equation}\label{number}\langle f,g\rangle=\frac{1}{|G|}\underset{x\in G}{\sum}f(x) {g(x)}^{-1}.\end{equation}

The following lemma is standard, and the reader may refer to \cite{serre} for a proof.
\begin{lemma} \label{lemma6}$\widehat{G}$ forms an orthonormal basis for $\mathcal{F}$, where orthonormality is defined in terms of the inner product in equation \ref{number}.  \end{lemma}
\begin{par}Fixing a subgroup $U$ of the automorphism group of $G$, consider the right action on $G$ given by  $x\mapsto xu, u\in U,x\in G$. Denote by $\bar{x}$ the orbit of $x\in G$. This right action of $U$ on $G$ induces a left action of $U$ on $\mathcal{F}$, setting $uf={^uf}$, where $^uf(x)=f(xu)$. We define the $U$-\emph{invariant functions} to be the elements of the following subset
$$\mathcal{F}^U=\{f\in \mathcal{F}: f(xu)=f(x), u\in U, x\in G\}\; ;$$thus they are the fixed points under this left action and are constant on any orbit of $U$ (orbits for the first action). A projection $P:\mathcal{F}\rightarrow\mathcal{F}^U$ may be defined as follows. For $f\in\mathcal{F}$,
\begin{align}(Pf)x&=\frac{1}{|\bar{x}|}\underset{y\in \bar{x}}{\sum} f(y)=\frac{1}{U}\underset{u\in U}{\sum} f(xu)\\\label{proj}
& =\frac{1}{U}\underset{u\in U}{\sum}{\;} {^uf(x)}=\frac{1}{|\bar{f}|}\underset{g\in \bar{f}}{\sum} g(x).\end{align}Use has been made here of that $|U|=|\bar{x}||U_x|$, where $U_x$ is the stabilizer of $x$. One can verify easily that $P$ is a linear projection, that is, $P\circ P=P$. \end{par}

\begin{lemma}\label{lemma}If for some $u\in U$, $g={^uf}$, then $Pg=Pf$.
\end{lemma}

\begin{proof} Clear from the last equality in equation \ref{proj}, showing that $Pf$ is the average of those functions in $\bar{f}$.\end{proof}

\begin{proposition}\label{P}For any two characters $\pi,\psi \in \widehat{G}$, $\psi={^u\pi}$ (also denoted $\pi\sim \psi $) for some $u\in U$, if and only if $P\psi=P\pi$.\end{proposition}
\begin{proof}In view of Lemma \ref{lemma}, we only need to prove the ``if'' direction. Suppose $P\psi=P\pi$, then
$$\underset{u\in U}{\sum}{^u\psi}=\underset{v\in U}{\sum}{^v\pi}.$$
Since elements of $U$ are automorphisms of $G$, all the functions inside the summations are still characters. Now, the linear independence of characters, Theorem \ref{linearindep}, implies $\psi={^v\pi}$ for some $v\in U$.\end{proof}

\begin{theorem}\label{ex9}The distinct $P\pi$'s form an orthogonal system in $\mathcal{F}^U$, and hence are linearly independent.\end{theorem}

\begin{proof}If $P\psi\neq P\pi$, then $$|U|^2 \langle P\psi,P\pi\rangle=\langle\underset{u\in U}{\sum}{^u\psi},\underset{v\in U}{\sum}{^v\pi}\rangle=\underset{u,v}{\sum}\langle{^u\psi},{^v\pi}\rangle.$$ By Lemma \ref{lemma6}, $\langle{^u\psi},{^v\pi}\rangle=0$, being the product of distinct characters for every $u,v \in U$ (Proposition \ref{P}).
\end{proof}

\subsection{Extension Theorem}
\begin{definition}Let $R$ be a finite ring with unity, and let $_RA$  be a finite left $R$-module. The \emph{complete weight composition} is a function $n:A^n\times A\rightarrow\mathbb{N}$ given by $$n(x,a)=n_a(x)=|\{i:x_i=a\}|,\quad x=(x_1,\ldots,x_n)\in A^n,\quad a\in A.$$
\end{definition}

\begin{par}Let $G$ be any subgroup of $\mathrm{Aut}_R(A)$, the automorphism group of $A$. Define a relation $\sim_G$ on $A$ as follows: for $a,b\in A$, $a\sim_G b$ if there exists $\tau\in G $ such that $a\tau=b$. It is clear $\sim_G$ is an equivalence relation, its orbit space is denoted $A/G$. For simplicity, $\sim_{\mathrm{Aut}_R(A)}$ will be written $\sim$.\end{par}

\begin{definition}\label{swc}(Symmetrized Weight Compositions). Let $G$ be as above, the $G$-symmetrized weight composition ($\mathrm{swc}_G$)  is the function
$\mathrm{swc}_G$ : $A^n \times A\rightarrow \mathbb{Q}$ defined by, $$\mathrm{swc}_G(x, a) = |\{i:x_i\sim a\}|,$$where $x=(x_1,\ldots,x_n)\in A^n$ and $a\in A$. Thus, $\mathrm{ swc }_G$ counts the number of components in each orbit. As an abbreviation, we will write $\mathrm{swc}$ to denote $\mathrm{swc}_{\mathrm{Aut}_R(A)}$. \end{definition}

In these terms, the Hamming weight $\mathrm{wt}(x)$ of $x\in A^n$ can be written $\mathrm{wt}(x)=\sum_{a\neq0}n_a(x)$ which counts the number of non-zero components of $x$.  Also, $$\mathrm{swc}_G(x,a)=\underset{b\sim_G a}{\sum}n_b(x).$$

\begin{definition}(Extension Property).  The alphabet $A$ has the \emph{extension
property} (EP) with respect to $\mathrm{ swc }_G$  if for every $n$, and any two linear codes $ C_1, C_2\subset A^n $,  any
$R$-linear isomorphism $ f:C_1 \rightarrow C_2$   preserving $\mathrm{ swc }_G$  extends to a
$G$-monomial transformation of $ A^n$. By saying that $f$ preserves $\mathrm{ swc }_G$ it is meant that $\mathrm{ swc }_G(xf,a)=\mathrm{ swc }_G(x,a)$ for all $x\in C_1$ and  $a\in A$. Following the abbreviation mentioned in Definition \ref{swc}, the EP for $\mathrm{swc}$ means the EP for $\mathrm{swc}_{\mathrm{Aut}_R(A)}$.\end{definition}


\begin{par}Now, we set in use the previous methods of averaging characters. In the above notation, let $U$ be any subgroup of the group $\mathcal{U}$ of units  of $R$. Of course, right multiplication by  $u\in U$ defines an automorphism of $_RR$.\end{par}

\begin{theorem}Let $R$ be a Frobenius ring and $C\subset R^n$  a left linear code over the alphabet $_RR$. Then, any monomorphism $f:C\rightarrow R^n$ preserving $\mathrm{swc}_U$ extends to a $U$-monomial transformation of $R^n$.\end{theorem}

\begin{proof} \begin{par}Adopting the usual plot of the character theoretic proofs, view $C$ as the image, in $R^n$, of the map $\lambda=(\lambda_1,\ldots,\lambda_n)$ where $\lambda$ is the inclusion map $C\rightarrow R^n$, the components $\lambda_1,\ldots,\lambda_n$ are homomorphisms from $C$ into $R$. Similarly consider $\mu=(\mu_1,\ldots,\mu_n)=\lambda\circ f$. We will show that $\lambda_i=\mu_{\sigma(i)}u_i$ for some $\sigma\in S_n$ and units $u_1,\ldots,u_n\in U$. To plunge into the usual theme, the preservation of weight  equation $\mathrm{swc}_U(x\lambda,r)=\mathrm{swc}_U(x\mu,r)$ should be written as an equation of characters.\end{par}
\begin{par}By item 6 in Proposition \ref{char}, we have
$$n_r(x)=\frac{1}{|R|}\overset{n}{\underset{i=1}{\sum}}\underset{\pi\in\widehat{R}}{\sum}\pi(x_i-r)=\frac{1}{|R|}\overset{n}{\underset{i=1}{\sum}}
\underset{\pi\in\widehat{R}}{\sum}\pi(x_i)\pi^{-1}(r),$$from which we obtain
$$\mathrm{swc}_U(x,r)=\frac{|\bar{r}|}{|R|}\underset{\pi\in\widehat{R}}{\sum}\left (\overset{n}{\underset{i=1}{\sum}} \pi(x_i)\right)(P\pi^{-1})(r).$$
\end{par}
Thus the equation $\mathrm{swc}_U(x\lambda,r)=\mathrm{swc}_U(x\mu,r), x\in C, r\in R$ is written as
\begin{equation}\label{average} \underset{\pi\in\widehat{R}}{\sum}\left(\overset{n}{\underset{i=1}{\sum}}\pi(x\lambda_i)\right)(P\pi^{-1})(r)=
\underset{\pi\in\widehat{R}}{\sum}\left(\overset{n}{\underset{j=1}{\sum}}\pi(x\mu_j)\right)(P\pi^{-1})(r),\end{equation}for all $x\in C, r\in R$.
Fixing $\in C$, the above is an  equation of $U$-invariant functions on $R$. By Theorem \ref{ex9} and Proposition \ref{P} we deduce, for every $\pi\in \widehat{R} $, an equation of characters
\begin{equation}\label{last}\overset{n}{\underset{i=1}{\sum}}\underset{\psi\in\bar{\pi}}{\sum}\lambda_i\circ\psi=
\overset{n}{\underset{j=1}{\sum}}\underset{\phi\in\bar{\pi}}{\sum}\mu_j\circ\phi,\end{equation}
where $\bar{\pi}$ denotes the equivalence class of $\pi$ with relative to the left action of $U$ on $\mathcal{F}$. Being Frobenius, $R$ has a generating character $\chi$ (Corollary \ref{honcor}). For $\chi$, equation \ref{last} holds, and for $i=1\text{ and } \psi=\chi$ (on the left) the linear independence of characters implies the existence of $\phi\in \bar{\chi}$ and $j=\sigma(1)$ such that $\lambda_1\circ\chi=\mu_{\sigma(1)}\circ \phi$. In other words, $\phi=\chi^{u_1}$  for some $u_1\in U$ with $\lambda_1\circ\chi=(\mu_{\sigma(1)}u_1)\circ\chi $.   Appealing to the useful injectivity property of generating characters, Corollary \ref{genkerempty}, we get $\lambda_1=\mu_{\sigma(1)}u_1$. Thus, $\underset{\psi\in\bar{\chi}}{\sum}\lambda_1\circ\psi=\underset{\phi\in\bar{\chi}}{\sum}\mu_{\sigma(1)}\circ\phi$, allowing to reduce the outer sum by one. Proceeding that way we obtain units $u_1,\ldots,u_n\in U$ and build a permutation $\sigma\in S_n$ with $\lambda_i=\mu_{\sigma(i)}u_i$.
\end{proof}

\section{General Module Alphabets}
\subsection{Frobenius Bimodules}
No doubt, the next step is trying to generalize the previous ``model'' theorem for more module alphabets. The following theorem, which appears in \cite{r7}, 2009, provides such a generalization.
\begin{theorem}\cite[$\mathrm{p.37}$]{r7}. Let $A$ be a Frobenius bimodule over a finite ring $R$, then $A$ has the extension property with respect to $\mathrm{swc}_G$, where $G$ is any subgroup $G$ of $\mathrm{Aut}_R(A)$. \end{theorem}

\begin{proof}Let $f:C\rightarrow A^n$ be a monomorphism preserving $\mathrm{swc}_G$, where $C\subset A^n$ is an $R$-linear code (a left R-module). Let $M$ denote the left $R$-module underlying the code $C$, and $\lambda:M\rightarrow A^n$ be the inclusion map of $C$ into $A^n$. As in the  model theorem, set $\mu=\lambda\circ f:M\rightarrow A^n$, then the hypothesis says that $\mathrm{swc}_G(x\lambda,a)=\mathrm{swc}_G(x\mu,a)$ for all $a\in A, x\in M$. Express these maps as $\lambda=(\lambda_1,\ldots,\lambda_n)$ and $\mu=(\mu_1,\ldots,\mu_n)$, where $\lambda_i,\mu_j\in \mathrm{Hom}_R(M,A)$.

Thus, for all $a\in A, x\in M$,
$$ \mathrm{swc}_G(x\lambda,a)=\frac{1}{|\widehat{A}|}\overset{n}{\underset{i=1}{\sum}}\underset{b\sim_G a}{\sum}\underset{\pi\in\widehat{A}}{\sum}\pi(x\lambda_i-b)=\frac{1}{|\widehat{A}|}\overset{n}{\underset{i=1}{\sum}}\underset{b\sim_G a}{\sum}\underset{\pi\in\widehat{A}}{\sum}\pi(x\lambda_i) \pi^{-1}(b),$$
then the preservation of $\mathrm{swc}_G$ becomes
\begin{equation}\label{found 37 15}\underset{\pi\in\widehat{A}}{\sum}\left(\overset{n}{\underset{i=1}{\sum}}\pi(x\lambda_i)\right)(P\pi^{-1})(a)=
\underset{\pi\in\widehat{A}}{\sum}\left(\overset{n}{\underset{j=1}{\sum}}\pi(x\mu_j)\right)(P\pi^{-1})(a).\end{equation}
\begin{par}
Fixing $x\in M$,  equation (\ref{found 37 15}) becomes an equation of linear combinations of averaged characters. The linear independence of averaged characters allows equating corresponding coefficients, after taking into regard that $\psi \sim\pi$ if and only if $P\psi=P\pi$. Thus
\begin{equation}
\label{found 37 16}\overset{n}{\underset{i=1}{\sum}}\underset{\theta\sim\pi}{\sum}\theta(x\lambda_i)=\overset{n}{\underset{j=1}{\sum}}
\underset{\phi\sim\pi}{\sum}\phi(x\mu_j),\quad x\in M.\end{equation} Confronted with these equations of characters on $M$, one for each $P\pi$, where $\pi\in\widehat{A}$, we then use that $A$ is  a Frobenius bimodule. This implies that $\widehat{A}$ has a generating character $\rho$. For $\pi=\rho$, in equation (\ref{found 37 16}), consider the term with $i=1$ and $\theta=\rho$ on the left side. Linear independence of characters on $M$ implies the existence of $\phi\sim\rho$ and $\sigma(1)$ such that $\rho(x\lambda_1)=\phi(x\mu_{\sigma(1)})$ for all $x\in M$.  But $\phi\sim\rho$ means that there exists $\tau_1\in G$ such that $\phi=\tau_1\rho$. Then $\rho(x\lambda_1)=\rho(x\mu_{\sigma(1)}\tau_1)$ for all $x\in M$. By Corollary \ref{genkerempty}, $\lambda_1=\mu_{\sigma(1)}\tau_1$. Reindexing, one gets $$\underset{\theta\sim\rho}{\sum}\theta(x\lambda_1)=\underset{\phi\sim\rho}{\sum}\phi(x\mu_{\sigma(1)}), \quad x\in M.$$Then, the outer sum is reduced by one, proceeding that way, we end with a $G$-monomial transformation $T$ of $A^n$ such that $\lambda=\mu T$.
\end{par}
\end{proof}

\subsection{Cyclic Socle is Sufficient}

\begin{par}In 2014, N. ElGarem, N. Megahed, and J. A. Wood \cite{r4} proved the following lessened sufficient condition (which turns out to be necessary for a certain class of modules).\end{par}

\begin{theorem}\label{Noha}Let $A$ be a finite left $R$-module. If $A$ can be embedded into $\widehat{R}$, then $A$ has the extension property with respect to $\mathrm{swc}_G$ for any subgroup $G$ of $\mathrm{Aut}(A)$. In particular, the theorem holds for Frobenius bimodules.\end{theorem}

\begin{proof}\begin{par}Let $G\subset\mathrm{Aut}(A)$. Suppose $C_1,C_2\subset A^n$ are two $R$-linear codes, and let $f:C_1\rightarrow C_2$ be an isomorphism preserving $\mathrm{swc}_G$. Let $M$ denote the left $R$-module underlying $C_1$ and $C_2$, and let $\lambda:M\rightarrow A^n$ be the inclusion map of $C_1$ into $A^n$. As usual, set $\mu=\lambda\circ f:M\rightarrow A^n$, and express these maps as $\lambda=(\lambda_1,\ldots,\lambda_n)$ and $\mu=(\mu_1,\ldots,\mu_n)$, where $\lambda_i,\mu_j\in \mathrm{Hom}_R(M,A)$.\end{par}

\begin{par}As $f$ preserves $\mathrm{swc}_G$, we have $\mathrm{swc}_G(x\lambda,a)=\mathrm{swc}_G(x\mu,a)$ for all $a\in A, x\in M$. Fixing $x\in M$, there is a permutation $\sigma_x\in S_n$ and elements $\phi_{j,x}\in G$ such that $x\lambda_j=x\mu_{\sigma_x(j)}\phi_{j,x}$ for each $j\in\{1,\ldots,n\}$. Let $\zeta\in G$, then \begin{equation}\label{4.1.noha}x\lambda_j\zeta=x\mu_{\sigma_x(j)}\phi_{j,x}\zeta.\end{equation}\end{par}
By Proposition \ref{prop1.2.11}, $A$ admits a generating character $\rho:A\rightarrow\mathbb{C}^\times$. Applying $\rho$ to equation (\ref{4.1.noha}), we get $$(x\lambda_j\zeta)\rho=(x\mu_{\sigma_x(j)}\phi_{j,x}\zeta)\rho.$$

Taking the summation over all $j\in\{1,\ldots,n\}$ and all $\zeta\in G$, we obtain
\begin{align*}\overset{n}{\underset{j=1}{\sum}}\underset{\zeta\in G}\sum(x\lambda_j\zeta)\rho &=\overset{n}{\underset{j=1}{\sum}}\underset{\zeta\in G}\sum(x\mu_{\sigma_x(j)}\phi_{j,x}\zeta)\rho\\
&=\overset{n}{\underset{k=1}{\sum}}\underset{\tau\in G}\sum (x\mu_k\tau)\rho.\end{align*}
As the previous equation holds for all $x\in M$, it yields the following equation of characters,
\begin{equation}\label{label}\overset{n}{\underset{j=1}{\sum}}\underset{\zeta\in G}\sum(\lambda_j\zeta)\rho=
\overset{n}{\underset{k=1}{\sum}}\underset{\tau\in G}\sum (\mu_k\tau)\rho.\end{equation}
On the left side, fixing $j=1$ and $\zeta=I_A$, the identity on $A$, the independence of characters gives that there is $k_1$ and $\tau_1\in G$ such that $\lambda_1\circ\rho=\mu_{k_1}\tau_1\circ\rho$. Then, $\mathrm{Im}(\lambda_1\-\mu_{k_1}\tau_1)\subset \mathrm{Ker}\rho$. By definition \ref{exactly}, of a generating character of a left $R$-module, it follows that $\mathrm{Im}(\lambda_1\-\mu_{k_1}\tau_1)=0$ and $\lambda_1=\mu_{k_1}\tau_1$. Setting $\gamma=\tau_1\zeta$, and re-indexing,
$$\underset{\zeta\in G}\sum(\lambda_1\zeta)\rho=\underset{\zeta\in G}\sum(\mu_{k_1}\tau_1\zeta)\rho=\underset{\gamma\in G}\sum(\mu_{k_1}\gamma)\rho.$$ This allows us to reduce, by one, the outer sum in equation \ref{label}. Continuing this process, we obtain automorphisms $\tau_1,\ldots,\tau_n\in G$, and a permutation $\sigma$, such that $\lambda_i=\mu_{\sigma(i)}\tau_i$.
\end{proof}

\subsection{Annihilator Weight: Partial Converse}

Here we define the \emph{annihilator weight} that will relate some of the extension problem with respect to symmetrized weight compositions to the already settled case of Hamming weight.

\begin{par}Let $A$ be a left $R$-module, and let $a\in A$ be any element. The \emph{annihilator} of $a$ is the set $\mathrm{ann}_a=\{r\in R: ra=0\}$. $\mathrm{ann}_a$ is easily seen to be a left ideal. We may define an equivalence relation $\approx$ on $A$ as follows: $a\approx b$ if $\mathrm{ann}_a=\mathrm{ann}_b$. Denote the orbit space of this relation by $A_\approx$.\end{par}
\begin{definition}(Annihilator Weight).
For any $n$,  we  define the annihilator weight $aw$ on $A^n$ to be the function $aw: A^n\times A\rightarrow\mathbb{Q}$ defined by $$aw(x,a)=|\{i: x_i\approx a\}|,$$where $x=(x_1,\ldots,x_n)\in A^n$ and $a\in A$. Thus, $aw$ counts the number of components in each orbit. \end{definition}

\begin{definition}(Extension Property). $A$ is said to have the EP with respect to $aw$ if, for every $n$, and any two linear codes $ C_1, C_2\subset A^n $,  any
$R$-linear isomorphism $ f:C_1 \rightarrow C_2$   preserving $aw$  extends to a
$\mathrm{Aut}_R(A)$-monomial transformation of $ A^n$. Where, by saying $f$ preserves $aw$, it is meant that $aw(xf,a)=aw(x,a)$ for all $x\in C_1$ and  $a\in A$. \end{definition}

\textbf{Remark:} It is easily seen that if $A$ has the EP  for Hamming weight then $A$ has  the EP for $\mathrm{swc}$ and the EP for $aw$ as well.

\begin{lemma}\label{lemma18888} Let $_RA$ be a pseudo-injective module (Definition \ref{1234}). Then for any two elements $a,b\in A$, $a\approx b$ if and only if $a\sim b$ (recall that $\sim$ is just another notation for $\sim_{\mathrm{Aut}_R(A)}$).\end{lemma}

\begin{proof}If $a\sim b$, this means $a=b\tau$ for some $\tau \in \mathrm{Aut}_R(A)$, and consequently $\mathrm{ann}_a=\mathrm{ann}_b$.

Conversely, if $a\approx b$, then we have (as left $R$-modules) $$Ra\cong{} _RR/\mathrm{ann}_a={} _RR/\mathrm{ann}_b\cong Rb,$$with $ra\mapsto r+\mathrm{ann}_a\mapsto rb$. By Lemma \ref{psss}, since $A$ is pseudo-injective, the isomorphism $Ra\rightarrow Rb\subseteq A$ extends to an automorphism of $A$ taking $a$ to $b$.\end{proof}

\begin{corollary}\label{cor1}If $_RA$ is a pseudo-injective module, then $A$ has the EP with respect to $\mathrm{swc}$ if and only if $A$ has the EP with respect to $aw$.\end{corollary}

\begin{theorem}\label{Midway}Let $R$ be a left principal ideal ring,  $_RA$  a pseudo-injective module, and let $C$ be a submodule of $A^n$ for some $n$. Then a monomorphism  $f:C\rightarrow A^n$ preserves Hamming weight if and only if it preserves $\mathrm{swc}$.\end{theorem}

\begin{proof} The ``if'' part is direct. For the converse,
 Suppose that \begin{equation}\label{eq18888}(c_1,c_2,\ldots,c_n)f=(b_1,b_2,\ldots,b_n).\end{equation} Choose, from $c_1,c_2,\ldots,c_n;b_1,b_2,\ldots,b_n$, a component with  maximal annihilator  $I=\langle e_I\rangle$. Act on equation (\ref{eq18888}) by $e_I$, then the only zero places  are those of the components in equation (\ref{eq18888}) with annihilator $I$, and the preservation of Hamming weight gives the preservation of $I$-annihilated components. Omit these components from the list $c_1,c_2,\ldots,c_n;b_1,b_2,\ldots,b_n$ and choose one with a new maximal, and repeat. This finally gives that $f$ preserves $aw$ and hence, by Lemma \ref{lemma18888}, $f$ preserves $\mathrm{swc}$.
\end{proof}

\begin{definition}(The Socle Condition). Let $A$ be a finite left $R$-module. We will say that $A$  satisfies the \emph{socle condition} if, for any $a\in \mathrm{soc}(A)$, any monomorphism $f:Ra\rightarrow A$ can be extended to $A$.\end{definition}

It is seen that the socle condition is weaker than being pseudo-injective. The class of modules satisfying the socle condition includes semisimple modules and pseudo-injective modules. By Lemma \ref{psss}, we know that the  indicated extension process implies the existence of an automorphism extending the given monomorphism.
The following theorem is the farthest we could reach using the former connection between the two extension problems.

\begin{proposition}\label{review2}Let $_RA$ be a finite left $R$-module satisfying the socle condition, and let $C\subset \mathrm{soc}(A)^n\subset A^n$ be a linear code of length $n$ over $A$.
Then a  monomorphism $f:C\rightarrow A^n$ preserves $\mathrm{swc}_{\mathrm{Aut}_R(\mathrm{soc}(A))}$  if and only if it preserves $\mathrm{swc}_{\mathrm{Aut}_R(A)}$.\end{proposition}

\begin{proof}
First notice that the components in the codewords of $C$ and $Cf$ are in $\mathrm{soc}(A)$. Let $a,b\in \mathrm{soc}(A)$. Suppose there exists an automorphism $\tau\in\mathrm{Aut}_R(A)$ such that $a\tau=b$, thus $a\sim_{\mathrm{Aut}_R(A)}b$. Since the restriction of $\tau$ to  $\mathrm{soc}(A)$ is an automorphism of $\mathrm{soc}(A)$, we have $a\sim_{\mathrm{Aut}_R(\mathrm{soc}(A))}b$.\\

Conversely, assume that $a\sim_{\mathrm{Aut}_R(\mathrm{soc}(A))}b$, and let  $\tau'\in\mathrm{Aut}_R(\mathrm{soc}(A))$ be such that $a\tau'=b$. Then, in particular, the restriction $Ra\overset{\tau'}{\rightarrow}Rb$ is an isomorphism. Since $A$ satisfies the socle condition, the last restriction can be extended to the whole of $A$, and by Lemma \ref{psss}, this extension is an automorphism of $A$. Thus, $a\sim_{\mathrm{Aut}_R(A)}b$.\\

This shows that in the case $A$ has the socle condition, the two relations $\sim_{\mathrm{Aut}_R(A)}$ and $\sim_{\mathrm{Aut}_R(\mathrm{soc}(A))}$ induce the same partition in $\mathrm{soc}(A)$. Thus, for $C\subset\mathrm{soc}(A)^n$, $f:C\rightarrow \mathrm{soc}(A)^n\subset A^n$ preserves $\mathrm{swc}_{\mathrm{Aut}_R(\mathrm{soc}(A))}$ if and only if $f$ preserves $\mathrm{swc}_{\mathrm{Aut}_R(A)}$.

 \end{proof}

Note that in the next theorem, there are no conditions on the finite ring $R$.

\begin{theorem}\label{main}Let $_RA$ be a finite left $R$-module which satisfies the socle condition. Then, $A$ has the extension property with respect to $\mathrm{swc}$ if and only if $\mathrm{soc}(A)$ is cyclic.\end{theorem}
\begin{proof} \begin{par}The ``if'' part is answered by Theorem \ref{Noha}. Now, if $\mathrm{soc}(A)$ is not cyclic, then by Lemma \ref{simui},  there is an index $i$ such that the socle contains a copy of $s_iT_i$, with $s_i>\mu_i$. So, we may assume  $s_iT_i\subset \mathrm{soc}(A)\subset A$. Recall that $T_i$  is the pullback to $R$ of the matrix module $_{\mathbb{M}_{\mu_i}(\mathbb{F}_{q_i})}\mathbb{M}_{\mu_i\times 1}(\mathbb{F}_{q_i})$, so that $s_iT_i$ is the pullback to $R$ of the $\mathbb{M}_{\mu_i}(\mathbb{F}_{q_i})$-module $B=\mathbb{M}_{\mu_i\times s_i}(\mathbb{F}_{q_i})$. Theorem \ref{wood} implies the existence of linear codes $C_+,C_-\subset B^N$, and an isomorphism $f:C_+\rightarrow C_-$ that preserves Hamming weight, yet $f$ can not extend to any monomial transformation.\\  \end{par}
\begin{par}The ring $\mathbb{M}_{\mu_i}(\mathbb{F}_{q_i})$ is a left principal ideal ring (in fact, more is true, Theorem ix.3.7, \cite{hungr}). Also, $B$ is injective by Theorem \ref{lastlast}, since the ring $\mathbb{M}_{\mu_i}(\mathbb{F}_{q_i})$ is simple. Then, Theorem \ref{Midway} implies that $f$ preserves $\mathrm{swc}$ (with respect to the group $\mathrm{Aut}_{\mathbb{M}_{\mu_i}(\mathbb{F}_{q_i})}(B)$).
The $R$-module structure of $B$  can be used to view any $\mathbb{M}_{\mu_i}(\mathbb{F}_{q_i})$-automorphism of $_{\mathbb{M}_{\mu_i}(\mathbb{F}_{q_i})}B$ as an $R$-automorphism of $_RB$ and vice versa. Besides, we can consider   the whole situation for $C_\pm$ as $R$-modules. \end{par}

\begin{par}Since $B$ pulls back to $s_iT_i$, we have
$C_\pm\subset(s_iT_i)^N\subset{\mathrm{soc}(A)}^N\subset A^N$, as $R$-modules. Thus $C_\pm$ are linear codes over $A$ that are isomorphic through an $R$-isomorphism preserving $\mathrm{swc}$ with respect to  $\mathrm{Aut}_R(s_iT_i)$. \end{par}

\begin{par}Any automorphism of $s_iT_i$ extends to an automorphism of $\mathrm{soc}(A)$. Conversely,  the homogeneous component $s_iT_i$ is  to be preserved under any automorphism $\tau\in \mathrm{Aut}_R(\mathrm{soc}(A))$. Hence, $f$ preserves $\mathrm{swc}$ with respect to $\mathrm{Aut}_R(\mathrm{soc}(A))$.\end{par}

\begin{par}By Proposition \ref{review2}, since $A$ has the socle condition, $f$  also  preserves $\mathrm{swc}$ with respect to $\mathrm{Aut}_R(A)$. However, this isomorphism does not extend to a monomial transformation of $A^N$, since, as appears in the proof of Theorem \ref{wood} (found in \cite{r6}), $C_+$ has an identically zero component, while $C_-$ does not. \\\end{par}\end{proof}

\textbf{Examples:}

\begin{par}(Additive Codes): Let $L$ be a finite field. Consider the alphabet $_KL$, where $K\subset L$ is a subfield. When $K=\mathbb{F}_p$ for a prime p, $K$-linear codes are called additive. In general, $K$-linear codes are called sublinear. Notice that $L$ is a $K$-vector space of dimension $n=[L:K]$. Let $A\subset L$ be any $K$-vector space ($K$-submodule). If $f:A\rightarrow L$ is a monomorphism, then, by extending the basis of $A$, $f$ can be extended to $L$. Thus, by the previous work, $_KL$ satisfies the extension property with respect to $\mathrm{swc}_{\mathrm{Aut}_K(L)}$ if and only if $\mathrm{soc}(_KL)$ is cyclic. In this setting, $R=K, k=1$ and $\mu_1=1$. This implies that  $\mathrm{soc}(_KL)$ is cyclic if and only if  $K=L$, for if this is not the case, then $L$ contains more than one (check the Gaussian number) $K$-subspace of dimension 1 and at least one of them has zero intersection with $K$ making the socle noncyclic in view of the above value of $\mu_1$.\end{par}

\subsection{Future Work}
As the previous result strengthen the conjecture, that a finite $R$-module satisfies the extension property with respect to symmetrized weight compositions if and only if it has a cyclic socle, it is quite prompting to seek, with more confidence, an affirmative proof of the conjecture. It should be remarked that the class of modules considered in Theorem \ref{main} consists of modules that already have good extension properties, yet, the condition of having a cyclic socle was shown necessary for these modules to satisfy the extension property with respect to symmetrized weight compositions.

\chapter{Appendix: A Simple Proof}

\begin{par}Here we include the proof due to K. Bogart et al. \cite{bogart} for the MacWilliams theorem in the classical case over finite fields. As we shall see, the idea ends just like a `distillation process', through a filter of $1$-dimensional vector spaces.
Recall that an $[n,k]$ linear code $C$ over $\mathbb{F}_q$ is a $k$-dimensional subspace of $\mathbb{F}^n_q$, and hence $C$ may be represented a $k\times n$ `generator' matrix (see Chapter 2). \end{par}

\begin{par}Let $L_1,\ldots,L_{\mu(k)}$ be the list of all $1$-dimensional subspaces of $\mathbb{F}^k_q$, where $\mu(k)$ is just the Gaussian number, thus $\mu(k)=\begin{bmatrix}k\\1\end{bmatrix}_q=(q^k-1)/(q-1)$. Clearly, if $u_i$ and $u_j$ are non-zero vectors in $L_i$ and $L_j$, respectively, and $u_i\cdot u_j=0$, then the same is true for any cjosen vectors of these spaces, and in this case we say that $L_i$ is orthogonal to $L_j$ and we write $L_i\bot L_j$.\end{par}
Now set $T=(t_{ij})$ to be the $\mu(k)\times \mu(k)$ matrix with $t_{ij}=\left\{\begin{smallmatrix}0\;,&\quad L_i\bot L_j\\1\;,&\quad \text{otherwise.}\end{smallmatrix}\right.$

\begin{lemma}\label{a1}$|\{L_j: L_j\bot L_i\}|=\mu(k-1)$.\end{lemma}
\begin{proof}Let $u_i\in L_i$ be a fixed element. The indicated set consists exactly of the $1$-dimensional subspaces of the kernel of the map $P:\mathbb{F}^k_q\rightarrow \mathbb{F}_q$ where $P(u)=u\cdot u_i$. As we know $k=\dim Ker P+\dim Im P$, and hence the result follows.\end{proof}

\begin{lemma}\label{a2}If $i\neq j$, $|\{L_k: L_k\bot L_i \text{\;and\;} L_k\bot L_j\}|=\mu(k-2)$.\end{lemma}
\begin{proof}Use the same argument as above for the map $\mathbb{F}^k_q\rightarrow \mathbb{F}^2_q$ given by $u\mapsto (u\cdot u_i, u\cdot u_j)$.\end{proof}

\begin{lemma}\label{a3} The sum, over rational numbers, of the rows of $T$ is the constant row vector $x=(x_1,\ldots,x_\mu(k))$ with $x_i=\mu(k)-\mu(k-1)$ for all $i$.\end{lemma}
\begin{proof}Simply notice that $x_i=\mu(k)-|\{L_j: L_j\bot L_i\}|$.\end{proof}

\begin{lemma}The sum of the rows of $T$ which are $0$ in the $j^{\text{th}}$ column is the row vector $y(j)=(y_1,\ldots,y_{\mu(k)})$, with $y_j=0$ and $y_i=\mu(k-1)-\mu(k-2)$ for $i\neq j$.\end{lemma}
\begin{proof}Use the fact that for $i\neq j$, $$y_i=|\{L_k:L_k\bot L_j\}|-|\{L_k:L_k\bot L_j\text{\;and\;} L_k\bot L_i\}|$$.\end{proof}

\begin{lemma}\label{LA} The matrix $T$ is invertible over $\mathbb{Q}$, the rational numbers.\end{lemma}
\begin{proof}If $\{e_1,\ldots,e_{\mu(k)}\}$ is the standard basis for $\mathbb{Q}^{\mu(k)}$, then we can express them as
$$e_j=\frac{x}{\mu(k)-\mu(k-1)}-\frac{y(j)}{\mu(k-1)-\mu(k-2)}.$$
Therefore, the row space of $T$ is all of $\mathbb{Q}^{\mu(k)}$ and the rank of $T$ is $\mu(k)$, whence $T$ is invertible.\end{proof}

Now, let $X$ be a generator matrix of $C$, and let $f:\mathbb{F}_q^k\rightarrow C$ be the isomorphism defined by $f(u)=uX, u\in \mathbb{F}^k_q$. Then $f(e_i)$ is the $i^{\text{th}}$ row of $X$. Certainly, $f$ maps the $1$-dimensional subspaces of $\mathbb{F}^k_q$ to those of $C$, and $f(L_i)=\langle f(u_i) \rangle$, where $u_i$ is a fixed element of $L_i$.

\begin{definition} In what follows, $c$ will denote a column of $X$. We define $r=(r_1,\ldots,r_{\mu(k)})^t$ to be the column vector with $r_i=|\{c\neq 0: c^t\in L_i\}|$.

The vector $r$ thus indicates how the (linearly independent) columns of $X$ are distributed among the `filter' of the $L_i$'s.\end{definition}

\begin{lemma}$(Tr)_i=\mathrm{wt}(f(u_i)), u_i\in L_i$\end{lemma}

\begin{proof}\begin{align*}(Tr)_i&=\overset{\mu(k)}{\underset{j=1}{\sum}}t_{ij}r_j\\
                                     &=\underset{j:\neg(L_j\bot L_i) }{\sum}r_j\\
                                     &=\underset{j:\neg(L_j\bot L_i) }{\sum}|\{c\neq 0: c^t\in L_j\}|\\
                                     &=|\{c\neq 0: c^t\in L_j \text{\;and\;} \neg(c^t\bot u_i), u_i\in L_i\}|\\
                                     &=\mathrm{wt}(u_i X)\\
                                     &=\mathrm{wt}(f(u_i)).\end{align*} \end{proof}
\begin{theorem}If $C\overset{\varphi}{\rightarrow}D$ is an isometry, with $X$ being a generator matrix for $C$, then there is a generator matrix $Y$ of $D$ such that $X\Lambda= Y$, for some monomial matrix $\Lambda$. \end{theorem}
\begin{proof}Let $g:\mathbb{F}^k_q\rightarrow D$ be defined by $g=\varphi\circ f$, where $f$ is as before. Thus $$g(u)=\varphi\circ f (u)=\varphi(uX),\;\;u\in\mathbb{F}^k_q.$$
It is then clear that if $Y$ is the matrix of $g$ relative to the standard basis then it is a generator matrix for $D$. By the last lemma $(Tr')_i=\mathrm{wt}(g(u_i)),\; u_i\in L_i$, where $r'$ is the column vector for $Y$ defined just as $r$  corresponded to $X$. Moreover, we discover the following desired properties for $Y$.

Notice that if $u_i\in L_i$, then $$\mathrm{wt}(g(u_i))=\mathrm{wt}(\varphi\circ f(u_i))=\mathrm{wt}(f(u_i)).$$
Thus $(Tr)_i=(Tr')_i$, and hence, by lemma \ref{LA}, $r=r'$. Restoring our information, we recall theat $r,r'$ list the number of nonzero columns of $X,Y$, respectively, which lie in $L_1,\ldots,L_{\mu(k)}$, and therefore, the result $r=r'$ means that the columns of $Y$ are simply scalar multiples of those of $X$, up to some order. Thus we may take the monomial matrix $\Lambda=QP$, where the $n\times n$ matrices $Q$ and $P$ specify scalar multiplication and permutation respectively, and obtain $X\Lambda=Y$.\end{proof}

\begin{center}$\sim$\end{center}


\begin{thebibliography}{}


\bibitem{JAA} Ali Assem, \emph{On Modules with Cyclic Socle}, DOI: 10.1142/S0219498 816501620, Accepted: 30 September 2015, in the Journal of Algebra and Its Applications.
\bibitem{cat} F. W. Anderson, and K. R. Fuller, \emph{Rings and Categories of Modules},
Second Edition, Graduate Texts in Mathematics, Springer-Verlag, New
York, 1974.

\bibitem{bass}H. Bass, \emph{K-theory and Stable Algebra}, Publications Mathmatiques de
l'Institut des Hautes \'{e}tudes Scientifiques, Vol. 22, no. 1, December 1964,
p. 5-60.


\bibitem{bogart}K. Bogart, D. Goldberg, and J. Gordon, \emph{An elementary proof of the
MacWilliams theorem on equivalence of codes}, Inform. and Control 37 (1978),
no. 1, 19-22. MR MR0479646 (57 $\sharp$19067)

\bibitem{Claasen} H. L. Claasen and R. W. Goldbach, \emph{A field-like property of finite rings}, Indag. Math. 3 (1992), 11-26.

\bibitem{const1}
I. Constantinescu, W. Heise, and T. Honold, \emph{Monomial extensions of isometries
between codes over $Z_m$}, Proceedings of the Fifth International Workshop
on Algebraic and Combinatorial Coding Theory (ACCT '96) (Sozopol, Bulgaria),
Unicorn, Shumen, 1996, pp. 98-104.



\bibitem{curtis}C. W. Curtis and I. Reiner, \emph{Representation Theory of  Finite Groups and Associative Algebras}, Interscience Publishers, New York, 1962.

\bibitem{r1}
H. Q. Dinh and S. R. L\'{o}pez-Permouth, \emph{On the equivalence of codes over finite
rings}, Appl. Algebra Engrg. Comm. Comput. 15 (2004), no. 1, 37-50. MR
MR2142429 (2006d:94097)


\bibitem{r2}H. Q. Dinh and S. R. L\'{o}pez-Permouth, \emph{On the equivalence of codes over rings and modules}, Finite Fields
Appl. 10 (2004), no. 4, 615-625. MR MR2094161 (2005g:94098).


\bibitem{r3}N. ElGarem, \emph{Codes over Finite Modules}, Master Thesis - Cairo University, 2014.
\bibitem{r4}N. ElGarem, N. Megahed, and J.A. Wood, The extension Theorem with respect to Symmetrized Weight Compositions, 4th international castle meeting on coding theory and applications, 2014.

\bibitem{gref}M. Greferath and S. E. Schmidt, \emph{Finite-ring combinatorics and MacWilliams' equivalence
theorem}. J. Combin. Theory Ser. A. 92, 17-28 (2000).

\bibitem{grefff}M. Greferath and T. Honold, \emph{Monomial extensions of isometries of linear codes:} \emph{Invariant weight functions on} $Z_m$. In Proceedings of the Tenth International Workshop on Algebraic and Combinatorial Coding Theory (ACCT-10), Zvenigorod, Russia, September 2006, 106-111.


\bibitem{2004}
M. Greferath, A. Nechaev, and R. Wisbauer, \emph{Finite quasi-Frobenius modules
and linear codes}, J. Algebra Appl. 3 (2004), no. 3, 247-272. MR MR2096449
(2005g:94099)

\bibitem{radicalappr}M. Gray, \emph{A Radical Approach to Algebra}, Reading, Mass.: Addison-Wesley Publishing Company, Inc., 1970.
\bibitem{MHall} M. Hall, \emph{A type of algebraic closure}, Ann. of Math. 40 (1939), 360-369.
\bibitem{Hazewinkel} M. Hazewinkel, N. Gubareni, and V. Kirichenko, \emph{Algebras, Rings and Modules}, volume 1, Kluwer Academic Publishers, New York, Boston, Dordrecht, London, Moscow. ISBN: 1-4020-2690-0.


\bibitem{Honold}T. Honold, \emph{Characterization of finite Frobenius rings}, Arch. Math. (Basel)
76 (2001), no. 6, 406-415. MR MR1831096 (2002b:16033)

\bibitem{Honnech}T. Honold and A. A. Nechaev, \emph{Weighted modules and representations of codes}. Problems Inform.
Transmission 35 (3), 205-223 (1999).

\bibitem{hungr} T. W. Hungerford, \emph{Algebra}, Graduate Texts
in Mathematics, vol 73 1974, Springer-Verlag, New York, ISBN: 978-1-4612-6103-2.

\bibitem{Klemm} M. Klemm, \emph{Eine Invarianzgruppe f�\"{u}r die vollst\"{a}ndige Gewichtsfunktion selbstdualer Codes}, Arch.
Math. (Basel) 53 (1989), 332-336.

\bibitem{Klemm23} M. Klemm, \emph{Selbstduale codes \"{u}ber dem ring der ganzen zahlen modulo 4}, Arch.
Math. (Basel) 53 (1989), 201-207.


\bibitem{Lam}T. Y. Lam, \emph{Lectures on modules and rings}, Graduate Texts in Mathematics,
vol. 189, Springer-Verlag, New York, 1999. MR MR1653294 (99i:16001)

\bibitem{Lamnoncom}T. Y. Lam, \emph{A First course in noncommutative rings}, second ed., Graduate Texts
in Mathematics, vol. 131, Springer-Verlag, New York, 2001. MR MR1838439
(2002c:16001)

\bibitem{lmprcht1}E. Lamprecht, \emph{Allgemeine Theorie der Gau{\ss}schen Summen in endlichen kommutativen Ringen},
Math. Nachr. \textbf{9}, 149 - 196 (1953).

\bibitem{lmprcht2}E. Lamprecht, \emph{\"{U}ber I-regul\"{a}re Ringe, regul\"{a}re Ideale und Erkl\"{a}rungsmoduln}, I. Math. Nachr. \textbf{10},
353 - 382 (1953).

\bibitem{lmprcht3}E. Lamprecht, \emph{Struktur und Relationen allgemeiner Gau{\ss}scher Summen in endlichen Ringen. I}, J.
Reine Angew. Math. \textbf{197}, 1-26 (1957). MR: 19,245d.

\bibitem{lmprcht4}E. Lamprecht, \emph{Struktur und Relationen allgemeiner Gau{\ss}scher Summen in endlichen Ringen. II}, J.
Reine Angew. Math. \textbf{197}, 27-48 (1957). MR: 19,245d.
\bibitem{Lintcomb}J. H. van Lint, and R. M. Wilson, \emph{A Course in Combinatorics}, Cambridge
University Press, Cambridge, 1992.
\bibitem{Lint}J. H. Van Lint,  \emph{Introduction to Coding Theory}, 3rd edn. Springer, Berlin (1999).

\bibitem{Mac}F. J. MacWilliams, \emph{Combinatorial properties of elementary abelian groups}, Ph.D. thesis,
Radcliffe College, Cambridge, Mass., 1962.
\bibitem{Mac book}F. J. MacWilliams, N. J. A. Sloane, \emph{The Theory of Error-Correcting Codes}. North-Holland, Amsterdam
(1977).
\bibitem{eniusean} T. Nakayama, \emph{On Frobeniusean Algebras. II}, Ann. of Math. \textbf{42} (1941), 1-21.

\bibitem{supp} T. Nakayama, \emph{Supplementary remarks on Frobeniusean algebras}. I. Proc. Japan Acad. 25 (7), 45-50 (1949).







\bibitem{serre}J. P. Serre, \emph{Linear representations of finite groups}, Graduate Texts in Mathematics, vol. 42, Springer-Verlag, New York, Heidelberg, Berlin, 1977.

\bibitem{stanley}R. P. Stanley, \emph{Enumerative Combinatorics}, The Wadsworth and
Brooks/Cole Mathematics Series, Vol. 1, Wadsworth and Brooks/Cole
Mathematics Advanced Books and Software, Monterey, CA, 1986.
\bibitem{Lintcomb}J. H. van Lint, and R. M. Wilson, \emph{A Course in Combinatorics}, Cambridge
University Press, Cambridge, 1992.
\bibitem{Lint}J. H. van Lint,  \emph{Introduction to Coding Theory}, 3rd edn. Springer, Berlin (1999).

\bibitem{r1996}H. N. Ward and J. A. Wood, \emph{Characters and the equivalence of codes}, J. Combin. Theory Ser. A 73 (1996), no. 2, 348-352. MR MR1370137 (96i:94028)

\bibitem{r5}
J. A. Wood, \emph{Extension theorems for linear codes over finite rings}, Applied
algebra, algebraic algorithms and error-correcting codes (Toulouse, 1997)
(T. Mora and H. Mattson, eds.), Lecture Notes in Comput. Sci., vol. 1255,
Springer, Berlin, 1997, pp. 329-340. MR MR1634126 (99h:94062)

\bibitem{duality}
J. A. Wood, \emph{Duality for modules over finite rings and applications to coding theory}, Amer. J. Math. \textbf{121} (1999), no.3, 555-575. MR MR1738408 (2001d:94033)

\bibitem{r6}
J. A. Wood, \emph{Code equivalence characterizes finite Frobenius rings}, Proc. Amer.
Math. Soc. 136 (2008), 699-706.

\bibitem{r7}J. A. Wood, \emph{Foundations of Linear Codes Defined over Finite
Modules: The Extension Theorem and MacWilliams Identities,
Codes over Rings}, Proceedings of the CIMPA Summer School,
Ankara, Turkey, 18-29 August 2008, (Patrick Sol\'{e}) Series on
Coding Theory and Cryptology, Vol. 6, World Scientific,
Singapore, 2009, p. 124-190.

\bibitem{woodappl}J. A. Wood, \emph{Applications of Finite Frobenius Rings to the Foundations
of Algebraic Coding Theory}, Proceedings of the 44th Symposium on
Ring Theory and Representation Theory, Okayama University, Nagoya,
Japan, 25-27 September 2011, O. Iyama, ed., Nagoya, Japan, 2012, p.
235-245.

\bibitem{xue} W. Xue, \emph{A note on finite local rings}. Indag. Math. (N.S.) 9 (4), 627-628 (1998).

 \end{thebibliography}
\end{document}